\documentclass[reqno,11pt]{amsart}
\usepackage{graphicx,indentfirst}
\usepackage{amsmath,amssymb,mathrsfs, bm}
\usepackage{amsthm,amscd}
\usepackage{verbatim}
\usepackage{appendix}
\usepackage{enumitem,titletoc}
\usepackage{appendix}
\usepackage[utf8]{inputenc}
\usepackage{imakeidx}
\makeindex[columns=2, title=Alphabetical Index]
\usepackage{fancyhdr}
\usepackage{amsfonts,color}
\usepackage[all]{xy}
\usepackage{tikz-cd}
\usepackage{syntonly}
\usepackage[left=2.48cm, right=2.48cm, bottom=3cm, top=3cm]{geometry}
\usepackage{tikz}
\usepackage[backref = false]{hyperref}
\usepackage{lipsum} 
\usepackage{etoolbox}
\patchcmd{\subsection}{\bfseries}{\bfseries}{}{}
\patchcmd{\subsection}{-.5em}{.5em}{}{}
\newtheorem{theorem}{Theorem}[section]
\newtheorem{lemma}{Lemma}[section]
\newtheorem{remark}{Remark}[section]

\newcommand{\abs}[1]{|#1|^2}

\newcommand{\osc}{\mathrm{osc}}

\def\XXint#1#2#3{{\setbox0=\hbox{$#1{#2#3}{\int}$ }
\vcenter{\hbox{$#2#3$ }}\kern-.6\wd0}}

\newtheorem{prop}{Proposition}[section]
\newtheorem{defn}{Definition}[section]
\newtheorem{corr}{Corollary}[section]

\allowdisplaybreaks

\newcommand{\ddbar}{\sqrt{-1}\partial\bar\partial}
\newcommand{\ric}{\mathrm{Ric}}
\newcommand{\ddb}{\sqrt{-1}\partial\bar\partial}
\newcommand{\innpro}[1]{\langle#1\rangle}
\newcommand{\bk}[1]{\Big(#1\Big)}

\newcommand{\vol}{\mathrm{Vol}}
\newcommand{\xk}[1]{\big(#1\big)}
\newcommand{\ba}[1]{\big|#1\big|}
\newcommand{\Ba}[1]{\Big|#1\Big|}
\newcommand{\eqsp}[1]{\begin{equation}\begin{split} #1\end{split}\end{equation}}

\newcommand{\detla}{\delta}

\newcommand{\sS}{\mathcal{S}}
\newcommand{\Ss}{\mathcal{S}}
\newcommand{\dsq}{(D')^2}
\newcommand{\ibone}{\frac{1}{\beta_1}}
\newcommand{\ibtwo}{\frac{1}{\beta_2}}
\newcommand{\ms}{\backslash \sS}
\newcommand{\bbeta}{{\boldsymbol\beta}}\newcommand{\bb}{{\boldsymbol\beta}}
\newcommand{\pp}{{\mathcal P}}
\newcommand{\qq}{{\mathcal Q}}
\newcommand{\ep}{{\epsilon}}\newcommand{\epp}{{\epsilon}}
\newcommand{\hoep}{( \frac{\partial}{\partial t} - \Delta_{g_\epsilon}  )}

\newcommand{\na}{\nabla}

\newcommand{\tr}{{\mathrm{tr}}}
\newcommand{\cC}{{\mathcal C}}
\newcommand{\C}{{\mathcal C}}
\newcommand{\newT}{{\mathcal T}}

\makeindex

\hypersetup{
    colorlinks=true,
    citecolor=black,
    filecolor=green,
    linkcolor=black,
    urlcolor=black
}

\renewcommand{\Re}{\mathrm{Re}}
\numberwithin{equation}{section}

\begin{document}
\address{Department of Mathematics, Columbia University, New York, NY 10027}

\email{bguo@math.columbia.edu}

\address{Department of Mathematics, Rutgers University, Piscataway, NJ 08854}

\email{jiansong@math.rutgers.edu}

\thanks{Research supported in
part by National Science Foundation grants DMS-17-11439 and DMS-17-10500.}

\title{Schauder estimates for equations with cone metrics, II}\author{Bin Guo \and Jian Song }\date{}
\maketitle

\begin{abstract}
This is the continuation of our  paper \cite{GS}, to study the linear theory for equations with conical singularities. We derive interior Schauder estimates for linear elliptic and parabolic equations with a background K\"ahler metric of conical singularities along a divisor of simple normal crossings. As an application, we prove the short-time existence of the conical K\"ahler-Ricci flow with conical singularities along a divisor with simple normal crossings.

\end{abstract}

\tableofcontents

\section{Introduction}

Regularity of solutions of Complex Monge-Amp\`ere equations is a central problem in complex geometry.  Complex Monge-Amp\`ere equations with singular and degenerate data can be applied to study compactness and moduli problems of canonical K\"ahler metrics in K\"ahler geometry. In \cite{Y}, Yau has already considered special cases of complex singular Monge-Amp\`ere equations as generalization of his solution to the Calabi conjecture. Conical singularities along complex hypersurfaces of a K\"ahler manifold are among the mildest singularities in K\"ahler geometry and it has been extensively studied, especially in the case of Riemann surfaces \cite{Tr, LT}. The study of such K\"ahler metrics with conical singularities has many geometric applications, for example, the Chern number inequality in various settings \cite{T96, SW}. Recently, Donaldson \cite{D} initiated the program of studying analytic and geometric properties of K\"ahler metrics with conical singularities along a smooth complex hypersurface on a K\"ahler manifold. This is an essential step to the solution of the Yau-Tian-Donaldson conjecture relating existence of K\"ahler-Einstein metrics and algebraic K-stability on Fano manifolds \cite{CDS1, CDS2, CDS3, T12}.  In \cite{D}, the Schauder estimate for linear Laplace equations with conical background metric is established using classical potential theory. This is crucial for the openness of the continuity method to find desirable (conical) K\"ahler-Einstein metric. Donaldson's Schauder estimate is   generalized to the parabolic case \cite{CW} with similar classical approach. There is also an alternative approach for the conical Schauder estimates using microlocal analysis \cite{JMR}. There are also various global and local estimates and regularity derived in the conical setting \cite{Br, E,CW2, Da, DGSW, DS, GP, E, JLZ, M, R, Yi, YZ}.

The Schauder estimates play an important role in the linear PDE theory. Apart from the classical potential theory, various proofs have been established by different analytic techniques. In fact, the blow-up or perturbation techniques developed in \cite{Si, W} (also see \cite{Sa1, Sa2, C1, C2}) are much more flexible and sharper than the classical method.  The authors combined the perturbation method in \cite{GS} and geometric gradient estimates to establish sharp Schauder estimates for Laplace equations and heat equations on $\mathbb{C}^n$ with a background flat K\"ahler metric of conical singularities along the smooth hyperplane $\{ z_1=0\}$ and derived explicit  and optimal dependence on conical parameters. 

In algebraic geometry, one often has to consider pairs $(X, D)$ with $X$ being an algebraic variety of complex dimension $n$ and the boundary divisor $D$ as a complex hypersurface of $X$. After possible log resolution, one can always assume the divisor $D$ is a union of smooth hypersurfaces with simple normal crossings. The suitable category of K\"ahler metrics associated to $(X, D)$ is the family of K\"ahler metrics on $X$ with conical singularities along $D$. In order to study canonical K\"ahler metrics on pairs and related moduli problems, we are obliged to study regularity and asymptotics for complex Monge-Amp\`ere equation with prescribed conical singularities of normal crossings. However, the linear theory is still missing and has been open for a while. The goal of this paper is to extend our result \cite{GS} and establish the sharp Schauder estimates for linear equations with background K\"ahler metric of conical singularities along divisors of simple normal crossings. We can apply and extend many techniques developed in \cite{GS}, however,  new estimates and techniques have to be developed because in case of conical singularities along a single smooth divisor, the difficult estimate in the conical direction can sometimes be bypassed and reduced to estimates in the regular directions, while such treatment does not work in the case of simple normal crossings. One is forced to treat regions near high codimensional singularities directly with new and more delicate estimate beyond the scope of \cite{GS}.

The standard local models for such conical K\"ahler metrics can be described as below.  

\smallskip

Let $\bbeta = (\beta_1,\ldots,\beta_p)\in (0,1)^p$ and $p\le n$ and $\omega_{\bbeta}$ (or $g_\bbeta$) be the standard cone metric on $\mathbb C^p\times \mathbb C^{n-p}$ with cone singularity along $\sS = \cup_{i=1}^p \sS_i$, where $\sS_i = \{z_i = 0\}$, that is, 
\begin{equation}\label{eqn:standard cone metric}\omega_{\bbeta} = \sum_{j=1}^p\beta_j^2 \frac{\sqrt{-1}dz_j\wedge d\bar z_j}{|z_j|^{2(1-\beta_j)}}  + \sum_{j=p+1}^n \sqrt{-1} dz_j \wedge d\bar z_j.\end{equation}
We shall use $s_{2p+1},\ldots s_{2n}$ to denote the real coordinates of $\mathbb C^{n-p} = \mathbb R^{2n-2p}$, such that for $j = p+1,\ldots, n$ 
$$z_j = s_{2j-1} + \sqrt{-1}s_{2j}. $$ 


In this paper we wills study  the following conical Laplacian equation with the background metric $g_\bbeta$ on $\mathbb C^n$
\begin{equation}\label{eqn:main equation}
\Delta_\bb u = f,\quad \text{in }B_{\bbeta}(0,1)\backslash \sS,
\end{equation}
where  $B_\bbeta(0,1)$ is the unit ball with respect to $g_\bbeta$ centered at $0$. The Laplacian $\Delta_\bbeta$ is defined as
$$\Delta_\bbeta u = \sum_{j,k} g_{\bbeta}^{j\bar k} \frac{\partial^2 u }{\partial z_j \partial \bar z_k} = \sum_{j=1}^p|z_j|^{2(1-\beta_j)} \frac{\partial^2 u}{\partial z_j \partial \bar z_j} +\sum_{j=p+1}^n \frac{\partial^2 u}{\partial z_j  \partial \bar z_j}. $$
We always assume $f\in C^0(B_{\bbeta}(0,1))$ and $u\in C^0(\overline{B_{\bbeta}(0,1)})\cap C^2(B_{\bbeta}(0,1)\backslash \sS)$. 
Throughout this paper, given a continuous function $f$ we denote
$$\omega(r):=\omega_f(r) = \sup_{z,w\in B_\bbeta(0,1), d_\bbeta(z,w)<r} | f(z) - f(w) |$$ the oscillation of $f$ with respect to $g_\bbeta$ in the ball $B_\bbeta(0,1)$. It is clear that $\omega(2 r)\le 2 \omega(r)$ for any $r<1/2$. We say a continuous function $f$ is {\em Dini continuous} if $\int_0^1 \frac{\omega(r)}{r}dr <\infty$.
\begin{defn}
We will write the (weighted) polar coordinates of $z_j$ for $1\le j\le p$ as 
$$r_j = |z_j|^{\beta_j},\quad \theta_j = \arg z_j.$$ We denote $D'$ to be one of the first order operators $\{\frac{\partial}{\partial s_{2p+1}},\ldots,\frac{\partial }{\partial s_{2n}}\}$,  and $ N_j$ to be one of the operators $\{ \frac{\partial}{\partial r_j}, \frac{\partial }{\beta_j r_j \partial \theta_j}\}$ which as vector fields are transversal to $\sS_j$. 
\end{defn}

Our first main result is the H\"older estimates of the solution $u$ to the equation \eqref{eqn:main equation}. 

\begin{theorem}\label{thm:main 1}
Suppose $\bbeta\in (1/2,1)^p$ and $f\in C^0(B_\bbeta(0,1))$ is Dini continuous with respect to $g_\bbeta$. Let $u\in C^0(\overline{B_\bbeta(0,1)})\cap C^2(B_\bbeta(0,1)\backslash \sS)$ be the solution to the  equation \eqref{eqn:main equation}, then there exists  $C=C(n,\bbeta)>0$ such that for any two points $p,q\in B_\bbeta(0,1/2)\backslash \sS$,  
%
%
%
{\small
\eqsp{\label{eqn:usual estimate new}
& ~| (D')^2 u(p) - (D')^2 u(q)  | + \sum_{j=1}^p \Ba{|z_j|^{2(1-\beta_j)} \frac{\partial^2 u}{\partial z_j\partial\bar z_j}(p) - |z_j|^{2(1-\beta_j)}\frac{\partial^2 u}{\partial z_j\partial\bar z_j}(q)     }\\
\le & ~ C\bk{ d \| u\|_{L^\infty(B_\bbeta(0,1))} + \int_0^d \frac{\omega(r)}{r}dr + d \int_d ^1 \frac{\omega(r)}{r^2} dr  },
}  
}
for any $1\le j\le p$, 
{\small 
\begin{equation}\label{eqn:1.3}
| N_j D' u(p) - N_j D' u(q)  | \le C\bk{ d^{\frac{1}{\beta_j} - 1} \| u\|_{L^\infty(B_\bbeta(0,1))} + \int_0^d \frac{\omega(r)}{r}dr + d^{\frac{1}{\beta_j} - 1} \int_d ^1 \frac{\omega(r)}{r^{1/\beta_j}} dr    },
\end{equation}
}
and for any $1\le j, k\le p$ with $j\neq k$, 
{\small
\begin{equation}\label{eqn:newly added 1}
| N_j N_k u(p) - N_j N_k u(q)  | \le C\bk{ d^{\frac{1}{\beta_{\max}} - 1} \| u\|_{L^\infty(B_\bbeta(0,1))} + \int_0^d \frac{\omega(r)}{r}dr + d^{\frac{1}{\beta_{\max}} - 1} \int_d ^1 \frac{\omega(r)}{r^{1/\beta_{\max}}} dr    },
\end{equation}
}
where $d=d_\bbeta(p,q)>0$ is the $g_\bbeta$-distance of $p$ and $	q$ and $\beta_{\max} = \max\{ \beta_1,\ldots, \beta_p  \}\in (1/2,1)$.
\end{theorem}
\begin{remark}
\begin{enumerate}[label=(\arabic*), fullwidth]
\item We remark that the number $\beta_{\max}$  on the RHS of \eqref{eqn:newly added 1} can be replaced by $\max\{\beta_j,\beta_k\}$.

\medskip 

\item In Theorem \ref{thm:main 1} and Theorem \ref{thm:main 2} below, we assume $\bb\in (1/2,1)^p$ just for exposition purposes and cleanness of the statement. When some of angles $\beta_j$ lie in $(0,1/2]$, the pointwise H\"older estimates in Theorem \ref{thm:main 1} are adjusted as follows: if $\beta_j\in (0,1/2]$ in  \eqref{eqn:1.3},  we replace the RHS  by the RHS of \eqref{eqn:usual estimate new}. In \eqref{eqn:newly added 1}, if both $\beta_j\text{ and } \beta_k\in (0,1/2]$, we also replace the RHS  by that of \eqref{eqn:usual estimate new}; if at least one of the $\beta_j,\beta_k$ is bigger than $1/2$, \eqref{eqn:newly added 1} remains unchanged. The inequalities in Theorem \ref{thm:main 2} can be adjusted similarly. The proofs of these estimates are contained in the proof of the case when $\beta_j\in (1/2, 1)$ by using the corresponding estimates in \eqref{eqn:2.5}.  
\end{enumerate}

\end{remark}
An immediate corollary of Theorem \ref{thm:main 1} is a precise form of Schauder estimates for equation \eqref{eqn:main equation}.
 \begin{corr}\label{corr:1.1}
Given $\bb\in (0,1)^p$ and $f\in C^{0,\alpha}_\bbeta(\overline{B_\bbeta(0,1)})$ for some $0<\alpha< \min\{1,\frac{1}{\beta_{\max}} -1\}$, if $u\in C^0(B_\bbeta(0,1))\cap C^2(B_\bbeta(0,1)\backslash \sS)$ solves equation \eqref{eqn:main equation}, then $u\in C^{2,\alpha}_{\bbeta}(B_\bbeta(0,1)  )$. Moreover, for any compact subset $K\Subset B_\bb(0,1)$, there exists a constant $C=C(n, \bb, K)>0$ such that the following estimate holds (see Definition \ref{defn:2.2} for the notations)
\begin{equation}\label{eqn:Holder}
\| u\|_{C^{2,\alpha}_\bbeta ( K  )} \le C   \bk{ \| u\|_{C^0( B_\bbeta(0,1)  )}  + \frac{ \| f\|_{C^{0,\alpha}_\bbeta (B_\bbeta(0,1))  }  }{\alpha ( \min\{\frac{1}{\beta_{\max}} - 1, 1\}  -\alpha)}  }.
\end{equation}

\end{corr}
\begin{remark}
A scaling-invariant version of the Schauder estimate \eqref{eqn:Holder} is that for any $0< r < 1$, there exists a constant $C=C(n,\bbeta, \alpha)>0$ such that (see Definition \ref{defn:2.3} for the notations)
\begin{equation}\label{eqn:rescale Holder}
\| u\|^*_{C^{2,\alpha}_\bbeta ( B_\bbeta(0,r)   )} \le C \bk{ { \| u\|_{C^0( B_\bbeta(0,r)  )}} + \| f\|_{C^{0,\alpha}_\bb ( B_\bb(0,r)  )  }^{(2)}}, 
\end{equation}
which follows from a standard rescaling argument by scaling $r$ to $1$.
\end{remark}

Let $g$ be a $C^{0,\alpha}_\bb$-conical K\"ahler metric on $B_\bb(0,1)$ (see Definition \ref{defn:3.1} below). By definition $g$ is equivalent to $g_\bb$. We consider the equation
\begin{equation}\label{eqn:section 1.1}
\Delta_g u = f\text{ in }B_\bb(0,1),\text{ and } u = \varphi\text{ on }\partial B_\bb(0,1),
\end{equation}for some $\varphi\in C^0(\partial B_\bb(0,1))$.
The following theorem is the generalization of Corollary \ref{corr:1.1} for non-flat background conical K\"ahler metrics, which is useful for applications of global geometric complex Monge-Amp\`ere equations. 
\begin{theorem} \label{thm:1.2}
For any given $\bb\in (0,1)^p$, $f\in C^{0,\alpha}_\bb( \overline{B_\bb(0,1)})$  and $\varphi\in C^0(\partial B_\bb(0,1))$, there is a unique solution $u\in C^{2,\alpha}_{\bb}(B_\bb(0,1))\cap C^0(\overline{B_\bb(0,1)})$ to the equation \eqref{eqn:section 1.1}. Moreover, for any compact subset $K\Subset B_\bb(0,1)$, there exits $C=C(n, \bb, \alpha, g, K)>0$ such that 
\begin{equation*}
\| u\|_{C^{2,\alpha}_\bbeta ( K  )} \le C  \bk{ \| u\|_{C^0( B_\bbeta(0,1)  )}  +{ \| f\|_{C^{0,\alpha}_\bbeta (B_\bbeta(0,1))  }  } }.
\end{equation*}

\end{theorem}

Theorem \ref{thm:1.2} can immediately be applied to study complex Monge-Amp\`ere equations with prescribed conical singularities along divisors of simple normal crossings and most of the geometric and analytic results for canonical K\"ahler metrics with conical singularities along a smooth divisor can be generalized to the case of simple normal crossings.

\medskip

We now turn to the parabolic Schauder estimates for the solution $u\in \cC^0(\qq_\bb)\cap \cC^2( \qq_\bb^\#  )$ to the equation 
\begin{equation}\label{eqn:main para}
\frac{\partial u}{\partial t} = \Delta_{g_\bb} u + f,\quad 
\end{equation}
for a Dini continuous function  $f$ in $\qq_\bb$, where for notation convenience we write $\qq_\bb: = B_\bb(0,1)\times (0,1]$ and $\qq_\bb^\# := B_\bb(0,1)\backslash \sS \times (0,1]$. Our second main theorem is the following pointwise estimate. 
\begin{theorem}\label{thm:main 2}
Suppose $\bb\in (1/2, 1)^p$ and $u$ is the solution to \eqref{eqn:main para}. Then there exists a computable constant $C=C(n,\bb)>0$ such that for any $Q_p= (p,t_p)$, $Q_q=(q,t_q)\in B_\bb(0,1/2)\backslash \sS\times (\hat t,1]$ (for some $\hat t\in (0,1)$) such that 
{\small
\begin{align*}
& ~| (D')^2 u(Q_p) - (D')^2 u(Q_q)  | + \sum_{j=1}^p \Ba{|z_j|^{2(1-\beta_j)} \frac{\partial^2 u}{\partial z_j\partial\bar z_j}(Q_p) - |z_j|^{2(1-\beta_j)}\frac{\partial^2 u}{\partial z_j\partial\bar z_j}(Q_q)     }\\
 + &~~  \ba{ \frac{\partial u}{\partial t}(Q_p) - \frac{\partial u}{\partial t}(Q_q)     } \le ~ C\bk{\frac{ d}{\hat t^{3/2}} \| u\|_{L^\infty(B_\bbeta(0,1))} + \hat t^{-1}\int_0^d \frac{\omega(r)}{r}dr + \frac{d}{\hat t^{3/2}} \int_d ^1 \frac{\omega(r)}{r^2} dr  },
\end{align*}  }
and for any $1\le j\le p$
{\small
\begin{equation*}
| N_j D' u(Q_p) - N_j D' u(Q_q)  | \le C\bk{ \frac{d^{\frac{1}{\beta_j} - 1}}{\hat t^{3/2}} \| u\|_{L^\infty(B_\bbeta(0,1))} + \hat t^{-1}\int_0^d \frac{\omega(r)}{r}dr + \frac{d^{\frac{1}{\beta_j} - 1}}{\hat t^{3/2}} \int_d ^1 \frac{\omega(r)}{r^{1/\beta_j}} dr    },
\end{equation*}  }
and for any $1\le j, k\le p$ with $j\neq k$ 
{\small
\begin{equation*}
| N_j N_k u(Q_p) - N_j N_k u(Q_q)  | \le C\bk{ \frac{d^{\frac{1}{\beta_{\max}} - 1}}{\hat t^{3/2}} \| u\|_{L^\infty(B_\bbeta(0,1))} + \hat t^{-1} \int_0^d \frac{\omega(r)}{r}dr +\frac{ d^{\frac{1}{\beta_{\max}} - 1} }{\hat t^{3/2}}\int_d ^1 \frac{\omega(r)}{r^{1/\beta_{\max}}} dr    },
\end{equation*} }
where $d=d_{\pp,\bbeta}(Q_p,Q_q)>0$ is the parabolic $g_\bbeta$-distance of $Q_p$ and $Q_q$,  and $\beta_{\max} = \max\{ \beta_1,\ldots, \beta_p  \}$, and $\omega(r)$ is the oscillation of $f$ in $\qq_\bb$ under the parabolic distance $d_{\pp,\bb}$ (c.f. Section \ref{section parabolic notations}).

\end{theorem}

If $f\in \cC_\bb^{\alpha, \frac{\alpha}{2}}(\qq_\bb)$ for some $\alpha\in ( 0, \min( \frac{1}{\beta_{\max}} - 1, 1   )  )$, then we have the following precise estimates as the parabolic analogue of Corollary \ref{corr:1.1}.
\begin{corr}
Suppose $\bb\in (0,1)^p$ and $u\in \cC^0(\qq_\bb)\cap \cC^2( \qq_\bb^\#  )$ satisfies the equation \eqref{eqn:main para}, then there exists a constant $C=C(n,\bb)>0$ such that (see Definition \ref{defn:2.5} for the notations)
$$\|u\|_{\cC^{2+\alpha,\frac{\alpha + 2}{2}   }_{\bb} \big ( B_\bb(0, 1/2) \times (1/2, 1]  \big )     }  \le C \bk{ \| u\|_{\cC^0(\qq_\bb)} + \frac{\| f\|_{\cC^{\alpha,\frac{\alpha}{2}}_\bb( \qq_\bb   )}} {\alpha ( \min\{\frac{1}{\beta_{\max}} - 1, 1\}  -\alpha)}    }.    $$
\end{corr}

For general non-flat $\cC^{\alpha,\alpha/2}_\bb$-conical K\"ahler metrics $g$, we consider the linear parabolic equation  
\begin{equation}\label{eqn:section 1.2}
\frac{\partial u}{\partial t} = \Delta_g u + f,\text{ in }\qq_\bb, \, \, u = \varphi\text{ on }\partial_\pp\qq_\bb.
\end{equation}  
We then have  the following parabolic Schauder estimates as an analogue of Theorem \ref{thm:1.2}.
\begin{theorem}\label{prop:section 1.1 new} 
Given $\bb\in (0,1)^p$, $f\in \cC^{\alpha,\alpha/2}_\bb(\overline{\qq_\bb} )$ and $\varphi\in \cC^0(\partial_\pp \qq_\bb)$, there exists a unique solution $u\in \cC^{2+\alpha,\frac{\alpha+2}{2}}_\bb\xk{ B_\bb(0,1)\times (0,1] }\cap \cC^0(\overline{\qq_\bb})$ to the Dirichlet boundary value problem \eqref{eqn:section 1.2}. For any compact subset $K\Subset B_\bb(0,1)$ and $\varepsilon_0>0$ there exists $C=C(n,\bb,\alpha,K,\varepsilon_0,g)>0$ such that  the following interior Schauder estimate holds
$$ \| u\|_{ \cC^{2+\alpha,\frac{2+\alpha}{2}}_\bb( K\times [\varepsilon_0, 1]  )  }\le C\xk{ \| u\|_{\cC^0(\qq_\bb)}  + \| f\|_{ \cC^{\alpha,\alpha/2}_\bb(\qq_\bb)  } }.   $$
Furthermore, if we assume $u|_{t= 0 } = u_0\in C^{2,\alpha}_\bb(B_\bb(0,1))$, then $u\in \cC^{2+\alpha,\frac{\alpha+2}{2}}_\bb( B_\bb(0,1)\times [0,1]   )$ and there exists a constant $C=C(n,\bb,\alpha, g,K)>0$ such that 
\begin{equation*}
\| u\|_{\cC^{2+\alpha,\frac{\alpha+2}{2}}_\bb( K\times [0,1]  )  }\le C \xk{ \| u\|_{\cC^0(\qq_\bb)} + \| f\|_{\cC^{\alpha,\alpha/2}_\bb(\qq_\bb)} + \| u_0\|_{C^{2,\alpha}_\bb(B_\bb(0,1))}   }.
\end{equation*}

\end{theorem}

As an application of Theorem \ref{prop:section 1.1 new}, we derive the short-time existence of the conical K\"ahler-Ricci flow with background metric being conical along divisors with simple normal crossings.

\smallskip

Let $(X,D)$ be a compact K\"ahler manifold, where $D = \sum_j D_j$ is a finite union of smooth divisors $\{D_j\}$ and $D$ has only simple normal crossings. Let $\omega_0$ be a $C^{0,\alpha'}_\bb(X)$-conical K\"ahler metric with cone angle $2\pi\bb$ along $D$ (see Definition \ref{defn:2.7}) and $\hat \omega_t$ be a family of conical metrics with bounded norm $\|\hat \omega\|_{\cC_\bb^{\alpha',\alpha'/2}  }$ and $\hat \omega_0 = \omega_0$. We consider the complex Monge-Amp\`ere flow
\begin{equation}\label{eqn:MA section1.1}
\frac{\partial \varphi}{\partial t} = \log\bk{ \frac{ (\hat\omega_t + \ddb \varphi)^n  }{\omega_0^n}   } + f,
\text{ and }\varphi|_{t= 0} = 0,
\end{equation}
for some $f\in \C^{\alpha',\alpha'/2}_\bb(X\times [0,1]  )$. 
\begin{theorem}\label{thm:1.5}Given $\alpha\in (0,\alpha')$,
there exists $T = T(n,\hat \omega, f,\alpha',\alpha)>0$ such that \eqref{eqn:MA section1.1} admits a unique solution $\varphi \in \cC^{2+\alpha,\frac{2+\alpha}{2}}_\bb(X\times [0,T])$. 
\end{theorem}

An immediate corollary of Theorem \ref{thm:1.5} is the short time existence for the conical K\"ahler-Ricci flow defined as below
%
\begin{equation}\label{eqn:KRF}
\frac{\partial \omega}{\partial t} = -\ric(\omega) + \sum_j (1-\beta_j) [D_j],\quad \omega|_{t= 0 } = \omega_0,
\end{equation}
where $\ric(\omega)$ is the unique extension of the Ricci curvature of $\omega$ from $X\setminus D$ to $X$ and $[D_j]$ denotes the current of integration over the component $D_j$. In addition  we assume  $\omega_0$ is a $C^{0,\alpha'}_\bb(X,D)$-conical K\"ahler metric such that 
\begin{equation}\label{eqn:omega 0 assumption}\omega_0^n = \frac{\Omega}{\prod_j (|s_j|_{h_j}^2)^{1-\beta_j}},\end{equation} where $s_j, \, h_j$ are holomorphic sections and hermitian metrics of the line bundle associated to $D_j$, respectively, and $\Omega$ is a smooth volume form.

\begin{corr} \label{corr:1.3}
For any given $\alpha \in (0,\alpha')$, there exists a constant $T = T(n,\omega_0, \alpha, \alpha')>0$ such that the conical K\"ahler-Ricci flow \eqref{eqn:KRF} admits a unique solution $\omega = \omega_t$, such that $\omega\in \C^{\alpha,\alpha/2}_\bb( X\times [0, T]  )$ and for each $t\in [0, T]$, $\omega_t$ is still a conical metric with cone angle $2\pi \bb$ along $D$. 

Furthermore, $\omega$ is smooth in $X\backslash D \times (0, T]$ and the (normalized) Ricci potentials of $\omega$, $\log \xk{ \frac{\omega^n}{\omega_0^n}  }$ is still in $\C^{2+\alpha, \frac{2+\alpha}{2}}_\bb(X\times [0,T])$.


\end{corr}

The short time existence of the conical K\"ahler-Ricci flow with singularities along a smooth divisor is derived in \cite{CW} by adapting the elliptic potential techniques of Donaldson \cite{D}.  Corollary \ref{corr:1.3} treats the general case of conical singularities with simple normal crossings. 
There have been many results in the analytic aspects of the conical Ricci flow \cite{CW, CW2, E, E1, JLZ, MRS, Y}. In \cite{PSSW}, the conical Ricci flow on Riemann surfaces is completely classified with jumping conical structure in the limit. Such phenomena is also expected in higher dimension, but it requires much deeper and delicate technical advances both in analysis and geometry.

\section{Preliminaries}

We explain the notations and give some preliminary tools which will be used later in this section.

\subsection{Notations}

To distinguish the elliptic from parabolic norms, we will use the ordinary $C$ to denote the norms in the elliptic case and the script $\C$ to denote the norms in the parabolic case. 

\smallskip

We always assume the H\"older component $\alpha$ appearing in $C^{0,\alpha}_\bb$ or $\C^{\alpha,\alpha/2}_\bb$ (or other H\"older norms) to be in $ \xk{ 0, \min\{\beta_{\max}^{-1} - 1, 1\}   }  $.

\subsubsection{Elliptic case.} We will denote $d_\bbeta(x,y)$ to be the distance of two points $x,y\in\mathbb C^n$ under the metric $g_\bbeta$.  $B_{\bbeta}(x,r)$ will be the metric ball under the metric induced by $g_{\bbeta}$ with radius $r$ and center $x$. It is well-known that $(\mathbb C^n\backslash \sS,g_\bbeta)$ is geodesically convex, i.e. any two points $x,y\in \mathbb C^n\backslash \sS$ can be joined by a $g_\bbeta$-minimal geodesic $\gamma$ which is disjoint with $\sS$.

\begin{defn}\label{defn:2.2}We define the  $g_\bbeta$-H\"older norm of functions $u\in C^0(B_\bbeta(0,r))$ for some $\alpha\in (0,1)$ as 
$$ \| u\|_{C^{0,\alpha}_{\bbeta} (B_\bbeta (0,r)  )  } := \| u\|_{C^0( B_\bbeta(0,r)   )} + [u]_{C^{0,\alpha}_\bb(B_\bb(0,r))}, $$ where the semi-norm is defined as  $[u]_{C^{0,\alpha}_\bb(B_\bb(0,r))}:=     \sup_{x\neq y \in B_\bbeta(0,r)} \frac{|u(x) - u(y)|}{d_\bbeta(x,y)^\alpha}.$ We denote the subspace of all continuous functions $u$ such that $\| u\|_{C^{0,\alpha}_\bbeta}<\infty$ as $C^{0,\alpha}_{\bbeta}(B_\bbeta(0,r))$.
\end{defn}
\begin{defn}
The $C_\bb^{2,\alpha}$ norm of a function $u$ on $B_\bbeta(0,r)=:B_\bbeta$ is defined as:
\begin{align*}
\| u\|_{C^{2,\alpha}_\bbeta(B_\bbeta)}: = & \| u\|_{C^0(B_\bbeta)} + \| \nabla_{g_\bbeta} u\|_{C^0(B_\bbeta, g_\bbeta)} + \sum_{j=1}^p \|N_j D' u  \|_{C^{0,\alpha}_\bbeta(B_\bbeta)} \\
& + \sum_{1\le j\neq k\le p} \| N_j N_k u  \|_{C^{0,\alpha}_{\bbeta}(B_\bbeta)} + \sum_{j=1}^p \Big\| |z_j|^{2(1-\beta_j)}\frac{\partial^2 u}{\partial z_j \partial \bar z_j}  \Big\|_{C^{0,\alpha}_\bbeta(B_\bbeta)}.
\end{align*} 
\end{defn}

For a given set $\Omega\subset B_\bb(0,1)$ we define the following weighted (semi)norms.
\begin{defn}\label{defn:2.3}
Suppose  $\sigma\in\mathbb R$ is a given real number and $u$ is a $C^{2,\alpha}_\bb$-function in $\Omega$. We denote $d_x = d_\bb(x,\partial \Omega)$ for any $x\in \Omega$. We define the weighted (semi)norms 
\begin{equation*}
[u]^{(\sigma)}_{C^{0,\alpha}_\bb(\Omega)} = \sup_{x\neq y\in \Omega} \min(d_x,d_y  )^{\sigma + \alpha} \frac{| u(x) - u(y)  |}{d_\bb(x,y)^\alpha},
\end{equation*} 
\begin{equation*}
\| u\|_{C^0(\Omega)}^{(\sigma)} = \sup_{x\in \Omega} d_x^\sigma |u(x)|,\quad 
[u]^{(\sigma)}_{C^{1}_\bb(\Omega)} = \sup_{x\in \Omega\backslash\sS} d_x ^{\sigma + 1} \big(\sum_j |N_j u|(x) + |D'u|(x)  \big )  
\end{equation*} 
\begin{equation*}
[u]^{(\sigma)}_{C^{2}_\bb(\Omega)} = \sup_{x \in \Omega\backslash\sS} d_x^{\sigma+2 }|Tu(x)|,
\end{equation*} 
\begin{equation*}
[u]^{(\sigma)}_{C^{2,\alpha}_\bb(\Omega)} = \sup_{x\neq y\in \Omega\backslash\sS} \min(d_x,d_y  )^{\sigma+2 + \alpha} \frac{| Tu(x) - T u(y)  |}{d_\bb(x,y)^\alpha},
\end{equation*} 
and
\begin{equation*}
\| u\|_{C^{2,\alpha}_\bb(\Omega)}^{(\sigma)} = \| u\|_{C^0(\Omega)}^{(\sigma)} + [u]^{(\sigma)}_{C^1_\bb(\Omega)} + [u]^{(\sigma)}_{C^2_\bb(\Omega)} + [u]^{(\sigma)}_{C^{2,\alpha}_\bb(\Omega)},
\end{equation*}
where $T$ denotes the following operators of second order: \begin{equation}\label{eqn:operator T}\Big\{ |z_j|^{2(1-\beta_j)}\frac{\partial^2}{\partial z_j\partial \bar z_j}, N_jN_k~( j\neq k), N_j D', (D')^2   \Big\}.\end{equation}
When $\sigma = 0$, we denote the norms above as $[\cdot]^*$, $\| \cdot\|^*$ for notation simplicity.
\end{defn}

\subsubsection{Parabolic case.} \label{section parabolic notations}

We denote $\mathcal Q_\bb = \mathcal Q_\bb(0,1) = B_\bb(0,1)\times (0,1]$ to be parabolic cylinder and 
$$\partial_{\mathcal P} \mathcal Q_\bb = ( \overline{B_\bb(0,1)}\times\{0\}  ) \cup (\partial B_\bb(0,1) \times (0,1] ) $$
to be the parabolic boundary of the cylinder $\mathcal Q_\bb$. We write $\sS_{\mathcal P} = \sS \times [0,1]$ as the singular set and $\qq_\bb^\# = \qq_\bb\backslash \sS_\pp$ the complement of $\sS_\pp$. For any two space-time points $Q_i = (p_i,t_i)$, we define their parabolic distance $d_{\mathcal P, \bb}(Q_1,Q_2)$ as
$$d_{\mathcal P,\bb}(Q_1, Q_2) = \max\{ \sqrt{|t_1 - t_2|}, d_\bb(p_1, p_2)   \}. $$

\begin{defn}\label{defn:2.4}We define the  $g_\bbeta$-H\"older norm of functions $u\in \C^0(\qq_\bb)$ for some $\alpha\in (0,1)$ as 
$$ \| u\|_{\C^{\alpha,\alpha/2}_{\bbeta} (\qq_\bb )  } := \| u\|_{\C^0( \qq_\bb)} + [u]_{\C^{\alpha,\alpha/2}_\bb(\qq_\bb)}, $$ where the semi-norm is defined to be $[u]_{\C^{\alpha,\alpha/2}_\bb(\qq_\bb)}:=     \sup_{Q_1\neq Q_2 \in \qq_\bb} \frac{|u(Q_1) - u(Q_2)|}{d_{\pp,\bb}(Q_1,Q_2)^\alpha}.$ We denote the subspace of all continuous functions $u$ such that $\| u\|_{\C^{\alpha,\alpha/2}_\bbeta(\qq_\bb)}<\infty$ as $\C^{\alpha,\alpha/2}_{\bbeta}(\qq_\bb)$.
\end{defn}

\begin{defn}\label{defn:2.5}
The $\C_\bb^{2+\alpha,\frac{\alpha+2}{2}}$ norm of a function $u$ on $\qq_\bb$ is defined as:
\begin{align*}
\| u\|_{\C^{2+\alpha,\frac{\alpha+2}{2}}_\bbeta(\qq_\bb)}: = & \| u\|_{\C^0(\qq_\bbeta)} + \| \nabla_{g_\bbeta} u\|_{\C^0(\qq_\bbeta, g_\bbeta)} + \|\newT u \|_{\C^{\alpha,\alpha/2}_\bbeta(\qq_\bbeta)},
\end{align*} 
where $\newT $ denotes all the second order operators in \eqref{eqn:operator T} and the first order operator $\frac{\partial }{\partial t}$. 
\end{defn}

For a given set $\Omega\subset \qq_\bb$ we define the following weighted (semi)norms.
\begin{defn}
Suppose $\sigma\in\mathbb R$ is a real number and $u$ is a $\C^{2+\alpha,\frac{\alpha+2}{2}}_\bb$-function in $\Omega$. We denote $d_{\pp, Q} = d_{\pp,\bb}(Q,\partial_\pp \Omega)$ for any $Q\in \Omega$. We define the weighted (semi)norms 
\begin{equation*}
[u]^{(\sigma)}_{\C^{\alpha,\alpha/2}_\bb(\Omega)} = \sup_{Q_1\neq Q_2\in \Omega} \min(d_{\pp, Q_1},d_{\pp,Q_2}  )^{\sigma + \alpha} \frac{| u(Q_1) - u(Q_2)  |}{d_{\pp,\bb}(Q_1,Q_2)^\alpha},
\end{equation*} 
\begin{equation*}
\| u\|_{\C^0(\Omega)}^{(\sigma)} = \sup_{Q\in \Omega} d_{\pp,Q}^\sigma |u(Q)|,\quad 
[u]^{(\sigma)}_{\C^{1}_\bb(\Omega)} = \sup_{Q\in \Omega\backslash\sS_\pp} d_{\pp,Q} ^{\sigma + 1} \big(\sum_j |N_j u|(Q) + |D'u|(Q)  \big )  
\end{equation*} 
\begin{equation*}
[u]^{(\sigma)}_{\C^{2,1}_\bb(\Omega)} = \sup_{Q \in \Omega\backslash\sS_\pp} d_{\pp,Q}^{\sigma+2 }\;|\newT u(Q)|,
\end{equation*} 
\begin{equation*}
[u]^{(\sigma)}_{\C^{2+\alpha,\frac{\alpha+2}{2}}_\bb(\Omega)} = \sup_{Q_1\neq Q_2\in \Omega\backslash\sS_\pp} \min(d_{\pp, Q_1},d_{\pp,Q_2}  )^{\sigma+2 + \alpha} \frac{| \newT u(Q_1) - \newT  u(Q_2)  |}{d_{\pp,\bb}(Q_1,Q_2)^\alpha},
\end{equation*} 
and
\begin{equation*}
\| u\|_{\C^{2+\alpha,\frac{\alpha+2}{2}}_\bb(\Omega)}^{(\sigma)} = \| u\|_{\C^0(\Omega)}^{(\sigma)} + [u]^{(\sigma)}_{\C^1_\bb(\Omega)} + [u]^{(\sigma)}_{\C^{2,1}_\bb(\Omega)} + [u]^{(\sigma)}_{\C^{2+\alpha,\frac{\alpha+2}{2}}_\bb(\Omega)}.
\end{equation*}
When $\sigma = 0$, we denote the norms above as $[\cdot]^*$ or $\| \cdot\|^*$ for simplicity.
\end{defn}

\subsubsection{Compact K\"ahler manifolds}
Let $(X,D)$ be a compact K\"ahler manifold with a divisor $D=\sum_j D_j$ with simple normal crossings, i.e. on an open coordinates chart $(U,z_j)$ of any $x\in D$, $D\cap U$ is given by $\{z_1\cdots z_p=0\}$,  and $D_j\cap U = \{z_j = 0\}$ for any component $D_j$ of $D$. We fix a finite cover $\{U_a, z_{a, j}\}$ of $D$.

\begin{defn}\label{defn:2.7} A (singular) K\"ahler metric $\omega$ is called a conical metric with cone angle $2\pi \bb$ along $D$, if locally on any coordinates chart $U_a$, $\omega$ is equivalent to $\omega_\bb$ under the the coordinates $\{z_{a, j}\}$, where $\omega_\bb$ is the standard cone metric \eqref{eqn:standard cone metric} with cone angle $2\pi \beta_j$ along $\{z_{a,j} = 0\}$, and on $X\backslash \cup_a U_a$ $\omega$ is a smooth K\"ahler metric in the usual sense. 

A conical metric $\omega$ is in $C^{0,\alpha}_\bb(X,D)$ if for each $a$, $\omega$ is $C^{0,\alpha}_\bb(U_a)$ and on $X\backslash \cup_a U_a$ $\omega$ is smooth in the usual sense. Similarly we can define the $\C^{\alpha,\alpha/2}_\bb$-conical K\"ahler metrics on $X\times [0,1]$.
%
\end{defn}

\begin{defn}\label{defn:2.8}
A continuous function $u\in C^0(X)$ is said to be in $C^{0,\alpha}_\bb(X,D)$ if locally on each $U_a$, $u$ is in $C^{0,\alpha}_\bb(U_a)$ and on $X\backslash \cup_a U_a$ it is $C^{0,\alpha}$-continuous in the usual sense. We define the $C^{0,\alpha}_\bb(X,D)$-norm of $u$ as 
$$\| u\|_{C^{0,\alpha}_\bb(X,D)}: = \| u\|_{C^{0,\alpha}(X\backslash \cup_a U_a,\; \omega)} + \sum_a \| u\|_{C^{0,\alpha}_\bb(U_a)}. $$ The $C^{0,\alpha}_\bb(X,D)$-norm depends on the choice of finite covers, and another cover yields a different but equivalent norm. 
The space $C^{0,\alpha}_\bb(X,D)$ is clearly independent of the choice of finite covers.

The other spaces and norms like $C^{2,\alpha}_\bb(X,D)$, $\C^{\alpha,\alpha/2}_\bb( X\times [0,1], D )$, etc, can be defined similarly.

\end{defn}

\subsection{A useful lemma}
We will frequently use the following elementary estimates from \cite{GS}.

\begin{lemma}[Lemma 2.2 in \cite{GS}]Given $r\in (0,1]$, suppose $v\in C^0(B_{\mathbb C}(0,r))\cap C^2(B_{\mathbb C}(0,r)\backslash \{0\})$ satisfies  the equation $$|z|^{2(1-\beta_1)} \frac{\partial ^2 v}{\partial z\partial \bar z} = F,\quad \text{in } B_{\mathbb C} (0, r)\backslash \{0\},$$
for some $F\in L^\infty(B_{\mathbb C}(0,r))$, then we have the pointwise estimate that  for any $z\in B_{\mathbb C}(0,9r/10)\backslash \{0\}$
{\small
\begin{equation}\label{eqn:2.5}
\ba{ \frac{\partial v}{\partial z}(z)  }\le C\frac{\| v\|_{L^\infty}}{r} + C\|F\|_{L^\infty} \cdot \left\{
\begin{aligned}
&r^{2\beta_1 - 1},\text{ if }\beta_1\in (1/2,1)\\
&|z|^{2\beta_1 -1},\text{ if }\beta_1\in (0,1/2)\\
&\ba{ \log(|z|/2r)  }, \text{ if }\beta_1 = 1/2,
\end{aligned}\right.
\end{equation}
}
where the $L^\infty$-norms are taken in $B_{\mathbb C}(0,r)$ and $C>0$ is a uniform constant.
\end{lemma}

\medskip

Finally we remark that the idea of the proof of the estimates in Theorems \ref{thm:main 1} and \ref{thm:main 2} is the same for general $2\le p\le n$. To explain the argument clearer we prove the theorems assuming $p = 2$, i.e. the cone metric of $\omega_\bb$ is singular along the two components $\sS_1$ and $\sS_2$.

\section{Elliptic estimates} 

In this section, we will prove Theorems \ref{thm:main 1} and \ref{thm:1.2}, the Schauder estimates for the Laplace equation \eqref{eqn:main equation}. To begin with, we first observe the simple $C^0$-estimate based on maximum principle.

Suppose $u\in C^2(B_{\bbeta}(0,1)\backslash \sS)\cap C^0(\overline{B_{\bbeta}(0,1)})$ satisfies the equation
\begin{equation}\label{eqn:harmonic}
\left\{\begin{aligned}
&\Delta_\bb u = 0,~\text{in }B_{\bbeta}(0,1)\backslash \sS,\\
&u = \varphi,~ \text{on } \partial B_{\bbeta}(0,1)\end{aligned}\right.
\end{equation}
for some $\varphi\in C^0(\partial B_{\bbeta}(0,1))$, then
\begin{lemma}\label{lemma:MP} We have the following maximum principle,
\begin{equation}\label{eqn:MP 1}
\inf_{\partial B_{\bbeta}(0,1)}\varphi\le \inf_{B_{\bbeta}(0,1)} u\le \sup_{B_{\bbeta}(0,1)} u\le \sup_{\partial B_{\bbeta}(0,1)} \varphi.
\end{equation}
\end{lemma}
\begin{proof}Consider the functions $\tilde u_\epsilon = u \pm \epsilon (\log \abs{z_1} + \log\abs{z_2})$ for any $\epsilon>0$. By the same proof of Lemma 2.1 in \cite{GS}, \eqref{eqn:MP 1} is established.
\end{proof}

\medskip

Next step is to show the equation \eqref{eqn:harmonic} is solvable for suitable boundary values.
\subsection{Conical harmonic functions}

\subsubsection{Smooth approximating metrics}
Let $\epsilon\in (0,1)$ be a given small positive number and we define a smooth  approximating K\"ahler metric on $B_\bb(0,1)$
\begin{equation}\label{eqn:para app}g_\epsilon = \beta_1^2\frac{\sqrt{-1}dz_1\wedge d\bar z_1}{(\abs{z_1} + \epsilon)^{1-\beta_1}} + \beta_2^2 \frac{\sqrt{-1} d z_2\wedge d\bar z_2}{(\abs{z_2} + \epsilon)^{1-\beta_2}}  + \sum_{j=3}^n \sqrt{-1} dz_j\wedge d\bar z_j. \end{equation}
$g_\epsilon$ are product metrics on $\mathbb C\times\mathbb C \times \mathbb C^{n-2}$. It is clear that its Ricci curvature satisfies $$\ric(g_\epsilon) = \ddbar \log \xk{(\abs{z_1} +\epsilon)^{1-\beta_1} (\abs{z_2} + \epsilon)^{1-\beta_2}}\ge 0.$$

Let $u_\epsilon\in C^2(B_{\bbeta}(0,1))$ be the solution to the equation with a given $\varphi\in C^0(\partial B_{\bbeta}(0,1))$
\begin{equation}\label{eqn:app Diri}
\Delta_{g_\epsilon} u_\epsilon = 0,~\text{in } B_{\bbeta}(0,1),\text{ and }
u_\epsilon = \varphi,~ \text{on }\partial B_{\bbeta}(0,1).
\end{equation} 
Note that the metric balls $B_{\bbeta}(0,1)$ and $B_{g_\epsilon}(0,1)$ are uniformly close when $\epsilon$ is sufficiently small, so for the following estimates we will work on $B_{\bbeta}(0,1)$.

Let $u_\epsilon$ be the harmonic function for $\Delta_{\epsilon} = \Delta_{g_\epsilon}$ as in \eqref{eqn:app Diri}, which we may assume without of loss of generality is positive by replacing $u_\ep$ by $u_\ep - \inf u_\ep $ if necessary. We will  study the Cheng-Yau-type gradient estimate of $u_\epsilon$ and the estimate of $\Delta_{\epsilon,1} u_\epsilon := (\abs{z_1}+\epsilon)^{1-\beta_1}\frac{\partial^2 u_\epsilon}{\partial z_1\partial \bar z_1}$. Let us recall Cheng-Yau's gradient estimate first.

\smallskip

In Subsections \ref{subsection start} - \ref{subsection end}, for notation convenience, we will omit the subscript $\epsilon$ in $g_\epsilon,\,u_\epsilon$, in the proofs of lemmas.

\subsubsection{Cheng-Yau gradient estimate revisit}\label{subsection start}

 We will also assume $u_\epsilon>0$, otherwise consider the function $u_\epsilon +\delta$, for some $\delta>0$ then letting $\delta\to 0$. We fix a ball metric $B_{g_\ep}(p,R)\subset B_\bb(0,1)$ centered at some point $p\in B_\bb(0,1)$. Since $\ric(g_\ep)\ge 0$, the Cheng-Yau gradient estimate holds for $\Delta_{g_\ep}$-harmonic functions.

\begin{lemma}[\cite{CY}]\label{lemma:CY gradient}
Let $u_\ep \in C^2(B(p,R))$ be a positive $\Delta_{g_\ep}$-harmonic function. There exists a uniform constant $C=C(n)>0$ such that (the metric balls are taken under the metric $g_\ep$)
\begin{equation}\label{eqn:grad final prop}
\sup_{x\in B(p,3R/4)} |\nabla u_\ep|_{g_\ep}(x)\le C(n) \frac{\mathrm{osc}_{B(p,R)} u_\ep}{R}.
\end{equation}
\end{lemma}
 
As we mentioned above, we will omit the $\ep$ in the subscript of $u_\ep$ and $g_\ep$. The proof of the lemma is standard (\cite{CY}).  For completeness and to motivate the proof of Lemmas \ref{lemma 2.1} and \ref{lemma 2.2} below, we provide a sketched proof. 
Define $f= \log u$, and it can be calculated that
\begin{equation}\label{eqn:1st equation}\Delta f = \frac{\Delta u}{u} - \frac{\abs{\nabla u}}{u^2} = -\abs{\nabla f}.\end{equation}
Then by Bochner formula, we have
\eqsp{\label{eqn:Bochner new}\Delta \abs{\nabla f} &= \abs{\nabla \nabla f} + \abs{\nabla \bar \nabla f} + 2\Re\innpro{\nabla f, \bar \nabla \Delta f} + \ric(\nabla f,\bar \nabla f)\\
&\ge \abs{\nabla \bar \nabla f}  - 2 \Re \innpro{ \nabla f,\bar \nabla \abs{\nabla f}  }.}
Let $\phi:[0,1]\to [0,1]$ be a standard cut-off function such that $\phi|_{[0,3/4]} = 1$ and $\phi_{[5/6,1]} = 0$ and between $0,1$ otherwise. let $r(x) = d_{g_\epsilon}(p,x)$ be the distance function to $p$ under the metric $g = g_\epsilon$. By abusing notation we also write $\phi(x) = \phi\bk{\frac{r(x)}{R}}$. It can be calculated by Laplacian comparison and the Bochner formula \eqref{eqn:Bochner new} that at the (positive) maximum point $p_{\max}$ of $H:= \phi^2 \abs{\na f}$ that  
\begin{equation*}
\frac{2}{n}H^2 - \frac{4|\phi'|}{R} H^{3/2} - \frac{8(\phi')^2}{R^2}H +  \frac{2H}{R^2} \bk{(2n-1)\phi\phi' + \phi\phi'' + (\phi')^2   }\le 0,
\end{equation*}
therefore for any $x\in B(p, 3R/4)$
\begin{equation}\label{eqn:final gradient 1}
\frac{\abs{\nabla u}}{u^2}(x)=\abs{\nabla f(x)} = H(x)\le H(p_{\max})\le \frac{C(n)}{R^2}.\end{equation}

\subsubsection{Laplacian estimate in singular directions}\label{section 1.2}
We will prove the estimates of $\Delta_{j,\ep} u_\ep := (|z_j|^{2}+\ep)^{1-\beta_j} \frac{\partial u_\ep}{\partial z_j \partial \bar z_j}$ for a $\Delta_{g_\ep}$-harmonic function $u_\ep$.

\begin{lemma}\label{lemma 2.1}
Under the same assumptions as in Lemma \ref{lemma:CY gradient},
along the ``bad'' directions $z_1, z_2$, $\Delta_{\epsilon,1} u_\ep$ and $\Delta_{\epsilon,2} u_\ep$ satisfy the estimates
\begin{equation}\label{eqn:lap final prop}
\sup_{x\in B(p,R/2)}\bk{ |\Delta_{\epsilon,1} u_\epsilon|(x) +| \Delta_{\epsilon,2} u_\epsilon|(x)   }\le C(n)\frac{\mathrm{osc}_{B(p,R)}u_\epsilon}{R^2}.
\end{equation}

\end{lemma}
As in the proof of Cheng-Yau gradient estimates, we will work on the function $f=f_\ep = \log u$ and we only need to prove the estimate for $\Delta_{1,\ep} u_\ep$.  We write $\Delta_{1,\ep} f := (\abs{z_1} + \epsilon)^{1-\beta_1} \frac{\partial^2 f}{\partial z_1 \partial \bar z_1}. $

As above, we will omit the subscript $\ep$ in $\Delta_{1,\ep} f$.
We first observe that \begin{equation}\label{eqn:2}\Delta_1 \Delta_{g_\epsilon} f = \Delta_{g_\epsilon}\Delta_1 f. \end{equation}
\eqref{eqn:2} can be checked from the definitions by the property that $g_\epsilon$ is a product-metric. Indeed
{\small
\begin{align*}
\Delta_1 \Delta_{g_\ep} f& = (\abs{z_1}+\epsilon)^{1-\beta_1} \frac{\partial^2}{\partial z_1\partial \bar z_1}\bk{(\abs{z_1}+\epsilon)^{1-\beta_1} \frac{\partial^2 f}{\partial z_1\partial \bar z_1} + (\abs{z_2}+\epsilon)^{1-\beta_2} \frac{\partial^2 f}{\partial z_2\partial \bar z_2} +\sum_j \frac{\partial ^2 f}{\partial z_j\partial \bar z_j}       }\\
& = (\abs{z_1}+\epsilon)^{1-\beta_1} \frac{\partial^2}{\partial z_1\partial \bar z_1} \Delta_1 f + (\abs{z_2}+\epsilon)^{1-\beta_2} \frac{\partial^2}{\partial z_2\partial \bar z_2}\Delta_1 f + \sum_{j=3}^n \frac{\partial^2 }{\partial z_j\partial \bar z_j} \Delta_1 f\\
& = \Delta_{g_\epsilon} \Delta_1 f.
\end{align*}
}
On the other hand, note that by \eqref{eqn:1st equation} $\Delta_{g_\epsilon} f = \Delta_g f = -\abs{\nabla f}$. By choosing a normal frame $\{e_1,\ldots,e_n\}$ at some point $x$ such that $dg(x) = 0$ and $\Delta_1 f = f_{1\bar 1}$, we calculate
\eqsp{\label{eqn:useful 1}
\Delta_1 \abs{\nabla f}& = (f_j f_{\bar j})_{1\bar 1}\\
& = f_{j1} f_{\bar j\bar 1} + f_{j\bar 1} f_{\bar j 1} + f_{j1\bar 1} f_{\bar j} + f_{j}f_{\bar j 1 \bar 1}\\
&= f_{j1} f_{\bar j\bar 1} + f_{j\bar 1} f_{\bar j 1} +  f_{j}f_{\bar 1 1 \bar j} + f_{\bar j}\xk{f_{1\bar 1 j} + f_m R_{1\bar m j \bar 1}     }\\
& = \abs{\nabla_1 \nabla f} + \abs{\nabla_1\bar \nabla f}  + 2\Re \innpro{\nabla f, \bar \nabla \Delta_1 f} + f_m f_{\bar j} R_{1\bar 1 j\bar m}\\
&\ge (\Delta_1 f)^2 + 2\Re \innpro{\nabla f,\bar \nabla\Delta_1 f}.}
Then it follows that
\begin{equation*}
\Delta \xk{-\Delta_1 f }= -\Delta_1 \Delta f = \Delta_1 \abs{\nabla f} \ge (\Delta_1 f)^2 + 2\Re\innpro{\nabla f, \bar \nabla\Delta_1 f}.
\end{equation*}
Let $\varphi:[0,1]\to [0,1]$ be a standard cut-off function such that $\varphi|_{[0,1/2]} = 1$ and $\varphi|_{[2/3,1]} = 0$. We also define $\varphi(x) = \varphi\xk{ \frac{r(x)}{R} }$. Then consider the function $G := \varphi^2\cdot (-\Delta_1 f)$. We calculate
\eqsp{\label{eqn:grad 3}
\Delta G =& \Delta \xk{ \varphi^2 (-\Delta_1 f)  }\\
= &\varphi^2 \Delta(-\Delta_1 f) + 2 \Re\innpro{\nabla \varphi^2, \bar \nabla (-\Delta_1 f)} + (-\Delta_1 f) \Delta \varphi^2\\
\ge & \varphi^2 \xk{ (\Delta_1 f)^2 + 2 \Re\innpro{\nabla f, \bar \nabla \Delta_1 f} } + 2 \Re\innpro{\nabla \varphi^2, \bar \nabla (-\Delta_1 f)} + (-\Delta_1 f) \Delta \varphi^2.
} 
We want to estimate the upper bound of $G$. If the maximum value of $G = \varphi^2(-\Delta_1 f)$ is negative, we are done. So we assume the maximum of $G$ on $B(p,R)$ is {\bf positive}, which is achieved at some point $p_{\max}\in B(p,2R/3)$. Hence at $p_{\max}$, we have $(-\Delta_1 f)>0$. By Laplacian comparison that $\Delta r\le \frac{2n-1}{r}$,  we get at $p_{\max}$, 
\begin{equation}
\label{eqn:grad 2}
\Delta \varphi^2 \ge \frac{2}{R^2}\bk{ (2n-1) \varphi\varphi' + \varphi\varphi'' +(\varphi')^2   }.
\end{equation}
Thus at $p_{\max}$, the last term on RHS of \eqref{eqn:grad 3} is 
$$\ge (-\Delta _1 f) \frac{2}{R^2}\bk{ (2n-1) \varphi\varphi' + \varphi\varphi'' +(\varphi')^2   }.$$
Substituting this into \eqref{eqn:grad 3}, it follows that at $p_{\max}$, $\Delta G\le 0$ and $\nabla \Delta_1 f = - \frac{2}{\varphi} \Delta_1 f \nabla \varphi$ and 
{\small
\eqsp{ 0& \ge \Delta G\\
&\ge \varphi^2 (\Delta_1 f)^2 + 2 \varphi^2 \Re \innpro{\nabla f,\bar \nabla \Delta_1 f} + 4\varphi \Re \innpro {\nabla \varphi, \bar \nabla (-\Delta_1 f)} \\
&\qquad   +  (-\Delta _1 f) \frac{2}{R^2}\bk{ (2n-1) \varphi\varphi' + \varphi\varphi'' +(\varphi')^2   } \\
&\ge \varphi^2 (\Delta_1 f)^2 - 4 \varphi |\Delta_1 f| |\nabla f| |\nabla \varphi|  + 8 \Delta_1 f \abs{\nabla \varphi} +  (-\Delta _1 f) \frac{2}{R^2}\bk{ (2n-1) \varphi\varphi' + \varphi\varphi'' +(\varphi')^2   }\\
&  =    \frac{G^2}{\varphi^2}  - 4\varphi^{-1}G |\nabla f | |\nabla \varphi|  - 8 G \frac{\abs{\nabla \varphi}}{\varphi^2} + \frac{2 G}{R^2 \varphi^2}\bk{ (2n-1)\varphi\varphi' + \varphi\varphi'' +(\varphi')^2   }\\
&\ge \frac{G^2}{\varphi^2} - 4\frac{|\varphi'| |\nabla f|  }{R \varphi} G - 8\frac{\abs{\varphi'}}{R^2 \varphi^2} G + \frac{2 G}{R^2 \varphi^2}\bk{ (2n-1)\varphi\varphi' + \varphi\varphi'' +(\varphi')^2   }  .  }
}
Therefore at $p_{\max}\in B(p,2R/3)$, it holds that
\begin{equation*}
G^2 - 4 \frac{\varphi|\varphi'\nabla f|}{R} G - 8\frac{|\varphi'|^2 }{R^2} G + \frac{2G}{R^2}\bk{ (2n-1)\varphi\varphi' + \varphi\varphi'' +(\varphi')^2  }\le 0,
\end{equation*}
combining \eqref{eqn:final gradient 1} and the fact that $\varphi, \varphi', \varphi''$ are all uniformly bounded, we can get at $p_{\max}$
$$G^2 \le C(n) R^{-2} G\quad \Rightarrow\quad G(p_{\max})\le \frac{C(n)}{R^2}.$$
Then for any $x\in B(p,R/2)$, where $\varphi =1$, we have
$$ -\Delta_1 f(x)  =G(x)\le G(p_{\max})\le \frac{C(n)}{R^2}.$$
Moreover, recall that $f= \log u$  and $-\Delta_1 f = -\frac{\Delta_1 u}{u} + \abs{\nabla_1 f}$, 
therefore it follows that
\begin{equation}\label{eqn:lap 1}\sup_{x\in B(p,R/2)} \bk{ -\frac{\Delta_1 u}{u}(x)  }\le \frac{C(n)}{R^2}.   \end{equation}
This in particular implies that
\begin{equation}\label{eqn:lap 2}
\sup_{x\in B(p,R/2)} \xk{ -\Delta_1 u(x)  }\le C(n)\frac{\mathrm{osc}_{B(p,R/2)} u}{R^2}\le C(n)\frac{\mathrm{osc}_{B(p,R)} u}{R^2}.
\end{equation}

On the other hand, by considering the function $\hat u = \max_{B(p,R)} u - u$, which is still a positive $g_\epsilon$-harmonic function $\Delta_{g} \hat u = \Delta_{g_\epsilon} \hat u = 0$. Applying \eqref{eqn:lap 1} to the function $\hat u$, we get
\begin{equation}\label{eqn:lap 3}
 \sup_{x\in B(p,R/2)} \bk{ \frac{\Delta_1 u(x)}{\max_{B(p,R)} u - u(x)}   }   =\sup_{x\in B(p,R/2)} \bk{ -\frac{\Delta_1 \hat u}{\hat u}(x)   }\le \frac{C(n)}{R^2}
\end{equation}
which yields that
\begin{equation}\label{eqn:lap 4}
\sup_{x\in B(p,R/2)} \Delta_1 u(x)\le C(n)\frac{\mathrm{osc}_{B(p,R/2)} u}{R^2}.
\end{equation}
Combining  \eqref{eqn:lap 4} and \eqref{eqn:lap 2}, we get
\begin{equation}\label{eqn:lap final}
\sup_{x\in B(p,R/2)} |\Delta_1 u|(x) \le C(n) \frac{\mathrm{osc}_{B(p,R)} u}{R^2}.  
\end{equation}

\subsubsection{Mixed derivatives estimates} In this subsection, we will estimate the following mixed derivatives 
$$\abs{\nabla_1\nabla_2 f} = \frac{\partial^2 f}{\partial z_1\partial z_2} \overline{  \frac{\partial^2 f}{\partial z_1\partial z_2} } g^{1\bar 1} g^{2\bar 2},\quad \abs{\nabla_1\nabla_{\bar 2} f} = \frac{\partial^2 f}{\partial z_1\partial\bar z_2} \overline{  \frac{\partial^2 f}{\partial z_1\partial\bar  z_2} } g^{1\bar 1} g^{2\bar 2},$$
where  as before $f = \log u$ and $u$ is a positive harmonic function of $\Delta_{g_\epsilon}$. Here for simplicity, we omit the subscript $\epsilon$ in $u_\epsilon$, $f_\epsilon$ and $g_\epsilon$. Observing that since $g_\epsilon = g$ is a product metric with the non-zero components $g_{k\bar k}$ depending {only} on $z_k$, it follows that the curvature tensor
$$R_{i\bar j k\bar l} = -\frac{\partial^2 g_{i\bar j}}{\partial z_k\partial \bar z_l} + g^{p\bar q}\frac{\partial g_{i\bar q}}{\partial z_k} \frac{\partial g_{p\bar j}}{\partial\bar z_l}$$
vanishes {unless $i=j=k=l = 1 \text{ or }2$}, and also $R_{i\bar i i\bar i}\ge 0$ for all $i = 1,\ldots,n$.

\medskip

We fix some notations:  we will write $f_{12} = \nabla_1\nabla_2 f$ (in fact this is just the ordinary derivative of $f$ w.r.t. $g$, since $g$ is a product metric),  $|f_{12}|^2_g = \abs{\nabla _1 \nabla_2 f}_g$, etc. 

Let us first recall  that the equation \eqref{eqn:useful 1} implies
\eqsp{ \label{eqn:useful 2}
\Delta (-\Delta_1 f - \Delta_2 f) & = \sum_{k=1}^n \xk{ g^{1\bar 1} g^{k\bar k} f_{1k} f_{\bar 1 \bar k} +  g^{1\bar 1}g^{k\bar k} f_{1\bar k} f_{\bar 1 k} + g^{2\bar 2} g^{k\bar k} f_{2k}f_{\bar 2 \bar k} + g^{2\bar 2} g^{k\bar k} f_{2\bar k}f_{k\bar 2}}\\
&\quad - 2 \Re \innpro{ \nabla f,\bar \nabla (-\Delta_1 f - \Delta_2 f)  } + f_{1}f_{\bar 1} g^{1\bar 1}g^{1\bar 1} R_{1\bar 1 1\bar 1 } + f_{2}f_{\bar 2} g^{2\bar 2}g^{2\bar 2} R_{2\bar 2 2\bar 2 }\\
& \ge \sum_{k=1}^n \xk{ \abs{\nabla_1\nabla_k f} + \abs{\nabla_1\nabla_{\bar k} f} + \abs{\nabla_2\nabla_ k f} + \abs{\nabla_2\nabla_{\bar k} f}   } \\
&\quad - 2 \Re \innpro{ \nabla f,\bar \nabla (-\Delta_1 f - \Delta_2 f)  }.
} 

\medskip

Next we calculate $\Delta \abs{\nabla_1\nabla_2 f}$. For notation convenience we will write $f^{12} = f_{\bar 1\bar 2}g^{1\bar 1}g^{2\bar 2}$, and hence $\abs{\nabla_1\nabla_2 f} = f_{12} f^{12}$. We calculate
\eqsp{ \label{eqn:first step 1}
\Delta \abs{\nabla_1\nabla_2 f} & = g^{k\bar l} \xk{ f_{12} f^{12}  } _{k\bar l}
 = g^{k\bar k} \xk{ f_{12} f^{12}  } _{k\bar k}\quad \text{(since $g$ is a product metric)}\\
& = g^{k\bar k}\xk{f_{12k\bar k} f^{12} + f_{12k}f^{12}_{\;\; \;,\bar k} + f_{12\bar k}f^{12}_{\;\;\;, k} + f_{12}f^{12}_{\;\;\; ,k\bar k}   }.
}
The first term on the RHS of \eqref{eqn:first step 1} is (by Ricci identities and switching the indices)
\eqsp{& g^{k\bar k}f^{12}\xk{ f_{k1\bar k 2} + g^{m\bar m}f_{m1} R_{k\bar m 2 \bar k} + g^{m\bar m}f_{km}R_{1\bar m 2\bar k} }\\
=& g^{k\bar k}f^{12}\xk{ f_{k\bar k 12} +  g^{m\bar m} f_{m 2} R_{k\bar m 1\bar k} + g^{m\bar m}f_{m} R_{k\bar m 1 \bar k,2} + g^{m\bar m}f_{m1} R_{k\bar m 2 \bar k} + g^{m\bar m}f_{km}R_{1\bar m 2\bar k}  }\\
=& g^{k\bar k}f^{12}\xk{ f_{k\bar k 12} +  g^{m\bar m} f_{m 2} R_{k\bar m 1\bar k}  + g^{m\bar m}f_{m1} R_{k\bar m 2 \bar k}  }\\
=& g^{k\bar k}f^{12} f_{k\bar k 12} + g^{1\bar 1} g^{1\bar 1} f^{12} f_{21} R_{1\bar 1 1\bar 1} + g^{2\bar 2} g^{2\bar 2} f^{12} f_{12} R_{2\bar 2 2\bar 2},
}
and the last term on the RHS of  \eqref{eqn:first step 1} is the conjugate of the first term, hence we get
\eqsp{\label{eqn:first step 2} \Delta \abs{\nabla_1 \nabla_2 f} & = 2 \Re \xk{ f^{12} (\Delta f)_{12}   }  + 2 f^{12} f_{12} \xk{ g^{1\bar 1} g^{1\bar 1}R_{1\bar 1 1\bar 1} + g^{2\bar 2}g^{2\bar 2} R_{2\bar 2 2\bar 2}  }  \\
&\quad + g^{k\bar k} f_{12k} f^{12}_{\;\;\;, \bar k} + g^{k\bar k} f_{12\bar k}f^{12}_{\;\;\;, k}.    } 
Recall from \eqref{eqn:1st equation} we have $\Delta f = -\abs{\nabla f}$, hence the first term on RHS of \eqref{eqn:first step 2} is 
\eqsp{ \label{eqn:first step 3}
2 \Re \xk{ f^{12} (\Delta f)_{12}   } & = 2 \Re\xk{ f^{12}(-\abs{\nabla f})_{12}    }   \\
& = -2 \Re \bk{f^{12} g^{k\bar k} \xk{  f_{k12} f_{\bar k} + f_{k1} f_{\bar k2} + f_{k2} f_{\bar k 1} + f_{k}f_{\bar k 12}      }}\\
& = -2 \Re \bk{ f^{12} g^{k\bar k}\xk{ f_{12 k}f_{\bar k} + f_{k1}f_{\bar k 2}  + f_{k2} f_{\bar k1} + f_k f_{12\bar k} - f_k f_m R_{1\bar m 2k}g^{m\bar m}   }   }\\
&= -4 \Re \innpro{\nabla f,\bar \nabla \abs{\nabla_1\nabla_2 f}} - 2 \Re \bk{f^{12}g^{k\bar k} f_{k1}f_{2\bar k} + f^{12} g^{k\bar k} f_{k2 }f_{\bar k1}   }
}
Combining  \eqref{eqn:first step 3} and \eqref{eqn:first step 2}, we get
\eqsp{ \label{eqn:first step 4}
\Delta \abs{\nabla _1 \nabla_2 f} & \ge - 4 \Re \innpro{ \nabla f, \bar \nabla  \abs{\nabla_1\nabla_2 f}} +\sum_k ( f_{12k}f^{12k} + f_{12\bar k}f^{12\bar k})\\
&\quad  - 2 \sum_k\bk{|\nabla_1\nabla _2 f| |\nabla_1\nabla_k f| |\nabla_2\nabla_{\bar k} f| + |\nabla _1 \nabla_2 f| |\nabla_2\nabla_k f| |\nabla _1\nabla_{\bar k} f|    } .  }
On the other hand we have by Kato's inequality
\eqsp{ \label{eqn:first step 5} 
\Delta\abs{\nabla_1\nabla _2 f} =& 2 |\nabla_1\nabla_2 f| \Delta |\nabla_1 \nabla_2 f| + 2 \big| \nabla |\nabla_1 \nabla_2 f|   \big|^2 \\
\le & 2|\nabla_1 \nabla_2 f| \; \Delta |\nabla_1 \nabla_2 f| + \sum_k \abs{\nabla_k \nabla_1\nabla_2 f} + \abs{\nabla_{\bar k} \nabla_1 \nabla_2 f} \\
= & 2|\nabla_1 \nabla_2 f| \; \Delta |\nabla_1 \nabla_2 f| + \sum_k f_{12k}f^{12k} + f_{12\bar k}f^{12\bar k} .}
Combining \eqref{eqn:first step 4} and \eqref{eqn:first step 5} it follows that
\eqsp{ \label{eqn:first step 6}
\Delta |\nabla_1\nabla_2 f| \ge& - 2 \Re \innpro{\nabla f,\bar \nabla |\nabla_1 \nabla_2 f|} - \sum_k\bk{ |\nabla_1\nabla_k f| |\nabla_2\nabla_{\bar k} f| + |\nabla_2\nabla_k f| |\nabla _1\nabla_{\bar k} f|    }.    }
Combining  \eqref{eqn:useful 2}, \eqref{eqn:first step 6} and applying Cauchy-Schwarz inequality, we have
\eqsp{&\Delta\xk{|\nabla_1\nabla_2 f| + 2(-\Delta_1 f - \Delta_2 f)}\\ \ge& -2 \Re\innpro{\nabla f, \bar\nabla\xk{ |\nabla_1\nabla_2 f| + 2(-\Delta_1 f - \Delta_2 f)  }}\\
&\quad + \sum_{k=1}^n \xk{ \abs{\nabla_1\nabla_k f} + \abs{\nabla_1\nabla_{\bar k} f} + \abs{\nabla_2\nabla_k f} + \abs{\nabla_2\nabla_{\bar k} f}   }.
}
Note that the sum on the RHS of \eqref{eqn:first step 6} is (recall under our notation $\abs{\nabla_1\nabla_{\bar 1} f} = (\Delta_1 f)^2$)
$$\ge \abs{\nabla_1\nabla_2 f} + \abs{-\Delta_1 f} + \abs{-\Delta_2 f}\ge \frac 1{12} \bk{ |\nabla_1 \nabla_2 f| + 2(-\Delta_1 f- \Delta_2 f)  }^2,$$
so we get the following equation
\eqsp{\label{eqn:useful 3}
& \Delta \xk{ |\nabla_1 \nabla_2 f| + 2(-\Delta_1 f - \Delta_2 f) }\\
\ge & -2 \Re\innpro{\nabla f, \bar \nabla \xk{|\nabla_1 \nabla_2 f| + 2(-\Delta_1 f - \Delta_2 f)  }}\\
& + \frac 1 {12} \bk{|\nabla_1 \nabla_2 f| + 2(-\Delta_1 f - \Delta_2 f)}^2.
}
Denote $Q = \eta^2 \xk{|\nabla_1 \nabla_2 f| + 2(-\Delta_1 f - \Delta_2 f)}=: \eta^2 Q_1$, where $\eta(x) = \tilde \eta\xk{r(x)/R}$ and $\tilde \eta$ is a cut-off function such that $\tilde \eta|_{[0,1/3]} = 1$ and $\tilde \eta|_{[1/2,1]} = 0$. The following arguments are similar to the previous two cases. We calculate
\eqsp{ \label{eqn:useful 4}
\Delta  Q & =  \eta^2 \Delta Q_1 + 2\Re \innpro{\nabla \eta^2, \nabla Q_1} + Q_1 \Delta \eta^2\\
&\ge -2 \eta^2 \Re\innpro{\nabla f, \bar \nabla Q_1} + 2\Re \innpro{\nabla \eta^2 ,\nabla Q_1}  + \eta^2\frac{Q_1^2}{12}   + Q_1 \Delta \eta^2.
  }
Apply maximum principle to $Q$ and if the $\max Q\le 0$, we are done. So we may assume that $\max Q>0$ and is attained at $p_{\max}$, thus at $p_{\max}$, $Q_1>0$, $\Delta Q\le 0$, $\nabla Q_1 = - 2\eta^{-1} Q_1 \nabla \eta$ and 
$$Q_1 \Delta\eta^2 \ge Q_1 \frac{2}{R^2} \bk{ (2n-1)\eta\eta' + \eta\eta'' + (\eta')^2  }.$$
So at $p_{\max}$ it holds that 
\eqsp{\label{eqn:middle useful}
0\ge & \Delta Q\\
\ge & 4 \eta Q_1 \Re\innpro{ \nabla f, \bar \nabla \eta } - 8 Q_1 \abs{\nabla \eta} + \eta^2\frac{Q_1^2 }{12} + Q_1 \frac{2}{R^2} \bk{ (2n-1)\eta\eta' + \eta\eta'' + (\eta')^2  }\\
= & \frac{Q^2}{12\eta^2} + 4Q \eta^{-1} \Re \innpro{ \nabla f, \bar \nabla \eta  } - \frac{8Q}{\eta^2} \frac{(\eta')^2}{R^2}  + \frac{2Q}{R^2\eta^2} \xk{ (2n-1)\eta\eta' + \eta\eta'' + (\eta')^2  }\\
\ge & \frac{1}{\eta^2}\bk{ \frac{Q^2}{12} - \frac{40 |\nabla f| }{R} Q - \frac{800}{R^2} Q -  \frac{100n}{R^2} Q }
  }
where we choose $\eta$ such that $|\eta'|, |\eta''|\le 10$, for example. Therefore at $p_{\max}\in B(p,R/2)$ we have
$$\frac{Q^2}{12} - Q \bk{ \frac{40 |\nabla f|}{ R}  + \frac{800}{R^2}  + \frac{100n}{R^2}}\le 0\quad \Rightarrow \quad Q(p_{\max})\le \frac{C(n)}{R^2},$$
since $\sup_{B(p,R/2)}|\nabla f|\le C(n)R^{-1}$ from the previous estimates. Then for any $x\in B(p,R/3)$ we have 
\begin{equation*}{ Q_1(x) = \eta^2(x) Q_1(x)=Q(x) \le Q(p_{\max})\le \frac{C(n)}{R^2} .  }\end{equation*}
Thus it follows that
\begin{equation*}
|\nabla_1 \nabla_2 f|(x)\le Q_1(x) + 2 \big(\Delta_1 f (x) + \Delta_2 f(x)\big)\le \frac{C(n)}{R^2} + 2 \xk{\Delta_1 f (x) + \Delta_2 f(x)  }.
\end{equation*}
On the other hand from $|\nabla_1 \nabla_2 f| = | \frac{\nabla_1 \nabla_2 u}{u} - \frac{\nabla_1 u}{u}\frac{\nabla_2 u}{u}  |$
\eqsp{ \label{eqn:final estimate 1}
|\nabla_1\nabla_2 u|(x) & \le | \nabla_1 \nabla_2 f(x) | u(x)  + u(x) \frac{|\nabla_1 u(x)|}{u} \frac{|\nabla_2 u(x)|}{u}\\
&\le C(n)\frac{u(x)}{R^2} + 2 \Delta_1 u(x) + 2 \Delta_2 u(x) + u(x) \frac{|\nabla_1 u(x)|}{u} \frac{|\nabla_2 u(x)|}{u}\\
& \le C(n) \frac{\mathrm{osc}_{B(p,R)} u}{R^2}.
  }
Therefore we obtain that
\begin{equation}\label{final estimate 2}
\sup_{B(p,R/3)} |\nabla_1 \nabla_2 u|\le C(n)\frac{\mathrm{osc}_{B(p,R)} u}{R^2}.
\end{equation}
By exactly the same argument we can also get similar estimates for $|\nabla_1 \nabla_{\bar 2 } u|$  and $|\nabla_1 \nabla_{k} u| + |\nabla_1 \nabla _{\bar k} u|$ for $k\neq 1$.

Hence we have proved the following lemma:
\begin{lemma}\label{lemma 2.2}
There exists a constant $C(n)>0$ such that  for the solution $u_\epsilon$ to the equation \eqref{eqn:app Diri}
\begin{equation*}
\sup_{B_{g_\ep}(0,R/2)}\bk{ |\nabla_i \nabla_j u_\epsilon|_{g_\epsilon} + |\nabla_i \nabla _{\bar j} u_\epsilon|_{g_\epsilon}   } \le C(n)\frac{\osc_{B_{g_\epsilon}(0,R)} u_\epsilon}{R^2},
\end{equation*}
for all $i,\,j = 1,2,\cdots,n$.

\end{lemma}
\subsubsection{Convergence of $u_\epsilon$} \label{subsection end}In this section, we will show the Dirichlet problem \eqref{eqn:harmonic} admits a unique solution for any $\varphi\in C^0(\partial B_\bb(0,1))$.  We will write $B_\bb = B_\bb(0,1)$ for notation simplicity in this subsection.

\begin{prop}\label{prop:existence}
The Dirichlet boundary value problem \eqref{eqn:harmonic} admits a unique solution $u\in C^2(B_{\bbeta}\backslash \sS)\cap C^0(\overline{B_{\bbeta}})$ for any $\varphi\in C^0(\partial B_{\bbeta})$. Moreover, $u$ satisfies the estimates in Lemmas \ref{lemma:CY gradient}, \ref{lemma 2.1} and \ref{lemma 2.2} with $u_\ep$ replaced by $u$ and the metric balls replaced by those under the metric $g_\bb$, which we will refer as ``derivatives estimates'' throughout this section.
\end{prop}
\begin{proof}
Given the estimates of $u_\epsilon$ as in lemmas \ref{lemma:CY gradient}, \ref{lemma 2.1} and \ref{lemma 2.2}, we can derive the uniform local $C^{2,\alpha}$ estimates of $u_\epsilon$ on any compact subsets of $B_{\bbeta}(0,1)\backslash \sS$.

The $C^0$ estimates of $u_\epsilon$ follow immediately from the maximum principle (see Lemma \ref{lemma:MP}).

Take any compact subsets $K\Subset K'\Subset B_{\bbeta}(0,1)$. By Lemma \ref{lemma 2.1}, we have
\begin{equation}\label{eqn:2.36}
\sup_{K'}\bk{ |z_1|^{1-\beta_1} \ba{ \frac{\partial u_\epsilon}{\partial z_1}  } +|z_2|^{1-\beta_2} \ba{ \frac{\partial u_\epsilon}{\partial z_2}  } + \ba{\frac{\partial u_\epsilon}{\partial s_j}  } }\le  C(n) \frac{\| u_\epsilon\|_\infty}{ d(K',\partial B_{\bbeta})},
\end{equation}
\begin{equation}\label{eqn:2.37}
\sup_{K'} \bk{|z_1|^{1-\beta_1} \ba{ \frac{\partial^2 u_\epsilon}{\partial s_k\partial z_1}  } + |z_2|^{1-\beta_2} \ba{ \frac{\partial^2u_\epsilon}{\partial s_k\partial z_2}  } + \ba{ \frac{\partial^2 u_\epsilon}{\partial s_k \partial s_j}  }  } \le C(n)\frac{\| u_\epsilon\|_{\infty}}{d(K',\partial B_{\bbeta})^2},
\end{equation}
and the third-order estimates
\begin{equation}\label{eqn:2.38}
\sup_{K'}\bk{ |z_1|^{1-\beta_1} \ba{ \frac{\partial^3 u_\epsilon}{\partial z_1 \partial s_k\partial s_l}  }  +  |z_2|^{1-\beta_2} \ba{ \frac{\partial^3 u_\epsilon}{\partial z_2 \partial s_k\partial s_l}  } + \ba{\frac{\partial^3 u_\epsilon}{\partial s_j\partial s_k\partial s_l}  }  }\le C(n)\frac{\| u_\epsilon\|_\infty}{d(K',\partial B_{\bbeta})^3}.
\end{equation}
Moreover, applying the gradient estimate to the  $\Delta_{g_\epsilon}$-harmonic function $\Delta_{\epsilon,1} u_\epsilon$, we get
\begin{equation}\label{eqn:2.39}
\sup_{K'}\bk{ |z_1|^{1-\beta_1} \ba{ \frac{\partial}{\partial z_1} \Delta_{\epsilon,1}u_\epsilon  } + |z_2|^{1-\beta_2} \ba{\frac{\partial}{\partial z_2} \Delta_{\epsilon,1}u_\epsilon  } + \ba{\frac{\partial }{\partial s_j} \Delta_{\epsilon,1} u_\epsilon  }  }\le C(n)\frac{\| u_\epsilon\|_\infty}{d(K',\partial B_{\bbeta})^3}.
\end{equation} 
From \eqref{eqn:2.36}, \eqref{eqn:2.37} and \eqref{eqn:2.38}, we see that the functions $u_\epsilon$ have uniform $C^3$ estimates in the ``tangential directions'' on any compact subset of $B_{\bbeta}(0,1)$. Moreover, for any fixed small constant $\delta>0$, let $T_\delta(\sS)$ be the tubular neighborhood of $\sS$. We consider the equation 
$$\Delta_\epsilon u_\epsilon = (|z_1|^2 + \epsilon) ^{1-\beta_1} \frac{\partial^2 u_\epsilon}{\partial z_1 \partial \bar z_1} +  (\abs{z_2} + \epsilon)^{1-\beta_2} \frac{\partial^2 u_\epsilon}{\partial z_2\partial \bar z_2} + \sum_{j=5}^{2n} \frac{\partial^2 u_\epsilon}{\partial s_j^2} = 0, \text{ on }K'\backslash T_{\delta/2}(\sS), $$
which is strictly elliptic (with ellipticity depending only on $\delta>0$). Hence by standard elliptic Schauder theory, we also have $C^{2,\alpha}$-estimates of $u_\epsilon$ in the ``transversal directions'' (i.e. normal to $\sS$) and the mixed directions, on the compact subset $K\backslash T_{\delta}(\sS)$. By taking $\delta\to 0$, $K\to B_{\bbeta}$, and a diagonal argument, up to a subsequence $u_\epsilon$ converge in $C^{2,\alpha}_{loc}(B_{\bbeta}\backslash \sS)$ to a function $u\in C^{2,\alpha}(B_{\bbeta}\backslash \Ss)$. Clearly, $u$ satisfies the equation $\Delta_\bb u = 0$ on $B_{\bbeta}\backslash \sS$, and the estimates  \eqref{eqn:2.36}, \eqref{eqn:2.37} and \eqref{eqn:2.38}  hold for $u$ outside $\sS$, which implies that $u$ can be continuously extended through $\sS$ and defines a continuous function in $B_{\bbeta}(0,1)$. It remains to check the boundary value of $u$.

\smallskip

\noindent {\bf Claim}: {$u = \varphi$ on $\partial B_{\bbeta}(0,1)$}

It remains to show the limit function $u$ of $u_\epsilon$ satisfies the boundary condition $u = \varphi$ on $\partial B_{\bbeta}(0,1)$, which will be proved by constructing suitable barriers as we did in \cite{GS}.

The metric ball $B_{\bbeta}(0,1)$ is given by
$$B_{\bbeta}(0,1) = \{z\in\mathbb C^n|~ d_\bb(0,z)^2 :=  |z_1|^{2\beta_1} + |z_2|^{2\beta_2} + \sum_{j=5}^{2n} s_j^2 < 1\}.$$
$B_{\bbeta}(0,1)\subset B_{\mathbb C^n}(0,1)$ and their boundaries only intersect at $\sS_1\cap \sS_2$, where $z_1 = z_2 = 0$. Fix any point $q\in \partial B_{\bbeta}(0,1)$ and we consider the cases when $q\in\sS_1\cap \sS_2$ or $q\not\in \sS_1\cap \sS_2$. 

\medskip

\noindent{\bf Case 1:} $q\in \sS_1\cap\sS_2$, i.e. $z_1(q) = z_2(q) = 0$. Consider the point $q' = - q\in\partial B_{\bbeta}(0,1)\cap\partial B_{\mathbb C^n}(0,1) $, and $q$ is the {\em unique} farthest  point to $q'$ on $\partial B_{\bbeta}(0,1)$ under the Euclidean distance, hence the function $\Psi_q(z) := d_{\mathbb C^n}(z, q')^2 - 4 $ satisfies $\Psi_q(q) = 0$ and $\Psi_q(z)<0$ for all $z\in\partial B_{\bbeta}(0, 1)\backslash \{q\}$.  By the continuity of $\varphi$ for any $\delta>0$, there is a small neighborhood $V$ of $q$ such that $\varphi(q) -\delta < \varphi(z)< \varphi(q)+\delta$ for all $z\in \partial B_{\bbeta}(0,1)\cap V$ and on $\partial B_{\bbeta}(0,1)\backslash V$, $\Psi_q$ is bounded above by a negative constant. Hence we can make $\varphi_q(z): = \varphi(q) - \delta  + A \Psi_q(z) < \varphi(z)$ for all $z\in\partial B_{\bbeta}(0,1)$ if $A$ is chosen large enough. The function $\varphi_q$ is $\Delta_{g_\epsilon}$-subharmonic hence by maximum principle we have $u_\epsilon(z) \ge \varphi_q(z)$ for all $z\in B_{\bbeta}(0,1)$. Letting $\epsilon\to 0$ we get $u(z)\ge \varphi_q(z)$. Then taking $z\to q$ it follows that $\liminf_{z\to q} u(z) \ge \varphi(q) - \delta$, but $\delta>0$ is arbitrary so we have $\liminf_{z\to q} u(z)\ge \varphi(q)$. \

By considering the barrier function $\varphi(q) +\delta - A \Psi_q(z)$ and similar argument it is not hard to see that $\limsup_{z\to q} u(z)\le \varphi(q)$, hence $\lim_{z\to q} u(z) = \varphi(q)$ and $u$ is continuous up to $q\in\partial B_{\bbeta}(0,1)$.

\medskip

\noindent{\bf Case 2:} $q\in \partial B_{\bbeta}(0,1)\backslash \sS_1\cap\sS_2$. We consider the case when $z_1(q)\neq 0$ and $z_2(q)\neq 0$. The boundary $\partial B_{\bbeta}(0,1)$ is smooth near $q$, hence satisfies the exterior sphere condition. We choose an exterior Euclidean ball $B_{\mathbb C^n}(\tilde q, r_q)$ which is tangential with $\partial B_{\bbeta}(0,1)$ (only) at $q$, i.e. under the Euclidean distance $q$ is the unique closest point to $\tilde q$ on $\partial B_{\bbeta}(0,1)$. So the function $G(z) = \frac{1}{|z - \tilde q|^{2n-2}} -\frac{1}{r_q^{2n-2}}$ satisfies $G(q) = 0$ and $G(z)<0$ for all $z\in \partial B_{\bbeta}(0,1)\backslash \{q\}$. We calculate 
\begin{align*}
\Delta_{g_\epsilon} G & = (|z_1|^2 +\ep)^{-\beta_1 + 1} \frac{\partial ^2 G}{\partial z_1 \partial \bar z_1} + (|z_2|^2 + \ep)^{-\beta_2 + 1} \frac{\partial ^2 G}{\partial z_2 \partial \bar z_2} + \sum_{k=3}^n \frac{\partial ^2G}{\partial z_k \partial \bar z_k}\\
& = ((|z_1|^2 +\ep)^{-\beta_1 + 1} - 1)  \frac{\partial ^2 G}{\partial z_1 \partial \bar z_1} +( (|z_2|^2 + \ep)^{-\beta_2 + 1} - 1) \frac{\partial ^2 G}{\partial z_2 \partial \bar z_2}\\
& = \sum_{k=1}^2(-n+1)  \frac{{(|z_k|^2 + \ep)^{-\beta_k + 1} - 1  }}{|z - \tilde q|^{2n}} \bk{ -\frac{n \abs{z_k  - \tilde q_k}}{\abs{z - \tilde q}} + 1   }\\
& \ge - C(q, r_q).
\end{align*} 
The function $\Psi_q (z) = A ( d_{\bbeta}(z,0)^2  - 1   ) + G(z)$ is $\Delta_{g_\epsilon}$-subharmonic for $A>>1$ and $\Psi_q(q) = 0$, $\Psi_q(z)<0$ for $\forall z\in \partial B_{\bbeta}(0,1)\backslash \{q\}$. We are in the same situation as the {\bf Case 1}, so by the same argument as above, we can show the continuity of $u$ at such boundary point $q$. 

In case $z_1(q) \neq 0$ but $z_2(q) = 0$. The boundary $\partial B_{\bbeta}(0,1)$ is not smooth at $q$ and we cannot apply the exterior sphere condition to construct the barrier. Instead we will use the geometry of the metric ball $B_{\bbeta}(0,1)$. Consider the standard cone metric $g_{\beta_1} = \beta_1^2\frac{dz_1 \otimes d\bar z_1}{|z_1|^{2(1-\beta_1)}} + \sum_{k=2}^n dz_k\otimes d\bar z_k$ with cone singularity only along $\sS _1 = \{z_1 = 0\}$. The metric ball $B_{\bbeta}(0,1)$ is strictly contained in $B_{g_{\beta_1}}(0,1)$, and their boundaries are tangential at the points with vanishing $z_2$-coordinate. Thus $q\in \partial B_{\bbeta}(0,1)\cap \partial B_{g_{\beta_1}}(0,1)$ and $\partial B_{g_{\beta_1}}(0,1)$ is smooth at $q$ so there exists an exterior sphere for $\partial B_{g_{\beta_1}}(0,1)$ at $q$.  We define similar function $G(z)$ as in the last paragraph, and by the strict inclusion of the metric balls $B_{\bbeta}(0,1)\subset B_{g_{\beta_1}}(0,1)$, it follows that $G(q) = 0$ and $G(z)<0$ for all $z\in\partial B_{\bbeta}(0,1)\backslash \{q\}$. The remaining argument is the same as before.  

\end{proof}

\begin{remark}\label{rem:existence}
For any constant $c\in\mathbb R$, the following Dirichlet boundary value problem $$
\Delta_{g_{\bbeta}} u = c, \text{ in }B_{\bbeta}(0, 1)\backslash \sS,\text{ and }
 u = \varphi, \text{ on }\partial B_{\bbeta}(0,1),
$$admits a solution $u\in C^2(B_{\bbeta}\backslash \sS)\cap C^0(\overline{B_{\bbeta}})$ for any given $\varphi\in C^0(\partial B_{\bbeta})$.  This follows from the solution $\tilde u$ of \eqref{eqn:harmonic} with the boundary value $\tilde\varphi = \varphi - \frac{c}{2(n-2)} \sum_{j=5}^{2n}s_j^2$. Then the function $u = \tilde u + \frac{c}{2(n-2)}\sum_j s_j^2$ solves the equation above.

\end{remark}

For later application, we prove the existence of solution for a more general RHS of the Laplace equation with the standard background metric. Note that this result is not needed in the proof of Theorem \ref{thm:main 1}.

\begin{prop}\label{prop:3.2 new}
For any given $\varphi\in C^0(\partial B_\bb(0,1))$ and $f\in C^{0,\alpha}_{\bb}(\overline{B_\bb(0,1)})$, the Dirichlet boundary value problem
\begin{equation}\label{eqn:elliptic Diri}
\left\{\begin{aligned}
&\Delta_{g_\bb} v = f,\text{ in }B_\bb(0,1)\backslash \sS,\\
& v = \varphi, \text{ on }\partial B_\bb(0,1)
\end{aligned}\right.
\end{equation}
admits a unique solution $v\in C^2(B_\bb(0,1)\backslash \sS)\cap C^0( \overline{B_\bb(0,1)}  )$.
\end{prop}
By Theorem \ref{thm:main 1}, the solution $v$ to \eqref{eqn:elliptic Diri} belongs to $C^{2,\alpha}_\bb ( B_\bb(0,1)  )\cap C^0(\overline{B_\bb(0,1)  })$.


\begin{proof}
The proof is similar to that of Proposition \ref{prop:existence}. As before let $g_\ep$ be the approximating metrics \eqref{eqn:para app} of $g_\bb$ which are smooth metrics on $B_\bb(0,1)$. By standard elliptic theory we can solve the equations
\eqsp{\label{eqn:app general 1}
\left\{\begin{aligned}\Delta_{g_\ep} v_\ep = f,\text{ in }B_\bb(0,1),\\
v_\ep = \varphi, \text{ on }\partial B_\bb(0,1).
\end{aligned}
\right.
   }
For any compact subset $K\Subset B_\bb(0,1)$ and small $\detla>0$, we have uniform $C^{2,\alpha'}$-bound of $v_\ep$ on $K\backslash T_\delta(\Ss)$ for some $\alpha'<\alpha$.  Thus $v_\ep$ converges in $C^{2,\alpha'}$-norm to a function $v$ on $K\backslash T_\delta( \sS)$, as $\ep\to 0$. By a standard diagonal argument, setting $K\to B_\bb(0,1)$ and $\delta\to 0$, we can achieve that 
$$v_\ep \xrightarrow{ C^{2,\alpha'}_{loc}( B_\bb(0,1)\backslash \sS   )   } v\in C^{2,\alpha'}_{loc}( B_\bb(0,1)\backslash \sS   ),\text{ as }\ep \to 0.$$ And clearly $v$ satisfies the equation  \eqref{eqn:elliptic Diri} in $B_\bb(0,1)\backslash \sS$. It only remains to show the boundary value of $v$ coincides with $\varphi$ and $v$ is globally continuous in $B_\bb(0,1)$.

\medskip

\noindent $\bullet$ $v\in C^0(B_\bb(0,1))$. It suffices to show $v$ is continuous at any $p\in \sS\cap B_\bb(0,1)$. Fix such a point $p$ and take $R_0>0$ small so that $B_{\mathbb C^n}(p, 10 R_0)\cap \partial B_\bb(0,1) = \emptyset$. We observe that $\frac 1 2 g_{\mathbb C^n}\le g_\ep\le g_\bb$, so for any $r\in (0,1/2)$
\begin{equation}\label{eqn:ball inclusions}B_{g_\bb}(p,r)\subset B_{g_\ep}(p, r)\subset B_{\mathbb C^n}(p,2r),\end{equation}
in particular, the balls $B_{g_\ep}(p, 5R_0)$ are also disjoint with $\partial B_\bb(0,1)$. 

Since $\ric(g_\ep)\ge 0$ we have the following Sobolev inequality (\cite{Li}): there exists a constant $C=C(n)>0$ such that for any $h\in C^1_0(B_{g_\ep}(p, r))$, it holds that
\begin{equation}\label{eqn:Sobolev 1}
\bk{\int_{B_{g_\ep}(p,r)}  h^{\frac{ 2n }{n-1}}  \omega_\ep^n    }^{\frac{n-1}{n}}\le C \bk{\frac{ r^{2n}  }{ \vol_{g_\ep}( B_{g_\ep}(p,r)  )    }    }^{1/n} \int_{B_{g_\ep}(p,r)} \abs{ \nabla h  }_{g_\ep} \omega_\ep^n.
\end{equation}
It can be checked by straightforward calculations that $\vol_{g_\ep}\xk{ B_{g_\ep}(p, 1)   }\ge c_0(n)>0$ for some constant $c_0$ independent of $\ep$. Then Bishop volume comparison yields that for any $r\in (0,1)$, $$C_1(n) r^{2n}\ge \vol_{g_\ep} \xk{B_{g_\ep}(p,r)   }\ge c_1(n) r^{2n}.$$
Thus the Sobolev inequality \eqref{eqn:Sobolev 1} is reduced to 
\begin{equation}\label{eqn:Sobolev}
\bk{\int_{B_{g_\ep}(p,r)}  h^{\frac{ 2n }{n-1}}  \omega_\ep^n    }^{\frac{n-1}{n}}\le C\int_{B_{g_\ep}(p,r)} \abs{ \nabla h  }_{g_\ep} \omega_\ep^n,\quad \forall ~ h\in C^1_0( B_{g_\ep}(p,r)  ).
\end{equation}
With \eqref{eqn:Sobolev} at hand, we can apply the same proof of the standard De Giorgi-Nash-Moser theory (see the proof of Corollary 4.18 of \cite{HL}) to derive the uniform H\"older continuity of $v_\ep$ at $p$, i.e. there exists a constant $C=C(n,\bb, R_0)>0$ such that $$\osc_{B_\bb(p,r)} v_\ep \le \osc_{B_{g_\ep}(p, r)  } v_\ep \le C r^{\alpha''},\quad \forall r\in (0,R_0)$$  for some $\alpha''=\alpha''(n,\bb,R_0)\in (0,1)$ where in the first inequality we use the relation \eqref{eqn:ball inclusions}. Letting $\ep\to 0$ we see the continuity of $v$ at $p$.


\medskip

\noindent $\bullet$ $v = \varphi$ on $\partial B_\bb(0,1)$. The proof is almost identical to  that of Proposition \ref{prop:existence}. For example, the function $\varphi_q(z) = \varphi(q) - \delta + A\Psi_q(z)$ defined in {\bf Case 1} in the proof of Proposition \ref{prop:existence} satisfies $\Delta_{g_\ep} \varphi_q(z)\ge \max_X f$ if $A>0$ is taken large enough. Then from $\Delta_{g_\ep}(\varphi_q - v_\ep)\ge 0$ in $B_\bb$ and $\varphi_q - \varphi\le 0$ on $\partial B_\bb$, applying maximum principle we  get $\varphi_q\le v_\ep$ in $B_\bb(0,1)$. The remaining are the same as in Proposition \ref{prop:existence}. The {\bf Case 2} can be dealt with similarly.

\end{proof}

\begin{remark}
Let $H_0^1(B_\bb(0,1), g_\bb)$ be the completion of the space $C^1_0(B_\bb(0,1))$-functions under the norm $$\| \nabla u\|_{L^2(g_\bb)} = \bk{ \int_{B_\bb(0,1)} |\nabla u|_{g_\bb}^2 \omega_\bb^n  }^{1/2}.$$ For any $h\in C^1_0( B_\bb(0,1)  )$, letting $\epsilon\to 0$ in \eqref{eqn:Sobolev} we get
\begin{equation}\label{eqn:SOB}
\bk{ \int_{B_\bb(p,r)} |h|^{\frac{2n}{n-1}} \omega_\bb^n }^{\frac{n-1}{n}}\le C \int_{B_\bb(p,r)} \abs{\nabla h}_{g_\bb} \omega_\bb^n,
\end{equation}
for the same constant $C$ in \eqref{eqn:Sobolev}. That is, Sobolev inequality also holds for the conical metric $\omega_\bb$.

\end{remark}

\subsection{Tangential and Laplacian estimates}\label{section 3.2}
In this section, we will prove the H\"older continuity of $\Delta_k u$ for $k = 1,2$  and $(D')^2u $ for the solution $u$ to \eqref{eqn:main equation}. The arguments of \cite{GS} can be adopted here. We recall that we assume $\beta_1,\beta_2\in (\frac 1 2 ,1)$.  We fix some notations first. 

For a given point $p\not\in \sS$, we denote $r_p = d_{g_{\bbeta}}(p, \sS)$, the $g_{\bbeta}$-distance of $p$ to the singular set $\sS$. For notation simplicity we will fix $\tau =  1/ 2$ and an integer $k_p\in\mathbb Z_+$ to be the smallest integer such that $\tau^k <  r_p$, and $k_{i,p}\in\mathbb Z_+$ the smallest integer $k$ such that $\tau^k < d_{\bbeta}(p,\sS_i)$, for $i=1,2$. So $k_p = \max\{k_{1,p}, k_{2,p}\}$   We denote $p_1\in \sS_1$ and $p_2\in \sS_2$ the projections of $p$ to $\sS_1$, $\sS_2$, respectively.

\smallskip

For $j = 1,2$ we will write $\Delta_j u := |z_j|^{2(1-\beta_j)} \frac{\partial^2 u}{\partial z_j \partial \bar z_j}$. 

\smallskip

We will consider a family of conical Laplace equations with different choices of $k\in \mathbb Z^+$. 
\begin{enumerate}[label=(\roman*)]
\item If $k\ge k_p$, the geodesic balls $B_{\bbeta}(p, \tau^k)$ are disjoint with $\sS$ and have smooth boundaries. $g_{\bbeta}$ is smooth on such balls. By standard theory we can solve the Dirichlet problem for $u_k\in C^\infty(B_{\bbeta}(p,\tau^k))\cap C^0(\overline{B_{\bbeta}(p,\tau^k)})$
\begin{equation}\label{eqn:uk 1}
\left\{\begin{aligned}
&\Delta_\bb u_k = f(p),&   \text{in }B_{\bbeta}(p,\tau^k)\\
& u_k = u, &  \text{on }\partial B_{\bbeta}(p,\tau^k)
\end{aligned}\right.
\end{equation}

\item Without loss of generality we assume $d_{\bbeta}(p,\sS_1) \le d_{\bbeta}(p,\sS_2)$, i.e. $k_{1,p}\ge k_{2,p}$. We now solve the Dirichlet problem $u_k\in C^2(B_{\bbeta}(p_1,2\tau^k)\backslash \sS_1)\cap C^0(\overline{B_{\bbeta}(p_1,2\tau^k)})$ for $k_{2,p}+2\le k< k_{1,p}$
\begin{equation}\label{eqn:uk 2}
\left\{\begin{aligned}
& \Delta_\bb u_k = f(p), & \text{in }B_{\bbeta}(p_1, 2 \tau^k)\\
& u_k = u, & \text{on }\partial B_{\bbeta}(p_1,2 \tau^k)
\end{aligned}\right.
\end{equation}
By similar argument as in the proof of Proposition \ref{prop:existence}, such $u_k$ exists. 

\item For $2\le k\le k_{2,p}+1$, let $u_k\in C^2( B_{\bbeta}(p_{1,2},2\tau^k)\backslash \sS  )\cap C^0(\overline{B_{\bbeta}(p_{1,2},2\tau^k)})$ solve the equation
\begin{equation}\label{eqn:uk 3}
\left\{\begin{aligned}
& \Delta_\bb u_k = f(p), & \text{in }B_{\bbeta}(p_{1,2}, 2 \tau^k)\\
& u_k = u, & \text{on }\partial B_{\bbeta}(p_{1,2},2 \tau^k)
\end{aligned}\right.
\end{equation}
whose existence follows from Remark \ref{rem:existence}. Here $p_{1,2} = (0; 0; s(p))\in\sS_1\cap \sS_2$  is the projection of $p_1$ to $\sS_2$.

\end{enumerate}
We remark that we may take $f(p) = 0$ by considering $\tilde u = u - \frac{f(p)}{2(n-2)} |s - s(p)|^2$. If the estimate holds for $\tilde u$, it also holds for $u$.  So from now on we assume $f(p) = 0$. 

\begin{lemma}\label{lemma 2.4}
Let $u_k$ be the solutions to the equations \eqref{eqn:uk 1}, \eqref{eqn:uk 2} and \eqref{eqn:uk 3}. There exists a constant $C= C(n)>0$ such that for all $k\in\mathbb Z_+$, the following estimates  hold
\begin{equation}\label{eqn:c0 1}
\| u_k - u\|_{L^\infty(\hat B_k (p) )} \le C(n ) \tau^{2k} \omega(\tau^k),
\end{equation}
where we denote $\hat B_k(p)$ by  \begin{equation}\label{eqn:hat Bk}
\hat B_k(p) :=\left\{\begin{aligned}
&B_{\bbeta}(p, \tau^k),  \text{ if }k \ge k_p\\
&B_{\bbeta}(p_1, 2\tau^k),  \text{ if } k_{2,p} + 2 \le k< k_{1,p}\\
&B_{\bbeta} (p_{1,2}, 2 \tau^k),  \text{ if } 1\le k\le k_{2,p}+1,
\end{aligned}\right.
\end{equation}in different choices of $k\in\mathbb Z_+$. 
\end{lemma}
We will also denote $\lambda \hat B_k(p)$ to the concentric ball with $\hat B_k(p)$ but the radius scaled by $\lambda\in (0,1)$.

This lemma follows straightforwardly  from  Lemma \ref{lemma:MP} and the definition of $\omega(r)$. So we omit the proof. By triangle inequality, we get the following estimates
\begin{equation}\label{eqn:uk comp}
\| u_k - u_{k+1}\|_{L^\infty (\frac  12  \hat B_k)} \le C(n) \tau^{2k} \omega(\tau^k),
\end{equation}

 Since $u_k - u_{k+1}$ are $g_{\bbeta}$-harmonic functions on $\frac 12 \hat B_k$, applying the gradient and Laplacian estimates \eqref{eqn:grad final prop} and \eqref{eqn:lap final prop} for harmonic functions, we get:
\begin{lemma}\label{lemma 2.5}
There exists a constant $C(n)>0$ such that  for all $k\in\mathbb Z_+$ it holds that
\begin{equation}\label{eqn:uk grad}
\| D'u_k - D'u_{k+1}\|_{L^\infty (\frac 1 3 \hat B_k)}\le C(n)\tau^k \omega(\tau^k),
\end{equation}and 
\begin{equation}\label{eqn:uk 2nd}
\sup_{\frac{1}{3} \hat B_k\backslash \sS}\bk{\sum_{i=1}^2 \ba{\Delta_i ( u_k - u_{k+1}   )} +\ba{ (D')^2 u_k - (D')^2 u_{k+1}   }  }  \le C(n) \omega(\tau^k),
\end{equation}
where we recall that $D'$ denotes the first order operators $\frac{\partial}{\partial {s_i}}$ for $i=5,\ldots, 2n$.
\end{lemma}
The following lemma can be proved by looking at the Taylor expansion of $u_k$ at $p$ for $k>>1$ as in Lemma 2.8 of \cite{GS}.
\begin{lemma}\label{lemma 2.6}
The following limits hold:
\begin{equation}\label{eqn:uk limit}
\lim_{k\to \infty} D' u_k(p) = D'u(p),\, \lim_{k\to \infty} (D')^2 u_k(p) = (D')^2 u (p), \, \lim_{k\to \infty} \Delta_i u_k(p) = \Delta_i u(p),
\end{equation}
where $i = 1,\, 2$.
\end{lemma}
Combining Lemmas \ref{lemma 2.5} and \ref{lemma 2.6}, we obtain the following estimates on the $2$nd-order  (tangential) derivatives.
\begin{prop}
There exists a constant $C=C(n, \beta)>0$ such that
\begin{equation}\label{eqn:2nd order}
\sup_{B_{\bbeta}(0,1/2)\backslash \sS}| (D')^2 u| + |\Delta_ i u| \le C\Big( \| u\|_{L^\infty(B_{\bbeta}(0,1))} + \int_0^1 \frac{\omega(r)}{r}dr + |f(0)|   \Big).
\end{equation}
\end{prop} 
\begin{proof}
From triangle inequality we have for any given $z\in B_{\bbeta}(0,1/2)\backslash \sS$
\begin{align*}
|(D')^2 u (z)| & \le \sum_{k=0}^\infty |(D')^2 u_k(z) - (D')^2 u_{k+1}(z)| + |(D')^2 u_1(z)|\\
&\le C(n)\sum_{k=1}^\infty  \omega(\tau^k) + C(n) \osc_{B_{\bbeta}(0,1)} u_0\\
&\le C(n,\bbeta)\Big(\| u\|_{L^\infty} + \int_0^1 \frac{\omega(r)}{r}dr + |f(0)|  \Big).
\end{align*}
The estimates for $\Delta_i u$ can be proved similarly.

\end{proof}

For any other given point $q\in B_{\bbeta}(0,1/2)\backslash \sS$, we can solve Dirichlet boundary problems as $u_k$ with the metric balls centered at $q$, and we obtain a family of functions $v_k$ such that 
\begin{equation}\label{eqn:vk 1}
\Delta_\bb v_k = f(q),\quad \text{in }\tilde B_k(q), \quad v_k = u\text{ on }\partial \tilde B_k(q),
\end{equation}
where $\tilde B_k(q)$ are metrics balls centered at $q$ given by
\begin{equation*}
\tilde B_k(q) = \tilde B_k: = \left\{\begin{aligned}
&B_{\bbeta}(q, \tau^k),\text{ if }k\ge k_q,\\
&B_{\bbeta}(q_i, 2\tau^k),\text{ if } k_{j, q} +2 \le k< k_q, \text{ here $k_{i,q} = \max(k_{1,q}, k_{2,q}  )$ and $j\neq i$}\\
&B_{\bbeta}(q_{i,j}, 2\tau^k),\text{ if } k\le k_{j,q} + 1.
\end{aligned} \right.
\end{equation*}
Similar estimates as in Lemmas \ref{lemma 2.4}, \ref{lemma 2.5} and \ref{lemma 2.5} also hold for $v_k$ within the balls $\tilde B_k(q)$.

We are now ready to state the main result in this subsection on the continuity of second order derivatives. 
\begin{prop}\label{prop 2.2}
Let $d = d_{\bbeta}(p,q)$. There exists a constant $C=C(n)>0$ such that if $u$ solves the conical Laplace equation \eqref{eqn:main equation}, then the following holds for $i=1,2$:
$$ | \Delta_i u(p) - \Delta_i u(q)   | +    | (D')^2 u(p) - (D')^2 u(q)   | \le C \Big( d \| u\|_{L^\infty(B_{\bbeta}(0,1))} + \int_0^d \frac{\omega(r)}{r}dr + d \int_d^1 \frac{\omega(r)}{r^2 }dr   \Big).   $$
\end{prop}
\begin{proof}We only prove the estimate for $\dsq u$, and the one for $\Delta_i u$ can be dealt with in the same way.

We may assume $r_p = \min(r_p,r_q)$. We fix an integer $\ell$ such that $\tau^\ell$ is comparable to $d$, more precisely, we take $$\tau^{\ell + 4} \le d < \tau^{\ell +3},\quad \text{or} \quad \tau^{\ell + 1} \le 8 d\le \tau^\ell.  $$
We calculate by triangle inequality
\begin{align*}
|(D')^2 u(p) - (D')^2u (q)  |\le & | (D')^2 u(p) - (D')^2 u_\ell (p)  | + |(D')^2 u_\ell(p) - (D')^2 u_\ell(q)|\\
& + | (D')^2 u_\ell(q) -(D')^2 v_\ell (q)| + | (D')^2 v_\ell (q) - (D')^2 u(q) |\\
=:& I_1 + I_2 + I_3 + I_4.
\end{align*}
We will estimate $I_1 $ - $I_4$ one by one. 

\noindent $\bullet$ $I_1$ and $I_4$: by  \eqref{eqn:uk 2nd} and \eqref{eqn:uk limit} we have
$$I_1 = |(D')^2 u(p) - (D')^2 u_\ell (p)|\le C(n) \sum_{k=\ell}^\infty \omega(\tau^k),$$
and similar estimate holds for $I_4$ as well
$$I_4 = |(D')^2 u(q) - (D')^2 v_\ell (q)|\le C(n) \sum_{k=\ell}^\infty \omega(\tau^k).$$
\noindent $\bullet$ $I_3$: by the choice of $\ell$, it is not hard to see that $\frac 23 \tilde B_\ell(q) \subset \hat B_\ell(p)$. In particular $u_\ell$ and $v_\ell$ are both defined on $\frac 2 3 \tilde B_\ell(q)$ and satisfy the equations
$$\Delta_\bb u_\ell = f(p),\quad \Delta_\bb v_\ell = f(q)$$
respectively on this ball. From \eqref{eqn:c0 1} for $u_\ell$ and similar estimate for $v_\ell$ we get 
$$\| u_\ell - v_\ell \| _{L^\infty( \frac 2 3 \tilde B_\ell(q)  )} \le C \tau^{2\ell} \omega(\tau^\ell).$$
Consider the function \begin{equation}\label{eqn:defn of U}U := u_\ell - v_\ell - \frac{f(p) - f(q)}{2(n-2)} |s - s(\tilde q)|^2\end{equation}
where $\tilde q$ is the center of the ball $\tilde B_\ell(q)$.  $U$ is $g_{\bbeta}$-harmonic in $\frac 23 \tilde B_\ell(q)$ and satisfies the estimate:
$$\| U\|_{L^\infty(\frac 23 \tilde B_\ell(q))} \le C \tau^{2\ell}\omega(\tau^\ell) + C \tau^{2\ell} \omega(d)\le C(n) \tau^{2 \ell} \omega(\tau^\ell).$$
The derivatives estimates imply that $$|(D')^2 U(q)|\le C \tau^{-2\ell}\| U\|_{L^\infty(\frac 23 \tilde B_\ell(q))} \le C(n) \omega(\tau^\ell). $$ Hence
$$I_3 = |(D')^2 u_\ell (q) - (D')^2 v_\ell(q)| \le C(n)\omega(\tau^\ell).$$
\noindent $\bullet$ $I_2$: this is a little more complicated than the previous estimates. We define $h_k = u_{k-1} - u_k$ for $k\le \ell$. $h_k$ is $g_{\bbeta}$-harmonic on $\hat B_k(p)$ and by \eqref{eqn:c0 1} $h_k$ satisfies the $L^\infty$-estimate $\| h_k\|_{\hat B_k(p)}\le C \tau^{2k} \omega(\tau^k)$ and the derivative estimates $\| (D')^2 h_k\|_{L^\infty(\frac 23 \hat B_k(p))}\le C \omega(\tau^k)$. On the other hand, the function $(D')^2 h_k$ is also $g_{\bbeta}$-harmonic on $\frac 23 \hat B_k(p)$ so the gradient estimate implies that
\begin{equation}\label{eqn:grad hk}\| \nabla_{g_{\bbeta}} (D')^2 h_k  \|_{L^\infty( \frac 1 2 \hat B_k(p)\backslash \sS  )} \le C \tau^{-k} \omega(\tau^k).\end{equation}
Integrating this along the minimal $g_{\bbeta}$-geodesic $\gamma$ connecting $p$ and $q$ and noting that $\gamma$ avoids $\sS$ since $(\mathbb C^n\backslash \sS,g_{\bbeta})$ is strictly geodesically convex, we get
$$| (D')^2 h_k(p) - \dsq h_k(q)   |\le d\cdot \| \nabla_{g_{\bbeta}} (D')^2 h_k  \|_{L^\infty( \frac 1 2 \hat B_k(p)\backslash \sS  )} \le d C \tau^{-k}\omega(\tau^k).$$ By triangle inequality for each $k\le \ell$
\begin{equation}\label{eqn:I2 1} I_2= | \dsq u_\ell (p) - \dsq u_\ell(q)  |\le |\dsq u_2(p) - \dsq u_2(q)  | + d C\sum_{k=2}^\ell \tau^{-k}\omega(\tau^k). \end{equation}
Observe that $p,q\in \hat B_2(p)$ and the function $\dsq u_2$ is $g_{\bbeta}$-harmonic on $\hat B_2(p)$. From \eqref{eqn:c0 1} and derivative estimates we have
$$\| \dsq u_2 \|_{L^\infty(\frac 23 \hat B_2(p))} \le C \| u_2\|_{L^\infty(\hat B_2(p))}\le C ( \| u\|_{L^\infty} + \omega(\tau^2)   ).$$ Again by gradient estimate we have $$\| \nabla_{g_{\bbeta}} \dsq u_2\|_{L^\infty (\frac 1 2 \hat B_2(p))}\le C ( \| u\|_{L^\infty} + \omega(\tau^2)   ).$$ Integrating along the minimal geodesic $\gamma$ we arrive at
$$|\dsq u_2(p) - \dsq u_2 (q)  |\le d C ( \| u\|_{L^\infty} + \omega(\tau^2)   ). $$
Combining \eqref{eqn:I2 1}, we obtain that 
$$I_2\le C d \bk{ \| u\|_{L^\infty(B_{\bbeta}(0,1))} + \sum_{k=2}^\ell \tau^{-k}\omega(\tau^k)   }.$$
Combing this with the estimates for $I_1, I_2, I_3, I_4$, we get
\begin{equation*}
 | (D')^2 u(p) - (D')^2 u(q) | \le  C \bk{ d \xk{ \| u\|_{L^\infty(B_{\bbeta}(0,1))} + \sum_{k=2}^\ell \tau^{-k}\omega(\tau^k)   }   + \sum_{k=\ell}^\infty \omega(\tau^k)   }.
\end{equation*}
Proposition \ref{prop 2.2} now follows from this and the fact that $\omega(r)$ is monotonically increasing.
\end{proof}

\subsection{Mixed normal-tangential estimates along the directions $\sS$}\label{section 3.3}
Throughout this section, we fix two points $p,q\in B_{\bbeta}(0,1/2)\backslash \sS$ and assume $r_p \le r_q$. Recall that we introduce the weighted ``polar coordinates'' $(r_i, \theta_i)$ for $(z_1,z_2)$ as 
$$\rho_i = |z_i|, \, r_i  = \rho_i^{\beta_i},\, \theta_i = \arg z_i,\quad i = 1, 2.$$
Under these coordinates it holds that
\begin{equation}\label{eqn:cone Lap}
\Delta_i u = |z_i|^{2(1-\beta_i)} \frac{\partial^2 u}{\partial z_i \partial \bar z_i} = \frac{\partial^2 u}{\partial r_i^2} + \frac{1}{r_i } \frac{\partial u}{\partial r_i } + \frac{1}{\beta_i^2 r_i ^2 }\frac{\partial^2 u}{\partial \theta_i ^2}.
\end{equation}
Let $u_k$ (resp. $v_k$) be the solutions to equations \eqref{eqn:uk 1}, \eqref{eqn:uk 2} and \eqref{eqn:uk 3} on $\hat B_k(p)$ (resp. $\tilde B_k(q)$). Recall $u_k - u_{k+1}$ satisfies \eqref{eqn:uk comp} and apply gradient estimates to the $g_{\bbeta}$-harmonic function $u_k- u_{k+1}$, we get the bound of $\| \nabla_{g_{\bbeta}} (u_k - u_{k+1})\|_{L^\infty(\frac 13 \hat B_k(p))}$, which in particular implies that for $i = 1 $ or $2$
\begin{equation}\label{eqn:normal grad}
\big\| |z_i|^{1-\beta_i}( \frac{\partial u_k}{\partial z_i} - \frac{\partial u_{k+1}}{\partial z_i}   )\big\|_{L^\infty\xk{ \frac 1 3 \hat B_k(p) }} \le C \tau^k \omega(\tau^k).
\end{equation} 
Similarly $D'u_k - D'u_{k+1}$ is also $g_{\bbeta}$-harmonic on $\frac 1 2 \hat B_k(p)$ and apply gradient estimates to this function we get for $i = 1,2$
\begin{equation}\label{eqn:normal grad 1}
\big\| |z_i|^{1- \beta_i} ( \frac{\partial D'u_k}{\partial z_i} - \frac{\partial D'u_{k+1}}{\partial z_i}  )\big\|_{L^\infty\xk{\frac 1 3 \hat B_k(p)  )}}\le C \omega(\tau^k).
\end{equation}
The following lemma can be proved by the same way as in Lemma 2.10 of \cite{GS} since $p\not\in\sS$, so we omit the proof.
\begin{lemma}\label{lemma 2.7}
The following limits hold: for $i = 1$ or $2$
\begin{equation*}
\lim_{k\to \infty}\frac{\partial u_k}{\partial r_i}(p) =\frac{\partial u}{\partial r_i}(p),\quad \lim_{k\to\infty}\frac{\partial u_k}{r_i \partial \theta_i} (p) = \frac{\partial u}{r_i \partial \theta_i}(p)
\end{equation*}
and 
\begin{equation}\label{eqn:limit for J}
\lim_{k\to\infty}\frac{\partial D'u_k}{\partial r_i} (p) = \frac{\partial D'u}{\partial r_i}(p),\quad \lim_{k\to\infty}\frac{\partial D'u_k}{r_i \partial \theta_i}(p) = \frac{\partial D' u}{ r_i \partial\theta_i}(p).
\end{equation}
Similar formulas also hold for $v_k$ at the point $q$.
\end{lemma}
We are going to estimate the quantities
$$J: =\big |\frac{\partial D' u}{\partial r_i}(p) -  \frac{\partial D' u}{\partial r_i}(q) \big |, \text{ and } K: = \big|\frac{\partial D' u }{r_i \partial \theta_i} (p) - \frac{\partial D' u }{r_i \partial \theta_i} (q)\big|,\quad i = 1,\, 2.  $$ Note that $J,K$ correspond to $|N_j D' u(p) - N_j D'u(q)|$ in Theorem \ref{thm:main 1}. 
We will estimate the case for $i=1$ and $J$, since the other cases are completely the same. By triangle inequality we have
\begin{align*}
J\le &  \big |\frac{\partial D' u}{\partial r_i}(p) -  \frac{\partial D' u_\ell}{\partial r_i}(p) \big | + \big |\frac{\partial D' u_\ell}{\partial r_i}(p) -  \frac{\partial D' u_\ell}{\partial r_i}(q) \big |\\
& + \big |\frac{\partial D' u_\ell}{\partial r_i}(q) -  \frac{\partial D' v_\ell}{\partial r_i}(q) \big | + \big |\frac{\partial D' v_\ell}{\partial r_i}(q) -  \frac{\partial D' u}{\partial r_i}(q) \big |\\
 =: & J_1 + J_2 + J_3+J_4.
\end{align*}
\begin{lemma}
There exists a constant $C(n)>0$ such that $J_1$, $J_3$ and $J_4$ satisfy
\begin{equation*}
J_1 + J_4\le C \sum_{k=\ell}^\infty \omega(\tau^k), \quad J_3 \le C \omega(\tau^\ell).
\end{equation*}
\end{lemma}
\begin{proof}
The estimates for $J_1$ and $J_4$ can be proved similarly as in proving those of $I_1$ and $I_4$ as in Section \ref{section 3.2}, using \eqref{eqn:normal grad 1} and \eqref{eqn:limit for J}. $J_3$ can be estimated similar to that of $I_3$ as in Section \ref{section 3.2}, using \eqref{eqn:normal grad 1}. So we omit the details.
\end{proof}
To estimate $J_2$, as in Section \ref{section 3.2} we denote $h_k : = u_{k-1} - u_k$ for $2\le k\le \ell$ which is $g_{\bbeta}$-harmonic on $\hat B_k(p)$ and satisfies the $L^\infty$-estimate $\| h_k\|_{L^\infty(\hat B_{k}(p))}\le C \tau^{2k} \omega(\tau^k)$ by \eqref{eqn:normal grad 1}. We  rewrite \eqref{eqn:grad hk} as
\begin{equation}\label{eqn:grad hk 1}
\| (D')^3 h_k \|_{L^\infty\xk{ \frac 12 \hat B_k(p) \backslash \sS }} +\sum_{i=1}^2 \big\| |z_i|^{1-\beta_i}\frac{\partial }{\partial z_i} (D')^2 h_k\big\|_{L^\infty\xk{ \frac 1 2 \hat B_k(p) \backslash \sS }} \le C \tau^{-k}\omega(\tau^k).
\end{equation}
\begin{lemma}\label{lemma 2.9}
There exists a constant $C=C(n,\bbeta)>0$ such that for any $z\in \frac 1 4 \hat B_k(p)\backslash \sS$, the following pointwise estimate holds for all $k\le \min(\ell, k_p)$
\begin{equation*}
\big| \frac{\partial D' h_k}{\partial r_1}(z) \big| + \big| \frac{\partial D'h_k}{ r_1 \partial \theta_1} (z) \big| \le C r_1(z) ^{\frac{1}{\beta_1} - 1} \tau^{- k (\frac 1{\beta_1} - 1)} \omega(\tau^k).
\end{equation*}
\end{lemma}
\begin{proof}We define a function $F$ as
\begin{equation}\label{eqn:F1}
|z_1|^{2(1-\beta_1)} \frac{\partial^2 D'h_k}{\partial z_1 \partial\bar z_1} = - |z_2|^{2(1-\beta_2)} \frac{\partial^2 D'h_k}{\partial z_2 \partial \bar z_2} - \sum_{j=5}^{2n} \frac{\partial^2 D'h_k}{\partial s_j^2}=: F.
\end{equation}
The Laplacian estimates \eqref{eqn:lap final prop} and derivative estimates applied to the $g_{\bbeta}$-harmonic function $D'h_k$ implies that $F$ satisfies \begin{equation}\label{eqn:F}\| F\|_{L^\infty\xk{\frac 1 2 \hat B_k(p)}}\le C(n) \tau^{-k}\omega(\tau^k).\end{equation}

\noindent For any $k\le \min(\ell, k_p)$ and  $x\in \sS_1 \cap \frac 1 4 \hat B_k(p)$,   $B_{\bbeta}(x, \tau^k)\subset \frac 1 3 \hat B_k (p)$. The intersection of $B_{\bbeta}(x,\tau^k)$ with complex plane $\mathbb C$ passing through $x$ and orthogonal to the hyperplane $\sS_1$ lies in a metric ball of radius $\tau^{k}$ under the standard cone metric $\hat g_{\beta_1}$ on $\mathbb C$.  We view the equation \eqref{eqn:F1} as defined  on the ball $\hat B: = B_{\mathbb C} (x, (\tau^k)^{1/\beta_1})\subset \mathbb C$. The estimate \eqref{eqn:2.5} applied to the function $D'h_k$ gives rise to 
\begin{equation*}
\sup_{B_{\mathbb C}(x,(\tau^k)^{1/\beta_1}/2)\backslash\{x\}} \big| \frac{\partial D'h_k}{\partial z_1}   \big| \le C \frac{\| D'h_k\|_{L^\infty(\hat B)}}{(\tau^k)^{1/\beta_1}} + C \| F\|_{L^\infty(\hat B)} (\tau^k)^{2- \frac{1}{\beta}}.
\end{equation*}
Therefore on $B_{\mathbb C}(x, (\tau^k)^{1/\beta_1}/2)\backslash\{x\}$ the following holds
\begin{equation}\label{eqn:my 1}
\begin{split}
&\big| \frac{\partial D'h_k}{\partial r_1}  \big| + \big|\frac{\partial D'h_k}{r_1 \partial \theta_1}  \big|
\le    r_1^{\frac{1}{\beta_1} - 1}  \big| \frac{\partial D'h_k}{\partial z_1}   \big|
\le  C r_1 ^{\frac 1{\beta_1} - 1} \tau^{k(1-\frac{1}{\beta_1})} \omega(\tau^k)
\end{split}\end{equation}
On other hand,  since $B_{\mathbb C} (x, (\tau^k)^{1/\beta_1}/2) = B_{\hat g_{\beta_1}} (x, 2^{-\beta_1} \tau^{k})$
\begin{equation}\label{eqn:inclusion}
\frac 1 4 \hat B_k(p) \subset \cup_{x\in \sS_1 \cap \frac{1}{4} \hat B_k} B_{\mathbb C }(x, (\tau^k)^{1/\beta_1}/2).  \end{equation}
equation \eqref{eqn:my 1} implies the desired estimate on the balls $\frac 1 4 \hat B_k(p)$.



\end{proof}

\begin{remark}
By similar arguments we can get the following estimates as well for any $k\le \min(\ell,k_p)$ and $z\in \frac 1 4 \hat B_k(p)\ms_1$
\begin{equation}\label{eqn:our 2}
\big| \frac{\partial (D')^2 h_k}{\partial r_1}(z) \big| + \big| \frac{\partial (D')^2h_k}{ r_1 \partial \theta_1} (z) \big| \le C r_1(z) ^{\frac{1}{\beta_1} - 1} \tau^{-  \frac k{\beta_1} } \omega(\tau^k).
\end{equation}

\end{remark}

\begin{lemma}\label{lemma 2.10}
There exists a constant $C=C(n,\bbeta)>0$ such that for all $k\le \min(k_p, \ell)$ and $z\in \frac 1 4 \hat B_k(p)\backslash\sS$ the following pointwise estimates hold
\begin{equation}\label{eqn:my 2}
\big| \frac{\partial^2 D'h_k}{r_1 ^2 \partial \theta_1 ^2} (z)  \big| + \big|\frac{\partial^2 D'h_k}{r_1 \partial r_1 \partial \theta_1} (z)  \big| \le C r_1(z)^{\frac{1}{\beta_1} - 2} \tau^{-k (\frac{1}{\beta_1} - 1)} \omega(\tau^k),
\end{equation}
\begin{equation}\label{eqn:my 3}
\big| \frac{\partial^2 D'h_k}{\partial r_1^2 }(z)  \big| \le C r_1(z)^{\frac{1}{\beta_1} - 2} \tau^{-k (\frac 1{\beta_1} - 1)} \omega(\tau^k).
\end{equation}
\end{lemma}
\begin{proof}
Applying the gradient estimate to the $g_{\bbeta}$-harmonic function $D'h_k$, we get
$$\| \frac{\partial D'h_k}{r_1 \partial \theta_1}\|_{L^\infty(\frac 12 \hat B_k(p))}\le \| \nabla_{g_{\bbeta}} D'h_k\|_{L^\infty(\frac 12 \hat B_k(p))}\le C\omega(\tau^k). $$
The function $\partial_{\theta_1} D'h_k$ is also a continuous $g_{\bbeta}$-harmonic function so the derivatives estimates implies on $\frac 1 3 \hat B_k(p)\ms$
$$|F_1| \le \Big|  |z_2|^{2(1-\beta_2)} \frac{\partial^2 (\partial_{\theta_1} D'h_k)}{\partial z_2 \partial \bar z_2}\Big| + \Big| \frac{\partial^2 (\partial_{\theta_1}D'h_k)}{\partial s_j^2}   \Big|\le C \tau^{-k}\omega(\tau^k),$$where $F_1$ is defined below
\begin{equation}\label{eqn:F2}
|z_1|^{2(1-\beta_1)} \frac{\partial^2 (\partial_{\theta_1} D'h_k)}{\partial z_1 \partial \bar z_1} = - |z_2|^{2(1-\beta_2)} \frac{\partial^2 (\partial_{\theta_1} D'h_k)}{\partial z_2 \partial \bar z_2} - \sum_{j=5}^{2n} \frac{\partial^2 (\partial_{\theta_1}D'h_k)}{\partial s_j^2}=:F_1.
\end{equation}
We apply similar arguments as in the proof of Lemma \ref{lemma 2.9}. For any $x\in \sS_1\cap \frac{1}{4}\hat B_k(p)$, we view the equation \eqref{eqn:F2} as defined on the $\mathbb C$-ball $B_{\mathbb C}(x, (\tau^k)^{1/\beta_1})$, and by the estimate \eqref{eqn:2.5} we have on $B_{\mathbb C} (x, (\tau^k)^{1/\beta_1}/2)\backslash \{x\}$ 
$$\Big| \frac{\partial (\partial_{\theta_1}D'h_k)}{\partial z_1}  \Big| \le C \frac { \| \partial_{\theta_1} D'h_k\|_{L^\infty(\hat B_{\mathbb C})}   }{(\tau^k)^{1/\beta_1}} + C \| F_1\|_{L^\infty(\hat B_{\mathbb C})} (\tau^k)^{2- \frac{1}{\beta_1}}.  $$
Equivalently, this means that on $B_{\mathbb C} (x, (\tau^k)^{1/\beta_1}/2)\backslash \{x\}$ 
$$ \Ba{\frac{\partial^2 D'h_k}{\partial r_1 \partial \theta_1}} + \Ba{\frac{\partial^2 D'h_k}{ r_1 \partial \theta_1^2}} \le r_1^{\frac{1}{\beta_1} - 1} \Ba{ \frac{\partial (\partial_{\theta_1} D'h_k)}{\partial z_1}  }\le C r_1 ^{\frac{1}{\beta_1} - 1} \tau^{k (1-\frac{1}{\beta_1})} \omega(\tau^k). $$
Again by the inclusion \eqref{eqn:inclusion}, we get  \eqref{eqn:my 2}. The estimate \eqref{eqn:my 3} follows from Lemma \ref{lemma 2.9}, \eqref{eqn:my 2}, \eqref{eqn:F} and the equation (from \eqref{eqn:F1}) below
$$\frac{\partial^2 D'h_k}{\partial r_1^2} =  \frac{1}{r_1}\frac{\partial D'h_k}{\partial r_1} - \frac{1}{\beta_1^2 r_1^2}\frac{\partial^2 D'h_k}{\partial \theta_1^2} + F.$$

\end{proof}

\begin{lemma}\label{lemma 2.11}
There exists a constant $C(n,\bbeta)>0$ for $k\le \min(k_{2,p},\ell)$, the following pointwise estimates hold for any $z\in \frac 1 4 \hat B_k(p)\backslash \sS$
\begin{equation}\label{eqn:my 8}
\Ba{ \frac{\partial }{\partial r_2}\bk{ \frac{\partial D'h_k}{\partial r_1}  }(z)    } + \Ba{ \frac{\partial}{r_2\partial \theta_2} \bk{ \frac{\partial D'h_k}{\partial r_1}  }(z)   } \le C(n,\bbeta) r_1^{\frac{1}{\beta_1} - 1} r_2^{\frac{1}{\beta_2} - 1} \tau^{-k (-1+\frac 1{\beta_1} + \frac{1}{\beta_2})} \omega(\tau^k).
\end{equation}

\end{lemma}
\begin{proof}By the Laplacian estimate in \eqref{eqn:lap final prop} for the harmonic function $D'h_k$ on $\frac 1 2 \hat B_k(p)$, we have
\begin{equation}\label{eqn:lap estimate}
\sup_{\frac{1}{2.2}\hat B_k(p)} \bk{ \ba{\Delta_1 D'h_k} + \ba{ \Delta_2 D'h_k }   } \le C(n)\tau^{-2k} \osc_{\frac 12 \hat B_k(p)} (D'h_k)\le C(n)\tau^{-k}\omega(\tau^k).
\end{equation}
Since $\Delta_1 (D' h_k)$ is also $g_{\bbeta}$-harmonic, the Laplacian estimates \eqref{eqn:lap final prop} imply that
\begin{equation}\label{eqn:my 9}
\sup_{\frac{1}{2.4} \hat B_k(p)} \bk{ \ba{\Delta_1 \Delta_1 D'h_k}  + \ba{ \Delta_2 \Delta_1 D'h_k } } \le C(n,\bbeta) \tau^{-2k} \osc_{\frac{1}{2.2} \hat B_k(p)} \Delta_1 D'h_k \le C \tau^{-3k} \omega(\tau^k).
\end{equation} 
Now from the equation $\Delta_\bb ( \Delta_1 D'h_k )= 0$ 
\begin{equation}\label{eqn:F4}
|z_1|^{2(1-\beta_1)}\frac{\partial^2}{\partial z_1 \partial \bar z_1} \Delta_1 D'h_k = - \Delta_2 \Delta_1 D'h_k - \sum_j \frac{\partial^2}{\partial s_j^2} \Delta_1 D'h_k =: F_2.
\end{equation}
From \eqref{eqn:my 9} and the Laplacian estimates \eqref{eqn:lap final prop}, we see that $\sup_{\frac{1}{2.4} \hat B_k(p)} | F_2 |\le C \tau^{-3k} \omega(\tau^k)$. By similar argument by considering $x\in \frac{1}{3} \hat B_k (p)\cap \sS_1$, we obtain from \eqref{eqn:F4} that on $\hat B :=B_{\mathbb C}(x, (\tau^k)^{1/\beta_1}/2)\backslash \{x\}$
\begin{equation*}
\ba{ \frac{\partial}{\partial z_1} \Delta_1 D'h_k  }\le C \frac{\| \Delta_1 D'h_k\|_{L^\infty(\hat B)}}{(\tau^k)^{1/\beta_1}} + C \| F_2\|_{L^\infty(\hat B)} (\tau^k)^{2- \frac{1}{\beta_1}}\le C \tau^{-k (1+ \frac{1}{\beta_1})} \omega(\tau^k).
\end{equation*}
This implies that for any $z\in \frac{1}{3}\hat B_k(p)\backslash \sS$
\begin{equation}\label{eqn:my 10}
\ba{ \frac{\partial}{\partial r_1}   \Delta_1 D'h_k  (z)} + \ba{ \frac{\partial}{r_1 \partial \theta_1}  \Delta_1 D'h_k (z)  } \le C r_1 ^{\frac{1}{\beta_1} - 1} \tau^{-k (1+ \frac{1}{\beta_1})}\omega(\tau^k).
\end{equation}
Now taking $\frac{\partial}{\partial r_1}$ on both sides of $\Delta_\bb D'h_k = 0$, we get 
\begin{equation}\label{eqn:F5}
|z_2|^{2(1-\beta_2)} \frac{\partial^2}{\partial z_2 \partial \bar z_2} \bk{ \frac{\partial D'h_k}{\partial r_1}  } = - \frac{\partial}{\partial r_1} (\Delta_1 D'h_k) - \sum_j \frac{\partial^2}{\partial s_j^2}\bk{ \frac{\partial D'h_k}{\partial r_1}}= : F_3.
\end{equation}
From \eqref{eqn:my 10}, for any $z\in \frac 13 \hat B_k\backslash \sS$, $|F_3|(z)\le C r_1^{\frac{1}{\beta_1} - 1} \tau^{-k(1+\frac{1}{\beta_1})} \omega(\tau^k)$. By similar argument for any $y\in \frac{1}{3.2}\hat B_k(p)\cap \sS_2$, we apply the estimate \eqref{eqn:2.5} to $\frac{\partial D'h_k}{\partial r_1}$, we get on $A_1:=B_{\mathbb C} (y, (\tau^k)^{1/\beta_2}/2)\backslash \{y\}$, the punctured ball in the complex plane $\mathbb C$ of (Euclidean) radius $(\tau^k)^{1/\beta_2}/2$ and orthogonal to $\sS_2$ passing through $y$, 
$$\ba{ \frac{\partial}{\partial z_2} \bk{\frac{\partial D'h_k}{\partial r_1}}  } \le C \frac{\| \frac{\partial D'h_k}{\partial r_1}\|_{L^\infty(A_1)}}{(\tau^k)^{1/\beta_2}} + C \| F_3\|_{L^\infty(A_1)} (\tau^k)^{2- \frac{1}{\beta_2}} \le C r_1^{\frac{1}{\beta_1} - 1} \tau^{-k ( \frac{1}{\beta_1} + \frac{1}{\beta_2} - 1  )} \omega(\tau^k).$$ Varying $y\in \frac{1}{3.2} \hat B_k(p)\cap \sS_2$ we get for any $z\in \frac{1}{4}\hat B_k \backslash \sS$, that the following pointwise estimate holds
\begin{equation}\label{eqn:my 11}
\ba{ \frac{\partial}{\partial r_2} \bk{ \frac{\partial D'h_k}{\partial r_1}  } (z) } + \ba{ \frac{\partial}{r_2 \partial \theta_2} \bk{ \frac{\partial D'h_k}{\partial r_1}  }(z)  } \le C r_1^{\frac{1}{\beta_1} - 1} r_2 ^{\frac{1}{\beta_2} - 1} \tau^{-k ( \frac{1}{\beta_1} + \frac{1}{\beta_2} - 1)} \omega(\tau^k).
\end{equation}

\end{proof}

\begin{lemma}\label{lemma 2.12}
Let $d=d_{\bbeta}(p,q)$. There exists a constant $C(n,\bbeta)$ such that for all $k\le \ell$
\begin{equation}\label{eqn:my 4}
\Ba{ \frac{\partial D'h_k}{\partial r_1}(p) - \frac{\partial D'h_k}{\partial r_1} (q)    } \le C d^{\frac{1}{\beta_1} - 1} \tau^{-k (\frac{1}{\beta_1} - 1)} \omega(\tau^k),
\end{equation}and 
\begin{equation}\label{eqn:my 5}
\Ba{\frac{\partial D'h_k}{r_1 \partial \theta_1}(p) - \frac{\partial D'h_k}{r_1\partial \theta_1} (q)   }\le C d^{\frac{1}{\beta_1} - 1} \tau^{-k (\frac{1}{\beta_1} - 1)} \omega(\tau^k).
\end{equation}

\end{lemma}
\begin{proof} We will consider the different cases $r_p = \min(r_p,r_q)\le 2d$ and $r_p = \min(r_p,r_q)>2d$.

\medskip

\noindent $\blacktriangleright$ $r_p\le 2d$. In this case, it is clear by the choice of $\ell$ that $r_p\approx \tau^{k_p} \le 2d \le \tau^{\ell +2}$, so  $k_p\ge \ell+2$. 

From our assumption when solving \eqref{eqn:uk 2}, $r_p = d_{\bbeta}(p,\sS_1)$, i.e. $r_1(p) = r_p\le 2d$. By triangle inequality we have $r_1(q)\le 3d$. We also remark that for $k\le \ell$, $\tau^k\ge \tau^\ell> 8d$, in particular, the geodesics considered below all lie inside the balls $\frac{1}{4} \hat B_k(p)$, and the estimates in Lemma \ref{lemma 2.9} - Lemma \ref{lemma 2.11} holds for points on these geodesics.

\smallskip

Let $p=(r_1(p),\theta_1(p); r_2(p),\theta_2(p); s(p))$ and $q = (r_1(q),\theta_1(q); r_2 (q),\theta_2(q); s(q))$ be the coordinates of the points $p,q$. Let $\gamma: [0,d]\to B_{\bbeta}(0,q)\backslash \sS$ be the unique $g_{\bbeta}$-geodesic connecting $p$ and $q$. We know the curve $\gamma$ is disjoint with $\sS$, and we denote $\gamma(t) = (r_1(t),\theta_1(t); r_2(t), \theta_2(t); s(t))$ be the coordinates of $\gamma(t)$ for $t\in [0,d]$. By definition we have for $\forall t\in[0,d]$
\begin{equation*}
|\gamma'(t)|_{g_{\bbeta}}^2 = (r_1 '(t))^2 + \beta_1^2 r_1(t)^2 (\theta_1'(t))^2 + (r_2 '(t))^2 + \beta_2^2 r_2(t)^2 (\theta_2'(t))^2 + |s'(t)|^2 = 1. 
\end{equation*}
So $|s(p) - s(q)|\le d$ and $|r_i(p) - r_i(q)|\le d$ for $i = 1, 2$. We denote
\begin{equation}\label{eqn:q q}q':= (r_1(q), \theta_1(q); r_2(p),\theta_2(p);s(p)),\quad p':= ( r_1(p), \theta_1(q) ; r_2(p),\theta_2(p);s(p) )\end{equation}
the points with coordinates related to $p$ and $q$. Let $\gamma_1$ be the $g_{\bbeta}$-geodesic connecting $q$ and $q'$; $\gamma_2$ the $g_{\bbeta}$-geodesic joining $q'$ to $p'$; $\gamma_3$ the $g_{\bbeta}$-geodesic joining $p'$ to $p$.

By triangle inequality, we have
\begin{align*}
&\quad  \Ba{ \frac{\partial D'h_k}{\partial r_1}(p) - \frac{\partial D'h_k}{\partial r_1}(q)   }\\ 
\le &\quad  \Ba{ \frac{\partial D'h_k}{\partial r_1}(p) - \frac{\partial D'h_k}{\partial r_1}(p')  } + \Ba{ \frac{\partial D'h_k}{\partial r_1}(p') - \frac{\partial D'h_k}{\partial r_1}(q')  }\\
& \quad + \Ba{ \frac{\partial D'h_k}{\partial r_1}(q') - \frac{\partial D'h_k}{\partial r_1}(q)  } =: J_1' + J_2' + J_3'.
\end{align*}
Integrating along $\gamma_3$ on which the points have fixed $r_1$-coordinate $r_1(p)$, we get by \eqref{eqn:my 2}
\begin{equation}\label{eqn:lemma 2.11 1}
J_1' = \Ba{ \int_{\gamma_3} \frac{\partial}{\partial\theta_1}\bk{ \frac{\partial D'h_k}{\partial r_1}  }d\theta_1  }\le C(n,\bbeta) r_1(p)^{\frac {1}{\beta_1} - 1} \tau^{-k(\frac{1}{\beta_1} - 1)} \omega(\tau^k).
\end{equation}
Integrating along $\gamma_2$ and by \eqref{eqn:my 3} we get
\begin{equation}\label{eqn:lemma 2.11 2}\begin{split}
J_2'  = &~~ \Ba{ \int_{\gamma_2} \frac{\partial}{\partial r_1}\bk{ \frac{\partial D'h_k}{\partial r_1}  } dr_1   }\\
\le & ~~C(n,\bbeta) \tau^{-k (\frac{1}{\beta_1} - 1)} \omega(\tau^k)\Ba{ \int_{r_1(p)} ^{r_1(q)}  t^{\frac{1}{\beta_1} - 2} dt   } \\
= & ~~ C(n,\bbeta) \tau^{-k (\frac{1}{\beta_1} - 1)}\omega(\tau^k) | r_1(p)^{\frac{1}{\beta_1} - 1} - r_1(q)^{\frac{1}{\beta_1} - 1}  |\\
\le & ~~ C(n,\bbeta) \tau^{-k (\frac{1}{\beta_1} - 1)}\omega(\tau^k) | r_1(p) - r_1(q)|^{\frac{1}{\beta_1} - 1}  \\
\le & ~~ C(n,\bbeta)\tau^{-k (\frac{1}{\beta_1} - 1)}\omega(\tau^k) d^{\frac{1}{\beta_1} - 1}.
\end{split}\end{equation}
To deal with $J_3'$, we need to consider different choices of $k\le \ell$.

\medskip

\noindent$\bullet$ If $k_{2,p}+1\le k\le \ell$, the balls $\hat B_k(p)$ are centered at $p_1\in\sS_1$ (recall $p_1$ is the projection of $p$ to $\sS_1$, hence $p$ and $p_1$ have the same $(r_2,\theta_2;s)$-coordinates). $\tau^{-k}\le 8^{-1} d^{-1}$ by the choice of $\ell$. The balls $\hat B_k(p)$ are disjoint with $\sS_2$, so we can introduce the smooth coordinates $w_2 = z_2^{\beta_2}$, and under the coordinates $(r_1,\theta_1; w_2, z_3,\ldots,z_n)$, the metric $g_{\bbeta}$ become the smooth cone metric with conical singularity {\em only} along $\sS_1$ with angle $2\pi \beta_1$. Therefore we can derive the following estimate  as in \eqref{eqn:grad hk 1} that 
\begin{equation}\label{eqn:our 1}
\sup_{\frac 1 2 \hat B_k(p)\backslash \sS_1} \Ba{ \frac{\partial (D')^2 h_k}{\partial r_1}   } + \Ba{ \frac{\partial}{\partial r_1}\bk{\frac {\partial D'h_k}{\partial w_2}}  } \le C \tau^{-k} \omega(\tau^k).
\end{equation}

Since $q$ and $q'$ have the same $(r_1,\theta_1)$-coordinates and  $g_{\bbeta}$ is a product metric, $\gamma_1$ is in fact a straight line segment (under the coordinates $(w_2,z_3,\ldots, z_n)$)  in the hyperplane with fixed $(r_1,\theta_1)$-coordinates. Integrating over $\gamma_1$, we get
\begin{align*}
J_3' \le &~  \int_{\gamma_1} \ba{\frac{\partial }{\partial w_2} \bk{\frac{\partial D'h_k}{\partial r_1}}} + \sum_j \ba{ \frac{\partial }{\partial s_j} \bk{\frac{\partial D'h_k}{\partial r_1}  }   }\\
\le &~ C \tau^{-k} \omega(\tau^k) d_{\bbeta}(q,q')
\le C\tau^{-k}\omega(\tau^k) d\\
\le &~ C(n,\bbeta)\tau^{-k (\frac{1}{\beta_1} - 1)}\omega(\tau^k) d^{\frac{1}{\beta_1} - 1}.
\end{align*}

\noindent $\bullet$ If $k\le k_{2,p}$, the balls $\hat B_k(p)$ are centered at the $p_{1,2}\in\sS_1\cap \sS_2$ and $\tau^k\ge r_2(p)$. By triangle inequality $r_2(q)\le d + r_2(p)\le \frac 9 8 \tau^k$. We choose points as follows
\begin{equation}\label{eqn:qq2}\tilde q = (r_1(q),\theta_1(q); r_2(p),\theta_2(p); s(q)),\quad \hat q = (r_1(q),\theta_1(q); r_2(q),\theta_2(p); s(q)).\end{equation}
Let $\tilde \gamma_1$ be the $g_{\bbeta}$-geodesic joining $q'$ to $\tilde q$; $\tilde\gamma$ the $g_{\bbeta}$-geodesic joining $\tilde q$ to $\hat q$; and $\hat\gamma$ the $g_{\bbeta}$-geodesic joining $\hat q$ to $q$. The curves $\tilde\gamma_1$, $\tilde \gamma$ and $\hat \gamma$ all lie in the hyperplane with constant $(r_1,\theta_1)$-coordinates $(r_1(q),\theta_1(q))$.  Then by triangle inequality we have
\begin{align*}
J_3'\le & \ba{ \frac{\partial D'h_k}{\partial r_1}(q') - \frac{\partial D'h_k}{\partial r_1}(\tilde q)  } + \ba{ \frac{\partial D'h_k}{\partial r_1}(\tilde q) - \frac{\partial D'h_k}{\partial r_1}(\hat q)  }\\
& + \ba{ \frac{\partial D'h_k}{\partial r_1}(\hat q) - \frac{\partial D'h_k}{\partial r_1}(q)  } =: J_1'' + J_2'' + J_3''.
\end{align*} We will use frequently the inequalities that $r_1(q)\le 3d$ and $\max(r_2(q),r_2(p)) \le 2\tau^k$ in the estimates below.
Integrating along $\hat \gamma$ we get by \eqref{eqn:my 8}
\begin{equation*}
J_3''\le \Ba{ \int_{\hat\gamma}\frac{\partial}{\partial \theta_2} \bk{\frac{\partial D'h_k}{\partial r_1}   } d\theta_2  }\le C r_1(q)^{\frac{1}{\beta_1} - 1} r_2(q)^{\frac{1}{\beta_2} } \tau^{-k ( -1 + \frac{1}{\beta_1} + \frac{1}{\beta_2}  )} \omega(\tau^k)\le C d^{\frac{1}{\beta_1} - 1}\tau^{-k(\frac{1}{\beta_1} - 1)} \omega(\tau^k)
\end{equation*}
Integrating along $\tilde \gamma$ we get again by \eqref{eqn:my 8}
\begin{align*}
J_2''\le & \Ba{ \int_{\tilde \gamma} \frac{\partial}{\partial r_2}\bk{ \frac{\partial D'h_k}{\partial r_1}   } dr_2  }\\
\le &  C r_1(q)^{\frac{1}{\beta_1} - 1} \tau^{-k ( -1 + \frac{1}{\beta_1} + \frac{1}{\beta_2}  )} \omega(\tau^k) \Ba{ \int_{r_2(q)}^{r_2(p)} t^{\frac{1}{\beta_2} - 1} dt  }\\
\le & C r_1(q)^{\frac{1}{\beta_1} - 1} \tau^{-k ( -1 + \frac{1}{\beta_1} + \frac{1}{\beta_2}  )} \omega(\tau^k) \max(r_2(q),r_2(p))^{\frac{1}{\beta_2}-1}    d\\
\le &  C d^{\frac{1}{\beta_1} - 1}\tau^{-k(\frac{1}{\beta_1} - 1)} \omega(\tau^k)
\end{align*}
Integrating along $\tilde \gamma_1$ we get by \eqref{eqn:our 2}
\begin{align*}
J_1''\le \Ba{ \int_{\tilde \gamma_1} \frac{\partial}{\partial s_j} \bk{ \frac{\partial D'h_k}{\partial r_1}  }   dt  } \le C r_1(q)^{\frac{1}{\beta_1} -1} \tau^{-  \frac{k}{\beta_1} } \omega(\tau^k) d\le  C d^{\frac{1}{\beta_1} - 1}\tau^{-k(\frac{1}{\beta_1} - 1)} \omega(\tau^k).
\end{align*}
Combining the above three inequalities, we get in the case $k\le k_{2,p}$ that 
$$J_3'\le  C d^{\frac{1}{\beta_1} - 1}\tau^{-k(\frac{1}{\beta_1} - 1)} \omega(\tau^k).$$
Combining the estimates on $J_1', J_2', J_3'$, we finish the proof of \eqref{eqn:my 4} in the case $r_p\le 2d$. 

\medskip

\noindent$\blacktriangleright$ $r_p>2d$ and $\ell \le k_p$. In this case $\tau^{k_p}\approx r_p >2d \ge \tau^{\ell +3}$. From triangle inequality we get $d_{\bbeta}(\gamma(t), \sS)\ge d$, in particular, the $r_1$ and $r_2$ coordinates of $\gamma(t)$ are both bigger than $d$. In this case $k\le \ell \le k_p$, Lemma \ref{lemma 2.9} - Lemma \ref{lemma 2.11} hold for the points in $\gamma$. $r_1(\gamma(t))\le r_1(p) + d\le 2 \tau^k$. We calculate the gradient of $\frac{\partial D'h_k}{\partial r_1}$ along $\gamma$
\begin{align*}
\ba { \nabla_{g_{\bbeta}} \frac{\partial D'h_k}{\partial r_1} }^2 = &~ \Ba{ \frac{\partial}{\partial r_1} \bk{\frac{\partial D'h_k}{\partial r_1} } }^2 + \Ba{ \frac{\partial }{\beta_1 r_1 \partial \theta_1}\bk{ \frac{\partial D'h_k}{\partial r_1}  } }^2 + \Ba{ \frac{\partial}{\partial r_2} \bk{\frac{\partial D'h_k}{\partial r_1}  }}^2 \\
&\quad + \Ba{ \frac{\partial}{\beta_2 r_2 \partial\theta_2} \bk{\frac{\partial D'h_k}{\partial r_1}  } }^2
 + \sum_j \Ba{\frac{\partial}{\partial s_j}\bk{\frac{\partial D'h_k}{\partial r_1}  } }^2.
\end{align*}
\noindent (1). When $k_{2,p}+1 \le k\le \ell$ we have by \eqref{eqn:our 1} that 
\begin{equation}\label{eqn:our 3}
\sup_{\frac 1 2 \hat B_k\backslash\sS_1} \Ba{ \frac{\partial }{\partial r_2} \bk{ \frac{\partial D'h_k}{\partial r_1}  }   } + \Ba{ \frac{\partial}{r_2 \partial \theta_2} \bk{ \frac{\partial D'h_k}{\partial r_1}  }   } \le C \tau^{-k}\omega(\tau^k).
\end{equation}
Thus by Lemma \ref{lemma 2.10}, \eqref{eqn:our 2} and \eqref{eqn:our 3} along $\gamma$ we have 
\begin{equation*}
\ba { \nabla_{g_{\bbeta}} \frac{\partial D'h_k}{\partial r_1} }\le C \omega(\tau^k)\xk{ d^{\ibone - 2} \tau^{-k (\ibone  - 1)} + \tau^{-k}    }
\end{equation*}
Integrating along $\gamma$ we get
\begin{align*}
\Ba{ \frac{\partial D'h_k}{\partial r_1}(p) - \frac{\partial D'h_k}{\partial r_1}(q)  } \le & ~\int_\gamma \ba { \nabla_{g_{\bbeta}} \frac{\partial D'h_k}{\partial r_1} } \le C  \omega(\tau^k)\xk{ d^{\ibone - 1} \tau^{-k (\ibone  - 1)} + d \tau^{-k}    }\\
\le & ~  C  \omega(\tau^k) d^{\ibone - 1} \tau^{-k (\ibone  - 1)}.
\end{align*}

\noindent (2). When $k\le k_{2,p}$, we have $r_2(\gamma(t)) \le r_2(p) + d \le \tau^k + d\le \frac{9}{8}\tau^k$ and similar estimates hold for $r_1(\gamma(t))$ too. Then by Lemma \ref{lemma 2.10}, Lemma \ref{lemma 2.11} and \eqref{eqn:our 2} along $\gamma$ the following estimate holds
\begin{equation*}
\ba { \nabla_{g_{\bbeta}} \frac{\partial D'h_k}{\partial r_1} }\xk{\gamma(t)}\le C \omega(\tau^k)\xk{ d^{\ibone - 2} \tau^{-k (\ibone  - 1)} + \tau^{-k}    }
\end{equation*}
Integrating along $\gamma$ we get 
\begin{align*}
\Ba{ \frac{\partial D'h_k}{\partial r_1}(p) - \frac{\partial D'h_k}{\partial r_1}(q)  } \le & ~\int_\gamma \ba { \nabla_{g_{\bbeta}} \frac{\partial D'h_k}{\partial r_1} } \le C  \omega(\tau^k)\xk{ d^{\ibone - 1} \tau^{-k (\ibone  - 1)} + d \tau^{-k}    }\\
\le & ~  C  \omega(\tau^k) d^{\ibone - 1} \tau^{-k (\ibone  - 1)}.
\end{align*}
This finishes the proof of the lemma in this case.

\medskip

\noindent $\blacktriangleright$ $r_p> 2d$ but $\ell \ge k_p + 1$. When $k\le k_p$, the estimate \eqref{eqn:my 4} follows in the same way as the case above. So it suffices to consider the case when $k_{p}+1 \le k\le \ell$. In this case the balls $\hat B_k(p) = B_{\bbeta}(p, \tau^k)$ and it can be seen by triangle inequality that the geodesic $\gamma\subset \frac{1}{3}\hat B_k(p)\ms$. Since the metric balls $\hat B_k(p)$ are disjoint with $\sS$ we can use the smooth coordinates $w_1 = z_1^{\beta_1}$ and $w_2 = z_2^{\beta_2}$ as before, and everything becomes smooth under these coordinates in $\hat B_k(p)$.

\medskip

The estimate \eqref{eqn:my 5} can be shown by the same argument, so we skip the details.
\end{proof}
Iteratively applying \eqref{eqn:my 4} for $k\le\ell$, we get
\begin{align*}J_2 = & \ba{ \frac{\partial D'u_\ell}{\partial r_1} (p) - \frac{\partial D'u_\ell}{\partial r_1}(q) } \le \ba{\frac{\partial D'u_2}{\partial r_1}(p) - \frac{\partial D'u_2}{\partial r_1}  (q)} + C d^{\frac{1}{\beta_1} - 1} \sum_{k=3}^\ell \tau^{-k (\frac{1}{\beta_1} - 1)} \omega(\tau^k)\\
\le & C d^{\frac{1}{\beta_1} - 1} \bk{  \| u\|_{C^0} +    \sum_{k=2}^\ell \tau^{-k (\frac{1}{\beta_1} - 1)} \omega(\tau^k)  },
\end{align*} where the inequality
$$\ba{\frac{\partial D'u_2}{\partial r_1}(p) - \frac{\partial D'u_2}{\partial r_1}  (q)}  \le C d^{\frac{1}{\beta_1} - 1} \| u\|_{C^0} $$ can be proved by the same argument as in proving \eqref{eqn:my 4}.

Combining the estimates for $J_1,\, J_2,\, J_3, \, J_4$ we finish the proof of \eqref{eqn:1.3}.

\medskip

We remark that in solving \eqref{eqn:uk 2}, we assume $r_1(p)\le r_2(p)$, we need also to deal with the following case, whose proof is more or less parallel to that of  Lemma \ref{lemma 2.12}, so we just point out the differences and sketch the proof.

\begin{lemma}
Let $d=d_{\bbeta}(p,q)>0$. There exists a constant $C(n,\bbeta)>0$ such that for all $k\le \ell$, 
\begin{equation}\label{eqn:r 2 1}
\Ba{ \frac{\partial D'h_k}{\partial r_2}(p)  - \frac{\partial D' h_k}{\partial r_2} (q) } \le C d^{\frac{1}{\beta_2} - 1} \tau^{-k(\frac{1}{\beta_2} - 1)} \omega(\tau^k),
\end{equation}
\begin{equation}\label{eqn:r 2 2}
\Ba{ \frac{\partial D'h_k}{r_2 \partial \theta_2}(p) - \frac{\partial D'h_k}{r_2 \partial \theta_2} (q) } \le C d^{\frac{1}{\beta_2} - 1} \tau^{-k (\frac{1}{\beta_2} - 1) } \omega(\tau^k).
\end{equation}
\end{lemma}
\begin{proof} We consider two different cases on whether $k\le k_{1,p}$ or $k_{1,p}+1\le k\le \ell$.

\noindent $\bullet$ $k_{2,p}+1\le k\le \ell$. The balls $\hat B_k(p)$ are disjoint with $\sS_2$, so we can introduce the complex coordinate $w_2 = z_2^{\beta_2}$ on these balls as before. Let $t_1, t_2$ be the real and imaginary parts of $w_2$, respectively.  The derivatives estimates imply that
\begin{equation*}
\| \partial_{w_2} D'h_k\|_{L^\infty(\frac1{2} \hat B_k(p))} \le C \omega(\tau^k),\quad \| \partial^2_{w_2} D'h_k\|_{L^\infty(\frac{1}{2} \hat B_k(p))} \le C\tau^{-k } \omega(\tau^k),
\end{equation*}
where $\partial^2_{w_2}$ denotes the full second order derivatives in the $\{t_1, t_2\}$-directions. And
\begin{equation*}
\big\| \frac{\partial }{\partial r_1} \bk{\frac{\partial D'h_k}{\partial w_2} }   \big\|_{L^\infty(\frac 1 2 \hat B_k(p))} + \big\|  \frac{\partial }{r_1\partial \theta_1} \bk{\frac{\partial D'h_k}{\partial w_2} } \big\|_{L^\infty(\frac 1 2 \hat B_k(p))} \le C \tau^{-k}\omega(\tau^k).
\end{equation*}

 Since \begin{equation}\label{eqn:radial complex}\frac{\partial}{\partial r_2} = \frac{w_2}{r_2}\frac{\partial}{\partial w_2} + \frac{\bar w_2}{r_2 } \frac{\partial}{\partial \bar w_2},\end{equation} it holds that
$$\frac{\partial}{\partial w_2} \bk{ \frac{\partial D'h_k}{\partial r_2}  } = \frac{1}{r_2} \frac{\partial D'h_k}{\partial w_2  } - \frac{|w_2|^2}{2 r_2^3}\frac{\partial D'h_k}{\partial w_2} -\frac{\bar w_2 \cdot\bar w_2}{2 r_2^3} \frac{\partial D'h_k}{\partial \bar w_2} +  \frac{w_2}{r_2} \partial^2_{w_2} D'h_k,$$
we have on $\frac 1 2 \hat B_k(p)$ $$\ba{\frac{\partial}{\partial w_2} \bk{ \frac{\partial D'h_k}{\partial r_2}  }   } \le \frac{C}{r_2} \omega(\tau^k) + C \tau^{-k} \omega(\tau^k).   $$And
\begin{equation*}
\big\| \frac{\partial }{\partial r_1} \bk{\frac{\partial D'h_k}{\partial r_2} }   \big\|_{L^\infty(\frac 1 2 \hat B_k(p))} + \big\|  \frac{\partial }{r_1\partial \theta_1} \bk{\frac{\partial D'h_k}{\partial r_2} } \big\|_{L^\infty(\frac 1 2 \hat B_k(p))} \le C \tau^{-k}\omega(\tau^k).
\end{equation*}

Therefore
\begin{align*}
|\nabla_{g_{\bbeta}} \frac{\partial D'h_k}{\partial r_2}|^2 & = \Ba{ \frac{\partial^2 D'h_k}{\partial r_1 \partial r_2}  }^2 + \Ba{ \frac{\partial ^2}{r_1 \partial \theta_1 \partial r_2}  }^2 + \Ba{ \frac{\partial^2 D'h_k}{\partial w_2 \partial r_2}  }^2 + \sum_j \Ba{\frac{\partial^2 D'h_k}{\partial s_j \partial r_2}}^2\\
& \le C \xk { \tau^{-k} \omega(\tau^k)  }^2 + C \frac{1}{r_2^2} \omega(\tau^k)^2.
\end{align*}
In this case we know that $r_1(p)\approx \tau^{k_p}\ge 2 \tau^k\ge \tau^\ell > 8d$, so along $\gamma$
\begin{equation*}
r_2\xk{\gamma(t)} \ge r_2(p) - d\ge r_1(p) - d\ge \frac 7 4 \tau^k.
\end{equation*}
%
Integrating along $\gamma$ we get
\begin{align*}
\Ba{ \frac{\partial D'h_k}{\partial r_2} (p) - \frac{\partial D'h_k}{\partial r_2}(q)   }\le  &\int_{\gamma} \ba{ \nabla_{g_{\bbeta}} \frac{\partial D'h_k}{\partial r_2}   } \le C \tau^{-k}\omega(\tau^k) d \le C \tau^{-k (\frac{1}{\beta_2} - 1)} \omega(\tau^k) d^{\frac{1}{\beta_2} - 1}.
\end{align*}
\noindent$\bullet$ $k\le k_{2,p}$.  This case is completely the same as in the proof of \eqref{eqn:my 4}, by replacing $r_1$ by $r_2$, $\beta_1$ by $\beta_2$. So we omit the details. 

\eqref{eqn:r 2 2} can be proved similarly.
\end{proof}

\subsection{Mixed normal directions}\label{section 3.4} In this section, we will deal with the H\"older continuity of the following four mixed derivatives:
\begin{equation}\label{eqn:operators} \frac{\partial^2 u}{\partial r_1 \partial r_2},\, \frac{\partial^2 u}{r_1 \partial \theta_1 \partial r_2}, \, \frac{\partial^2 u}{r_2 \partial r_1 \partial \theta_2},\, \frac{\partial^2 u}{r_1 r_2\partial\theta_1 \partial\theta_2 },  \end{equation} which by our previous notation correspond to $N_1 N_2 u$.
Since the proof for each of them is more or less the same, we will only prove the H\"older continuity for $\frac{\partial^2 u}{\partial r_1 \partial r_2}.$ The following holds at $p$ and $q$ by the same reasoning of Lemma \ref{lemma 2.6}
$$\lim_{k\to\infty} \frac{\partial ^2 u_k}{\partial r_1 \partial r_2}(p) = \frac{\partial^2 u}{\partial r_1 \partial r_2}(p), \quad \lim_{k\to\infty} \frac{\partial ^2 v_k}{\partial r_1 \partial r_2}(q) = \frac{\partial^2 u}{\partial r_1 \partial r_2}(q).$$ By triangle inequality
\begin{align*}
\Ba{\frac{\partial^2 u}{\partial r_1 \partial r_2}(p) - \frac{\partial^2 u}{\partial r_1 \partial r_2}(q)   } \le & \Ba{ \frac{\partial^2 u}{\partial r_1 \partial r_2}(p) - \frac{\partial^2 u_\ell}{\partial r_1 \partial r_2}(p)  } + \Ba{ \frac{\partial^2 u_\ell}{\partial r_1 \partial r_2}(p) - \frac{\partial^2 u_\ell}{\partial r_1 \partial r_2}(q)  }\\
& ~  + \Ba{\frac{\partial^2 u_\ell}{\partial r_1 \partial r_2}(q) - \frac{\partial^2 v_\ell}{\partial r_1 \partial r_2}(q)  } + \Ba{\frac{\partial^2 v_\ell}{\partial r_1 \partial r_2}(q) - \frac{\partial^2 u}{\partial r_1 \partial r_2}(q)}\\
=: & ~L_1 + L_2 + L_3 + L_4. 
\end{align*}


\begin{lemma}\label{lemma 2.14}
We have the following estimate
\begin{equation*}
L_1 + L_4\le  \sum_{k = \ell}^\infty \omega(\tau^k)
\end{equation*}
\end{lemma}
\begin{proof}
We consider the different cases that $k\ge k_p+1$ and $\ell \le k\le k_p$.

\noindent $\bullet$ $k\ge k_{p} + 1$. In this case the balls $\hat B_k(p)$ are disjoint with $\sS$ and we can introduce the smooth coordinates $w_1 = z_1^{\beta_1}$ and $w_2 = z_2^{\beta_2}$. Under the coordinates $\{w_1,w_2,z_3,\ldots, z_n\}$, $g_{\bbeta}$ becomes the standard Euclidean metric $g_{\mathbb C^n}$ and the metric balls $\hat B_k(p)$ become the standard Euclidean ball with the same radius and center $p$. Since the $g_{\bbeta}$-harmonic functions $u_k - u_{k+1}$ satisfy \eqref{eqn:uk comp}, by standard gradient estimates for Euclidean harmonic functions we get
\begin{equation*}
\sup_{\frac{1}{2.1} \hat B_k(p)} \Ba{ D_{w_1} D_{w_2} (u_k - u_{k - 1})  } \le C \omega(\tau^k),
\end{equation*}
where we use $D_{w_i}$ to denote either $\frac{\partial}{\partial w_i}$ or $\frac{\partial}{\partial \bar w_i}$ for simplicity. From \eqref{eqn:radial complex} and similar formula for $\frac{\partial}{\partial r_1}$, we get
\begin{equation}\label{eqn:2.80}
\sup_{\frac{1}{2.1}\hat B_k(p)} \Ba{ \frac{\partial^2}{\partial r_1 \partial r_2} (u_k - u_{k - 1})  }\le C \omega(\tau^k). 
\end{equation}

\noindent $\bullet$ If $\ell\ge k_{2,p}+1$ and $\ell \le k_p = k_{1,p}$. For all $\ell \le k$, the balls $\hat B_k(p)$ are disjoint with $\sS_2$ and center at $p_1$. We can still use $w_2 = z_2^{\beta_2}$ as the smooth coordinate. The cone metric $g_{\bbeta}$ becomes smooth in $w_2$-variable and we can apply the standard gradient estimate to the $g_{\bbeta}$-harmonic function $D_{w_2} (u_k - u_{k-1})$ to get 
\begin{equation*}
 \sup_{\frac{1}{2.2} \hat B_k(p)} \Ba{  \frac{\partial  }{\partial r_1} D_{w_2}(u_k - u_{k-1})  } + \Ba{ \frac{\partial}{r_1 \partial\theta_1}  D_{w_2}(u_k - u_{k-1})   } \le C \omega(\tau^k).
\end{equation*}
Again by \eqref{eqn:radial complex}, we get
\begin{equation}\label{eqn:2.81}
 \sup_{\frac{1}{2.2} \hat B_k(p)} \Ba{  \frac{\partial^2  }{\partial r_1\partial r_2} (u_k - u_{k-1})  } + \Ba{ \frac{\partial^2 }{r_1 \partial\theta_1\partial r_2} (u_k - u_{k-1})   } \le C \omega(\tau^k).
\end{equation}

\noindent $\bullet$ If $\ell \le k_{2,p}$, the case when $k\ge k_{2,p}+1$ can be dealt with similarly as above. In the case $\ell \le k\le k_{2,p}$, $r_2(p)\approx \tau^{k_{2,p}}\le \tau^k \le \tau^\ell \approx 8d$.
 Now the balls $\hat B_k(p)$ are centered at $p_{1,2}\in\sS_1\cap \sS_2$. We can proceed as in the proof of Lemma \ref{lemma 2.11} with the harmonic functions $u_k - u_{k-1}$ replacing $D'h_k$ as in that lemma to prove that for any $z\in {\frac{1}{3} \hat B_k(p)}\backslash \Ss$
 \begin{equation*}\begin{split}
&  \Ba{ \frac{\partial^2}{\partial r_1 \partial r_2} (u_k - u_{k-1})    } (z)+ \Ba{ \frac{\partial ^2}{r_2 \partial\theta_2 \partial r_1} (u_k - u_{k-1})  }(z)\\
 \le  & ~C(n,\bbeta) r_1(z)^{\frac{1}{\beta_1} - 1} r_2(z)^{\frac{1}{\beta_2} - 1} \tau^{-k(-2 + \frac{1}{\beta_1} + \frac{1}{\beta_2})} \omega(\tau^k).
 \end{split}\end{equation*}
In particular, the estimate in each case holds at $p$ and from $r_1(p)\le r_2(p)\le \tau^k$, we obtain
\begin{equation}\label{eqn:2.82}
\Ba{ \frac{\partial^2}{\partial r_1 \partial r_2} (u_k - u_{k-1})    } (p)+ \Ba{ \frac{\partial ^2}{r_2 \partial\theta_2 \partial r_1} (u_k - u_{k-1})  }(p)\le C \omega(\tau^k).
\end{equation}
Combining each case above, by \eqref{eqn:2.80}, \eqref{eqn:2.81} and \eqref{eqn:2.82}, we get for all $k\ge \ell$
$$\Ba{ \frac{\partial^2 u}{\partial r_1 \partial r_2}  ( u_k - u_{k-1})  }(p) \le C (n,\bb)\omega(\tau^k).$$
Therefore by triangle inequality
\begin{equation*}
L_1 \le \sum_{k=\ell+1} ^\infty \Ba{ \frac{\partial^2 u}{\partial r_1 \partial r_2}  ( u_k - u_{k-1})  }(p) \le C(n,\bbeta)\sum_{k=\ell+1}^\infty \omega(\tau^k).
\end{equation*}
The estimate for $L_4$ can be dealt with similarly by studying the derivatives of $v_k$ at $q$. 

\end{proof}

\begin{lemma}\label{lemma 2.15}
$L_3$ satisfies
\begin{equation*}
L_3 \le C(n,\bbeta) \omega(\tau^\ell).
\end{equation*}
\end{lemma}
\begin{proof}

      As in the proof of previous lemma, we need to consider different cases: $\ell \ge k_{1,p}+1$,  $\ell \ge k_{2,p}+1$ or $\ell \le k_{2,p}$. 

\medskip

\noindent$\bullet$ If $\ell \ge k_{1,p}+1$, then the ball $\hat B_\ell(p) = B_{\bbeta}(p,\tau^\ell)$ and the function $U$ defined in \eqref{eqn:defn of U} is $g_{\bbeta}$-harmonic in $\frac{1}{2}\hat B_\ell(p)$, and $\sup_{\frac 1 2 \hat B_\ell(p)} |U|\le C \omega^{2\ell} \omega(\tau^\ell)$. Since $\frac 1 2 \hat B_\ell(p)$ is disjoint with $\sS$, $w_1$ and $w_2$ are well-defined on $\frac {1}{2}\hat B_\ell(p)$ and thus we have the derivative estimate:
\begin{equation*}
\sup_{\frac{1}{3}\hat B_\ell(p)} \Ba{ \frac{\partial^2 U}{\partial r_1\partial r_2}     } \le \sup_{\frac{1}{3}\hat B_\ell(p)} \Ba{ D_{w_1} D_{w_2} U    } \le C(n,\bbeta) \omega(\tau^\ell).
\end{equation*}
In particular, at $q\in\frac 1 3 \hat B_\ell(p)$
$$L_3 = \Ba{  \frac{\partial^2 u_\ell}{\partial r_1\partial r_2}(q)  -  \frac{\partial^2 v_\ell}{\partial r_1\partial r_2}(q)    } = \Ba{ \frac{\partial^2 U}{\partial r_1\partial r_2} (q)  }\le C(n,\bbeta) \omega(\tau^\ell).$$

\medskip

\noindent$\bullet$ If $k_{1,p}\ge \ell \ge k_{2,p}$, then the ball $\hat B_\ell(p) = B_{\bbeta}(p_1, 2\tau^\ell) $ and the function $U$ defined in \eqref{eqn:defn of U} is $g_{\bbeta}$-harmonic and well-defined in a ball $B_q := B_{\bbeta}(q, \tau^\ell/10)\subset \frac{1}{2.2}\hat B_\ell(p)$, and $\sup_{\frac 1 2 \hat B_\ell(p)} |U|\le C \omega^{2\ell} \omega(\tau^\ell)$. Since $\frac 1 2 \hat B_\ell(p)$ is disjoint with $\sS_2$, $w_2$ is well-defined on $\frac {1}{2.2}\hat B_\ell(p)$ and thus we have the derivatives estimates:
\begin{equation*}
\sup_{\frac{1}{2} B_q} \Ba{ \frac{\partial^2 U}{\partial r_1\partial r_2}     } \le \sup_{\frac{1}{2}B_q} \Ba{ \frac{\partial}{\partial r_1} D_{w_2} U    } \le C(n,\bbeta) \omega(\tau^\ell).
\end{equation*}
In particular, at $q\in\frac{1}{2} B_q$, we have
$$L_3 = \Ba{  \frac{\partial^2 u_\ell}{\partial r_1\partial r_2}(q)  -  \frac{\partial^2 v_\ell}{\partial r_1\partial r_2}(q)    } = \Ba{ \frac{\partial^2 U}{\partial r_1\partial r_2} (q)  }\le C(n,\bbeta) \omega(\tau^\ell).$$
 
\noindent $\bullet$ If $\ell \le k_{2,p} - 1$, then $ r_2(p)\approx  \tau^{k_{2,p}} \le \tau^{\ell+1}< 8d$, so $r_2(q) \le r_2(p) + d\le \frac{5}{8}\tau^\ell$  and $r_1(q)\le d + r_1(p)\le d+ r_2(p)\le \frac{5}{8}\tau^\ell$. Therefore the ball $\tilde B_\ell(q)$ is centered at $q_1$, or $q_2$ or $q_{1,2}\in\sS_1\cap \sS_2$ with radius $2\tau^\ell$. It follows from this that the function $U$ defined in \eqref{eqn:defn of U} is well-defined on the ball $\frac{1}{1.8} \hat B_\ell(p)$. 

By the same strategy as in the proof of Lemma \ref{lemma 2.11}, with the harmonic function $D'h_k$ in that lemma replaced by $U$ on the metric ball $\frac{1}{1.8} \hat B_\ell(p)  $, we can prove that for any $z\in \frac{1}{2}\hat B_\ell(p)\backslash \sS$ the following holds:
\begin{equation*}
\Ba{ \frac{\partial ^ 2 U}{\partial r_1 \partial r_2} (z)   } \le C(n,\bbeta) r_1^{\frac{1}{\beta_1} - 1} r_2^{\ibtwo -1} \tau^{-\ell ( -2 + \ibone + \ibtwo   )} \omega(\tau^\ell).
\end{equation*}
Applying this inequality at $q$ we get
\begin{equation*}
L_3 = \Ba{ \frac{\partial ^ 2 (u_\ell - v_\ell)}{\partial r_1 \partial r_2} (q)   } \le C(n,\bbeta) r_1(q)^{\frac{1}{\beta_1} - 1} r_2(q)^{\ibtwo -1} \tau^{-\ell ( -2 + \ibone + \ibtwo   )} \omega(\tau^\ell)\le C \omega(\tau^\ell).
\end{equation*}

\medskip

\noindent In sum, in all cases $L_3\le C(n,\bbeta) \omega(\tau^\ell)$.

\end{proof}


\begin{lemma}\label{lemma 2.16}
There exists a constant $C=C(n,\bbeta)>0$ such that for all $k\le \ell $ and $z\in \frac{1}{3}\hat B_k(p)\backslash \Ss$
{\small
\begin{equation}\label{eqn:2.84}
\Ba{ \frac{\partial}{\partial\theta_1}\bk{ \frac{\partial^2 h_k}{\partial r_1 \partial r_2}  }(z)  } + \Ba{ \bk{ \frac{\partial^3 h_k}{r_1\partial\theta_1^2  \partial r_2 }  }  }\le C\cdot \left\{\begin{aligned}
& r_1^{\ibone - 1} \tau^{-k(-1+ \ibone)} \omega(\tau^k), \text{ if }k\in[k_{2,p} +1,\min(\ell, k_p)]\\
& r_1^{\ibone - 1} r_2^{\ibtwo - 1} \tau^{-k ( - 2 + \ibone + \ibtwo  )} \omega(\tau^k),\text{ if }k\le k_{2,p}.
\end{aligned}\right. 
\end{equation}
}
\end{lemma}
\begin{proof}
The proof is parallel to that of Lemma \ref{lemma 2.11}. The function $\frac{\partial h_k}{\partial \theta_1}$ is $g_{\bbeta}$-harmonic on $\hat B_k(p)$, and by the Laplacian estimate \eqref{eqn:lap final prop}, we have 
\begin{equation*}
\sup_{\frac{1}{1.2}\hat B_k(p)} \bk{ \ba{ \Delta_1 \frac{\partial h_k}{\partial \theta_1}  } + \ba{\Delta_2 \frac{\partial h_k}{\partial \theta_1}  }   } \le C(n,\bb) \omega(\tau^k).
\end{equation*}
The function $\Delta_2 \frac{\partial h_k}{\partial \theta_1} $ is also $g_{\bbeta}$-harmonic, so the Laplacian estimates \eqref{eqn:lap final prop} imply
\begin{equation*}
\sup_{\frac{1}{1.4} \hat B_k(p)}\bk{ \ba{ \Delta_1 \Delta_2 \frac{\partial h_k}{\partial \theta_1}  } + \ba{\Delta_2 \Delta_2 \frac{\partial h_k}{\partial \theta_1} }+ \ba{  (D')^2 \Delta_2 \frac{\partial h_k}{\partial \theta_1}  }    } \le C \tau^{-2k}\bk{ \osc_{\frac{1}{1.2}\hat B_k(p)} \Delta_2 \frac{\partial h_k}{\partial \theta_1} } \le C \tau^{-2k}\omega(\tau^k).
\end{equation*}We consider  the equation 
\begin{equation}\label{eqn:G5}|z_2|^{2(1-\beta_2)} \frac{\partial^2}{\partial z_2 \partial \bar z_2} \bk{\Delta_2 \frac{\partial h_k}{\partial \theta_1} }= - \Delta_1 \Delta_2  \frac{\partial h_k}{\partial \theta_1} - \sum_j \frac{\partial^2}{\partial s_j^2} \Delta_2 \frac{\partial h_k}{\partial \theta_1}=: F_5, \end{equation}
where the function $F_5$ satisfies $\sup_{\frac{1}{1.4}\hat B_k(p)} |F_5|\le C \tau^{-2k}\omega(\tau^k)$. 

$\bullet$ In case $k_{2,p}+1 \le k\le \min(\ell, k_p)$, we can introduce the smooth coordinate $w_2 = z_2^{\beta_2}$ in the ball $\frac{1}{1.5}\hat B_k(p)$ as before, since this ball is disjoint with $\sS_2$ and under the coordinates $(r_1,\theta_1;w_2,z_3,\ldots, z_n)$ we can use the usual standard gradient estimate to the $g_{\bbeta}$-harmonic function $\Delta_2 \frac{\partial h_k}{\partial \theta_1}$ to obtain that 
\begin{equation}\label{eqn:2.86}\sup_{\frac{1}{2}\hat B_k(p)} \Ba{ \frac{\partial}{\partial r_2}\bk{\Delta_2 \frac{\partial h_k}{\partial \theta_1} } } + \Ba{ \frac{\partial}{r_2 \partial \theta_2} \bk{ \Delta_2 \frac{\partial h_k}{\partial \theta_1}  }  } \le C \tau^{-k}\omega(\tau^k). \end{equation}

\medskip

$\bullet$ In case $k\le k_{2,p}$, the ball $\hat B_k(p)$ is centered at $p_{1,2}$. We apply the usual estimate \eqref{eqn:2.5} to the function $\Delta_2 \frac{\partial h_k}{\partial \theta_1}$, the solution to the equation \eqref{eqn:G5}, on any $\mathbb C$-ball $ A_2: = B_{\mathbb C}(y, (\tau^k)^{1/\beta_2})$ for any $y\in \sS_2 \cap \frac{1}{1.6} \hat B_k(p)$ where $A_2$ denotes the Euclidean ball in the complex plane orthogonal to $\sS_2$ and passing through $y$. Then for any $z\in B_{\mathbb C}(y, (\tau^k)^{1/\beta_2}/2)\backslash \{y\}$
\begin{equation*}
\Ba{ \frac{\partial}{\partial z_2} \bk{\Delta_2 \frac{\partial h_k}{\partial \theta_1}} (z) } \le C \frac{\|\Delta_2 \frac{\partial h_k}{\partial \theta_1} \|_{L^\infty(A_2)}}{\tau^{k/\beta_2}} + C \| F_5\|_{L^\infty(A_2)} (\tau^k)^{2 - \ibtwo}\le C \tau^{-k/\beta_2} \omega(\tau^k).
\end{equation*}
This implies that on $\frac 1 2 \hat B_k(p)\backslash \sS$
\begin{equation}\label{eqn:2.87}
\Ba{ \frac{\partial }{\partial r_2} \bk{ \Delta_2 \frac{\partial h_k}{\partial \theta_1}  }   } + \Ba{ \frac{\partial}{r_2 \partial\theta_2} \bk{\Delta_2 \frac{\partial h_k}{\partial \theta_1}}   } \le C r_2^{\ibtwo - 1} \tau^{-k/\beta_2} \omega(\tau^k).
\end{equation}

Taking $\frac{\partial}{\partial r_2}$ on both sides of $\Delta_\bb \frac{\partial h_k}{\partial \theta_1} = 0$, we get
\begin{equation}\label{eqn:G6}
|z_1|^{2(1-\beta_1)} \frac{\partial^2 }{\partial z_1 \partial \bar z_1} \bk{ \frac{\partial^2 h_k}{\partial r_2 \partial \theta_1}  } = - \frac{\partial}{\partial r_2}\bk{\Delta_2 \frac{\partial h_k}{\partial \theta_1} }  - \sum_j\frac{\partial^2 }{\partial s_j^2} \bk{ \frac{\partial ^2 h_k}{\partial r_2 \partial \theta_1}  } = : F_6.
\end{equation}
It is not hard to see from \eqref{eqn:2.86} and \eqref{eqn:2.87} and standard derivative estimates that on $\frac{1}{1.8} \hat B_k(p)\backslash \sS$
\begin{enumerate}[label=(\roman*), fullwidth] \item in case $k_{2,p}+1 \le k\le \min(\ell, k_p)$, $ |F_6| \le C \tau^{-k} \omega(\tau^k)$;  \item in case $k\le k_{2,p}$ $|F_6| \le C r_2^{\ibtwo - 1} \tau^{-\frac{k}{\beta_2}}\omega(\tau^k)$.
\end{enumerate} Then by applying estimate \eqref{eqn:2.5} to the function $\frac{\partial^2 h_k}{\partial r_2 \partial \theta_1}$ on any $\mathbb C$-ball $A_3: = B_{\mathbb C} (x, (\tau^k)^{1/\beta_1})$ for any $x\in \frac{1}{1.8}\hat B_k(p)\cap \sS_1$ we get that on $B_{\mathbb C}(x, (\tau^k)^{1/\beta_1}/2)\backslash \{x\}$
{\small
\begin{align*}
\Ba{ \frac{\partial}{\partial r_1} \bk{ \frac{\partial^2 h_k}{\partial r_2 \partial \theta_1}  }  } + & \Ba{ \frac{\partial}{ r_1 \partial \theta_1} \bk{ \frac{\partial^2 h_k}{\partial r_2 \partial \theta_1}  }  } \le  ~C r_1^{\ibone - 1} \frac{\| \frac{\partial^2 h_k}{\partial r_2 \partial \theta_1 }   \|_{L^\infty(A_3)}}{\tau^{k/\beta_1}} + C r_1 ^{\ibone - 1}\| F_6\|_{L^\infty(A_3)} \tau^{k( 2- \ibone  )}\\
\le &~ C\cdot \left\{\begin{aligned}
& r_1^{\ibone - 1} \tau^{-k(-1+ \ibone)} \omega(\tau^k), \text{ if }k\in[k_{2,p} +1,\min(\ell, k_p)]\\
& r_1^{\ibone - 1} r_2^{\ibtwo - 1} \tau^{-k ( -2 +   \ibone + \ibtwo  )} \omega(\tau^k),\text{ if }k\le k_{2,p}.
\end{aligned}\right. 
\end{align*}
}
Therefore this estimate  holds on $\frac{1}{3}\hat B_k(p)\backslash \sS$.

\end{proof}
\begin{lemma}\label{lemma 2.17}
For any $k\le \ell$ and any point $z\in \frac 1 3 \hat B_k(p)\backslash \sS$ the following estimates hold
\begin{equation}\label{eqn:2.84 new}
\Ba{  \frac{\partial^2 h_k}{\partial r_1 \partial r_2}  (z)  }\le C\cdot \left\{\begin{aligned}
& r_1^{\ibone - 1} \tau^{-k(-1+ \ibone)} \omega(\tau^k), \text{ if }k\in[k_{2,p} +1,\min(\ell, k_p)]\\
& r_1^{\ibone - 1} r_2^{\ibtwo - 1} \tau^{-k ( - 2 +   \ibone + \ibtwo  )} \omega(\tau^k),\text{ if }k\le k_{2,p}.
\end{aligned}\right. 
\end{equation}
\begin{equation*}
\Ba{  \frac{\partial^2 D'h_k}{\partial r_1 \partial r_2}  (z)  }\le C\cdot \left\{\begin{aligned}
& r_1^{\ibone - 1} \tau^{-k/\beta_1} \omega(\tau^k), \text{ if }k\in[k_{2,p} +1,\min(\ell, k_p)]\\
& r_1^{\ibone - 1} r_2^{\ibtwo - 1} \tau^{-k ( - 1 +   \ibone + \ibtwo  )} \omega(\tau^k),\text{ if }k\le k_{2,p}.
\end{aligned}\right. 
\end{equation*}
\end{lemma}
\begin{proof}
This follows from almost the same argument as in  the proof of Lemma \ref{lemma 2.16}, by studying the harmonic functions $h_k$ and $D'h_k$ instead of $\frac{\partial h_k}{\partial \theta_1}$. 
\end{proof}

\begin{lemma}
The following estimate holds for any $k\le \ell $ and any $z\in\frac{1}{3}\hat B_k(p)\backslash \sS$
{\small
\begin{equation}\label{eqn:r1 r2}
\Ba{ \frac{\partial}{\partial r_1}\bk{\frac{\partial^2 h_k}{\partial r_1 \partial r_2}   }   }(z) \le C\omega(\tau^k)\cdot \left\{\begin{aligned}
& \tau^{-k} + r_1(z)^{\ibone - 2} \tau^{-k (\ibone - 1)}    , \text{ if }k\in[k_{2,p}+1,\min(\ell, k_p)]\\
& r_2(z)^{\ibtwo - 1}\tau^{-\frac{k}{\beta_2}} + r_1(z)^{\ibone - 2} r_2(z)^{\ibtwo -1}\tau^{-k( - 2 +  \ibone + \ibtwo )}, \text{ if } k\le k_{2,p}
\end{aligned}\right.
\end{equation}
}

\end{lemma}
\begin{proof}
By the Laplacian estimates \eqref{eqn:lap final prop} we have
\begin{equation*}
\sup_{\frac{1}{1.2}\hat B_k(p)\backslash \Ss} \ba{ \Delta_1 h_k  } + \ba{\Delta_2 h_k} \le C(n,\bbeta)\omega(\tau^k).
\end{equation*}
Applying again the Laplacian estimate \eqref{eqn:lap final prop} to the $g_{\bbeta}$-harmonic function $\Delta_1 h_k$, 
\begin{equation*}
\sup_{\frac{1}{1.4} \hat B_k(p)} \bk{ \ba{\Delta_1 \Delta_1 h_k  } + \ba{\Delta_2\Delta_1 h_k}  + |(D')^2 \Delta_1 h_k| }\le C(n,\bbeta)\tau^{-2k}\omega(\tau^k).
\end{equation*}
We consider the equation
\begin{equation}\label{eqn:G7}
|z_2|^{2-2\beta_2}\frac{\partial^2}{\partial z_2 \partial \bar z_2} \Delta_1 h_k = -\Delta_1 \Delta_1 h_k - \sum_j \frac{\partial^2}{\partial s_j^2} \Delta_1 h_k = : F_7.
\end{equation}
From the estimates above, we see $\| F_7\|_{L^\infty(\frac{1}{1.8}\hat B_k(p))}\le C \tau^{-2k}\omega(\tau^k)$.

\noindent $\bullet$ When $k_{2,p}+1 \le k\le \min(\ell, k_p)$, we directly apply gradient estimate to $\Delta_1 h_k$ to get
\begin{equation}\label{eqn:grad lap hk}
\sup_{\frac{1}{1.5}\hat B_k(p)\backslash \Ss} \Ba{ \frac{\partial}{\partial r_2} \Delta_1 h_k  } + \Ba{ \frac{\partial}{r_2 \partial \theta_2} \Delta_1 h_k }\le C \tau^{-k}\omega(\tau^k).
\end{equation}
\noindent$\bullet$ When $k\le k_{2,p}$, the balls $\hat B_k(p)$ are centered at $p_{1,2}$, and we can apply the usual $\mathbb C$-ball type estimate to get that for any $z\in \frac{1}{2}\hat B_k(p)\backslash \sS$
\begin{align*}
\Ba{\frac{\partial}{\partial r_2} \Delta_1h_k   }(z) + \Ba{ \frac{\partial}{r_2\partial \theta_2} \Delta_1 h_k  }\le&  C r_2(z)^{\ibtwo - 1} \frac{\| \Delta_1 h_k\|_{L^\infty}}{\tau^{k/\beta_2}} + C r_2(z)^{\ibtwo - 1} \| F_7\|_{L^\infty} \tau^{k(2-\ibtwo)}  \\
\le &~ C r_2(z)^{\ibtwo - 1} \tau^{-k/\beta_2} \omega(\tau^k).
\end{align*}
Recall  the following equation holds
\begin{align*}
{ \frac{\partial}{\partial r_1} \bk{ \frac{\partial^2 h_k}{\partial r_1 \partial r_2}  }  } = & { \frac{\partial}{\partial r_2}\Delta_1 h_k - \frac{1}{r_1} \frac{\partial^2 h_k}{\partial r_1 \partial r_2} - \frac{1}{\beta_1^2 r_1^2} \frac{\partial^3 h_k}{\partial\theta_1^2 \partial r_2}   }
\end{align*}
from which we derive that for any $z\in\frac 1 2 \hat B_k(p)\backslash \sS$ 
{\small
\begin{align*}
\Ba{ \frac{\partial}{\partial r_1} \bk{ \frac{\partial^2 h_k}{\partial r_1 \partial r_2}  }  }(z) \le C\omega(\tau^k) \left\{\begin{aligned}
& \tau^{-k} + r_1(z)^{\ibone - 2} \tau^{-k (\ibone - 1)}    , \text{ if }k\in[k_{2,p}+1,\min(\ell, k_p)]\\
& r_2(z)^{\ibtwo - 1}\tau^{-\frac{k}{\beta_2}} + r_1(z)^{\ibone - 2} r_2(z)^{\ibtwo -1}\tau^{-k( - 2+ \ibone + \ibtwo )}, \text{ if } k\le k_{2,p}
\end{aligned}\right.
\end{align*}
}
\end{proof}

\begin{lemma}\label{lemma 2.19}
There exists a constant $C=C(n,\bbeta)>0$ such that for all $k\le \ell$ and $z\in \frac{1}{3}\hat B_k(p)\backslash \Ss$
{\small
\begin{equation}\label{eqn:2.93}
\Ba{ \frac{\partial}{\partial\theta_2}\bk{ \frac{\partial^2 h_k}{\partial r_1 \partial r_2}  }(z)  } + \Ba{ \bk{ \frac{\partial^3 h_k}{r_2\partial\theta_2^2  \partial r_1 }  } (z) }\le C\omega(\tau^k) \left\{\begin{aligned}
& r_1^{\ibone - 1} \tau^{-k(-1+ \ibone)} , \text{ if }k\in[k_{2,p} +1,\min(\ell, k_p)]\\
& r_1^{\ibone - 1} r_2^{\ibtwo - 1} \tau^{-k ( - 2 + \ibone + \ibtwo  )},\text{ if }k\le k_{2,p}.
\end{aligned}\right. 
\end{equation}
}
\end{lemma}
\begin{proof}
It follows from the Laplacian estimate \eqref{eqn:lap final prop} that
\begin{equation*}
\sup_{\frac{1}{1.2}\hat B_k(p)}\bk{ \ba{\Delta_1 \frac{\partial h_k}{\partial \theta_2}} + \ba{\Delta_2  \frac{\partial h_k}{\partial \theta_2}} } \le C(n) \omega(\tau^k).
\end{equation*}
And by \eqref{eqn:lap final prop}  again we have
\begin{equation*}
\sup_{\frac{1}{1.4}\hat B_k(p)}\bk{ \ba{\Delta_1 \Delta_1  \frac{\partial h_k}{\partial \theta_2}} + \ba{ \Delta_2 \Delta_1  \frac{\partial h_k}{\partial \theta_2}  } + \ba{ (D')^2 \Delta_1  \frac{\partial h_k}{\partial \theta_2} }    } \le C \tau^{-2k}\omega(\omega^k).
\end{equation*}
We look at the equation
\begin{equation*}
|z_1|^{2(1-\beta_1)} \frac{\partial^2}{\partial z_1 \partial \bar z_1} \bk{ \Delta_1  \frac{\partial h_k}{\partial \theta_2}  }  = - \Delta_2 \Delta_1  \frac{\partial h_k}{\partial \theta_2} - \sum_j \frac{\partial^2 }{\partial s_j^2} \bk{\Delta_1  \frac{\partial h_k}{\partial \theta_2}}=: F_8,
\end{equation*}
and $F_8$ satisfies $\sup_{\frac{1}{1.4} \hat B_k(p)} |F_8|\le C \tau^{-2k}\omega(\tau^k)$. By the estimate \eqref{eqn:2.5} as we did before it follows that for any $z\in \frac{1}{2}\hat B_k(p)\backslash \sS$ (remember here $k\le\min(\ell, k_p)$)
\begin{align*}
\ba{\frac{\partial}{\partial r_1}\Delta_1  \frac{\partial h_k}{\partial \theta_2}   } (z) + \ba{\frac{\partial}{r_1 \partial \theta_1} \Delta_1  \frac{\partial h_k}{\partial \theta_2}   }(z) \le&  Cr_1(z)^{\ibone - 1} \frac{\|\Delta_1  \frac{\partial h_k}{\partial \theta_2} \|_{L^\infty}}{\tau^{k/\beta_1}} + C r_1(z)^{\ibone - 1} \| F_8\|_{L^\infty} \tau^{k (2 - \ibone)}\\
\le &~ C r_1(z)^{\ibone - 1} \tau^{-k/\beta_1} \omega(\tau^k).
\end{align*}
Taking $\frac{\partial }{\partial r_1}$ on both sides of the equation $\Delta_\bb \frac{\partial h_k}{\partial \theta_2} = 0$, we get
\begin{equation}\label{eqn:G9}
|z_2|^{2(1-\beta_2)} \frac{\partial^2 }{\partial z_2 \partial \bar z_2}\bk{ \frac{\partial^2 h_k}{\partial r_1\partial \theta_2} } = - \frac{\partial}{\partial r_1}\bk{\Delta_1 \frac{\partial h_k}{\partial \theta_2}  } - \sum_{j} \frac{\partial}{\partial r_1}\bk{\frac{\partial^2}{\partial s_j^2} \frac{\partial h_k}{\partial \theta_2}   }=: F_9. 
\end{equation}
Here $|F_9(z)|\le C r_1(z)^{\ibone  - 1}\tau^{-k/\beta_1} \omega(\tau^k)$ for any $z\in \frac 1 2 \hat B_k(p)\backslash \sS$. Therefore we get by the usual $\mathbb C$-ball argument that

\noindent $\bullet$ If $k\le k_{2,p}$, then for any $z\in\frac 1 3 \hat B_k(p)\backslash \sS$ 
\begin{align*}
\Ba{\frac{\partial }{\partial r_2} \bk{\frac{\partial^2 h_k}{\partial r_1\partial \theta_2}  } (z)}    + \Ba{\frac{\partial}{r_2\partial \theta_2} \bk{ \frac{\partial^2 h_k}{\partial r_1\partial \theta_2}  } (z)}  \le C r_2(z)^{\ibtwo  -1 } r_1(z)^{\ibone - 1} \tau^{k(2-\ibone - \ibtwo  )}\omega(\tau^k)
\end{align*}
\noindent $\bullet$ If $k_{2,p}+1 \le k\le \min(\ell, k_p)$, then 
\begin{align*}
\Ba{\frac{\partial }{\partial r_2} \bk{\frac{\partial^2 h_k}{\partial r_1\partial \theta_2}  } (z)}    + \Ba{\frac{\partial}{r_2\partial \theta_2} \bk{ \frac{\partial^2 h_k}{\partial r_1\partial \theta_2}  }(z) }  \le C  r_1(z)^{\ibone - 1} \tau^{k(1-\ibone  )} \omega(\tau^k).
\end{align*}

\end{proof}
\begin{lemma}\label{lemma 2.20} 
The following estimate holds for any $k\le \ell$ and any $z\in \frac 1 3 \hat B_k(p)\backslash \sS$
{\small
\begin{equation}\label{eqn:you 2}
\Ba{\frac{\partial}{\partial r_2}\bk{ \frac{\partial ^2 h_k}{\partial r_1 \partial r_2}   }    }\le C \omega(\tau^k)\left\{\begin{aligned}
& r_1(z)^{\ibone - 1} \tau^{-k/\beta_1} + r_1^{\ibone - 1} r_2^{-1} \tau^{-k(-1+ \ibone)} , \text{ if }k\in[k_{2,p} +1,\min(\ell, k_p)]\\
& r_1(z)^{\ibone - 1} \tau^{-k/\beta_1} +  r_1^{\ibone - 1} r_2^{\ibtwo - 2} \tau^{-k ( -  2 +  \ibone + \ibtwo  )},\text{ if }k\le k_{2,p}.
\end{aligned}\right. 
\end{equation}
}

\end{lemma}
\begin{proof}
We first observe the following  equation
\begin{equation*}
\frac{\partial}{\partial r_2}\bk{  \frac{\partial ^2 h_k}{\partial r_1 \partial r_2}    } = \frac{\partial}{\partial r_1} \Delta_2 h_k - \frac{1}{r_2}\frac{\partial^2 h_k}{\partial r_1 \partial r_2} - \frac{1}{\beta_2^2 r_2^2} \frac{\partial^2}{\partial \theta_2^2} \bk{\frac{\partial h_k}{\partial r_1}   }.
\end{equation*}
It can be shown that for any $z\in \frac 1 2 \hat B_k(p)\backslash\sS$ by the $\mathbb C$-ball argument that 
\begin{equation*}
\Ba{ \frac{\partial}{\partial r_1} \Delta_2 h_k  (z) } \le C r_1(z)^{\ibone - 1} \tau^{-k/\beta_1} \omega(\tau^k).
\end{equation*}
From Lemma \ref{lemma 2.17}, we have for any $z\in \frac{1}{2} \hat B_k(p)\backslash \sS$ 
\begin{equation*}
\Ba{ \frac{1}{r_2} \frac{\partial^2 h_k}{\partial r_1 \partial r_2} (z)   } \le  C\cdot \left\{\begin{aligned}
& r_1^{\ibone - 1}r_2^{-1} \tau^{-k(-1+ \ibone)} \omega(\tau^k), \text{ if }k\in[k_{2,p} +1,\min(\ell, k_p)]\\
& r_1^{\ibone - 1} r_2^{\ibtwo - 2} \tau^{-k (  - 2 +  \ibone + \ibtwo  )} \omega(\tau^k),\text{ if }k\le k_{2,p}.
\end{aligned}\right. 
\end{equation*}
From Lemma \ref{lemma 2.19}, we have  for any $z\in \frac 1 2 \hat B_k(p)\backslash \sS$
\begin{equation*}
\Ba{ \frac{1}{r_2^2} \frac{\partial^3 h_k}{\partial r_1 \partial \theta_2^2} (z)   } \le C\omega(\tau^k)\cdot \left\{\begin{aligned}
& r_1^{\ibone - 1} r_2^{-1} \tau^{-k(-1+ \ibone)} , \text{ if }k\in[k_{2,p} +1,\min(\ell, k_p)]\\
& r_1^{\ibone - 1} r_2^{\ibtwo - 2} \tau^{-k ( - 2 +  \ibone + \ibtwo  )},\text{ if }k\le k_{2,p}.
\end{aligned}\right. 
\end{equation*}
Therefore for any $z\in \frac 1 3 \hat B_k(p)\backslash \sS$ we have
{\small
\begin{equation*}
\Ba{ \frac{\partial }{\partial r_2}\bk{ \frac{\partial^2 h_k}{\partial r_1 \partial r_2}  } (z) } \le C \omega(\tau^k)\cdot \left\{\begin{aligned}
& r_1(z)^{\ibone - 1} \tau^{-k/\beta_1} + r_1^{\ibone - 1} r_2^{-1} \tau^{-k(-1+ \ibone)} , \text{ if }k\in[k_{2,p} +1,\min(\ell, k_p)]\\
& r_1(z)^{\ibone - 1} \tau^{-k/\beta_1} +  r_1^{\ibone - 1} r_2^{\ibtwo - 2} \tau^{-k ( - 2 +  \ibone + \ibtwo  )},\text{ if }k\le k_{2,p}.
\end{aligned}\right. 
\end{equation*}
}

\end{proof}

It remains to estimate $L_2$. For simplicity, we denote $h_k := -u_{k} + u_{k-1}$ as before, where we take $k\le \ell$. We will denote $\beta_{\max} = \max(\beta_1,\beta_2  )$.
\begin{lemma}\label{lemma 2.21}
Let $d= d_{\bbeta}(p,q)$. There exists a constant $C(n,\bbeta)>0$ such that for all $k\le \ell$
\begin{equation*}
\Ba{ \frac{\partial^2 h_k}{\partial r_1 \partial r_2}(p) - \frac{\partial^2 h_k}{\partial r_1 \partial r_2}(q)  }\le   C \omega(\tau^k) \tau^{- k (\ibone  -1)} d^{\ibone  -1}\le C \omega(\tau^k) \tau^{- k (\frac{1}{\beta_{\max}}  -1)} d^{\frac{1}{\beta_{\max}}  -1}.
\end{equation*}

\end{lemma}
\begin{proof}
$ \blacktriangleright$ {\bf Case 1:}
First we assume that $r_p\le 2d $ so $r_q\le 3d $ and $\ell + 2 \le k_p$, in particular, the balls $\hat B_k(p)$ are centered at either $p_1\in\sS_1$ or $0$, depending on whether $k\ge k_{2,p}+1$ or $k\le k_{2,p}$. As in the proof of Lemma \ref{lemma 2.12}, let $\gamma:[0,d]\to B_{\bbeta}(0,1)\backslash S$ be the $g_{\bbeta}$-geodesic connecting $p$ and $q$. The two points $q'$ and $p'$ are defined as  in  \eqref{eqn:q q},  $\gamma_1, \gamma_2,\gamma_3$ the $g_{\bbeta}$-geodesics as defined in that lemma. We calculate by triangle inequality that 
\begin{align*}
& \Ba{ \frac{\partial ^2 h_k}{\partial r_1 \partial r_2}(p) - \frac{\partial ^2 h_k}{\partial r_1 \partial r_2}(q)   }\\
\le  & ~ \Ba{\frac{\partial ^2 h_k}{\partial r_1 \partial r_2}(p) - \frac{\partial ^2 h_k}{\partial r_1 \partial r_2}(p')    } + \Ba{ \frac{\partial ^2 h_k}{\partial r_1 \partial r_2}(p') - \frac{\partial ^2 h_k}{\partial r_1 \partial r_2}(q')  }\\
& ~ + \Ba{ \frac{\partial ^2 h_k}{\partial r_1 \partial r_2}(q') - \frac{\partial ^2 h_k}{\partial r_1 \partial r_2}(q)  } = : L_1' + L_2' + L_3'.
\end{align*}
Integrating along $\gamma_3$ on which the coordinates $(r_1; r_2,\theta_2; z_3,\ldots,z_n)$ are the same as $p$,  we get by \eqref{eqn:2.84}
\begin{align*}
L_1' =&  \Ba{ \int_{\gamma_3} \frac{\partial}{\partial \theta_1}\bk{ \frac{\partial ^2 h_k}{\partial r_1 \partial r_2}   }  d\theta_1 } \\
\le & C  \omega(\tau^k)\cdot \left\{\begin{aligned}
& r_1(p)^{\ibone - 1} \tau^{-k(-1+ \ibone)} , \text{ if }k\in[k_{2,p} +1,\ell]\\
& r_1(p)^{\ibone - 1} r_2(p)^{\ibtwo - 1} \tau^{-k ( - 2 +  \ibone + \ibtwo  )},\text{ if }k\le k_{2,p}.
\end{aligned}\right. 
\end{align*}
Integrating along $\gamma_2$ along which the coordinates $(\theta_1;r_2,\theta_2;z_3,\ldots, z_n)$ are the same as $p'$ or $q'$, we get by \eqref{eqn:r1 r2} that 
{\small
\begin{align*}
L_2' = & \Ba{\int_{\gamma_2} \frac{\partial }{\partial r_1}  \bk{ \frac{\partial^2 h_k}{\partial r_1 \partial r_2}   } dr_1  }\\
\le & C\omega(\tau^k)\cdot \left\{\begin{aligned}
& \tau^{-k} d + \tau^{-k (\ibone - 1)} |r_1(p) - r_1(q)|^{\ibone - 1}     , \text{ if }k\in[k_{2,p}+1,\ell]\\
& r_2(p)^{\ibtwo - 1}\tau^{-\frac{k}{\beta_2}} d + r_2(p)^{\ibtwo -1}\tau^{-k( -  2 + \ibone + \ibtwo )} | r_1(p) - r_1(q)  |^{\ibone  - 1}, \text{ if } k\le k_{2,p}
\end{aligned}\right.\\
\le & C\omega(\tau^k)\cdot \left\{\begin{aligned}
& \tau^{-k} d + \tau^{-k (\ibone - 1)} d^{\ibone - 1}     , \text{ if }k\in[k_{2,p}+1,\ell]\\
& r_2(p)^{\ibtwo - 1}\tau^{-\frac{k}{\beta_2}} d + r_2(p)^{\ibtwo -1}\tau^{-k( - 2 + \ibone + \ibtwo )} d^{\ibone  - 1}, \text{ if } k\le k_{2,p}
\end{aligned}\right.
\end{align*}
}

\noindent To deal with the term $L_3'$, we consider different cases of $k$, either $\ell \ge k\ge k_{2,p} + 1$ or $k\le k_{2,p}$.

\medskip

\noindent $\bullet$ If $k_{2,p}+1 \le k\le \ell$, the balls $\hat B_k(p)$ are centered at $p_1\in\sS_1$. Here $\tau^{-k}\le \tau^{-\ell}\le 8^{-1} d^{-1}$ and $\tau^k\le \tau^{k_{2,p}+1}\le\frac 1 2  r_2(p)$, so $r_2(q)\ge - d + r_{2}(p)\ge  \tau^k$. The balls $\hat B_k(p)$ are disjoint with $\sS_2$, we can use the smooth coordinate $w_2 = z_2^{\beta_2}$ as before. The functions $D_{w_2} D' h_k$ are $g_{\bbeta}$-harmonic, hence by gradient estimate we have
\begin{equation*}
\sup_{\frac{1}{1.2}\hat B_k(p)\backslash \sS_1} \ba{ \nabla_{g_{\bbeta}}\xk{ D_{w_2} D'h_k } }\le C(n) \frac{ \| D_{w_2} D'h_k\|_{L^\infty(\frac{1}{1.1} \hat B_k(p))}   }{\tau^{k}}\le C \tau^{-k}\omega(\tau^k).
\end{equation*}
From \eqref{eqn:radial complex}, we get  that
\begin{equation}\label{eqn:you 1}
\sup_{\frac{1}{1.2} \hat B_k(p)\backslash\sS_1} \ba{ \frac{\partial^2}{\partial r_1 \partial r_2} D'h_k   } \le C(n) \tau^{-k}\omega(\tau^k).
\end{equation}
Recall $r_1(p) = r_p\le 2d \le \frac{1}{2}\tau^{k}$, triangle inequality implies that $r_1(q)\le 3d\le \frac{1}{2}\tau^k$. 
The points in $\gamma_1$ have the fixed $(r_1,\theta_1)$-coordinates $(r_1(q),\theta_1(q))$, so integrating along $\gamma_1$ we get by \eqref{eqn:you 2} and \eqref{eqn:you 1} that
\begin{align*}
L_3'\le & \int_{\gamma_1} \Ba{\frac{\partial}{\partial r_2} \bk{ \frac{\partial^2 h_k}{\partial r_1 \partial r_2}  }   } + \Ba{ \frac{\partial}{r_2 \partial\theta_2}  \bk{ \frac{\partial^2 h_k}{\partial r_1 \partial r_2}   }  } + \Ba{ D'\bk{\frac{\partial^2 h_k}{\partial r_1 \partial r_2}   }  }\\
\le &~ C d \omega(\tau^k)\bk{ r_1(q)^{\ibone - 1} \tau^{-k/\beta_1} + \tau_1(q)^{\ibone - 1} \min(r_2(p), r_2(q))^{-1} \tau^{-k (\ibone - 1)} +  \tau^{-k}       }\\
\le & C  \tau^{-k}\omega(\tau^k)\cdot d\le C \tau^{-k(\ibone - 1)} \omega(\tau^k) d^{\ibone - 1}.
\end{align*}

\noindent$\bullet$ If $k\le k_{2,p}$, the $\tau^k\ge \tau^{k_{2,p}}\ge r_2(p)$ and $\tau^k\ge \tau^\ell \ge 8d$. Thus $r_2(q)\le r_2(p) + d \le \frac{3}{2} \tau^k$. We choose points $\tilde q, \hat q$ as in \eqref{eqn:qq2}, and let $\tilde\gamma_1$, $\tilde \gamma$ and $\hat \gamma$ be $g_{\bbeta}$-geodesics defined as in the proof of Lemma \ref{lemma 2.12}. Then we have
\begin{align*}
L'_3 \le &~ \Ba{ \frac{\partial^2 h_k}{\partial r_1 \partial r_2} (q') - \frac{\partial^2 h_k}{\partial r_1 \partial r_2}(\tilde q)   } + \Ba{\frac{\partial^2 h_k}{\partial r_1 \partial r_2}(\tilde q) - \frac{\partial^2 h_k}{\partial r_1 \partial r_2}(\hat q)   }\\
& ~ + \Ba{ \frac{\partial^2 h_k}{\partial r_1 \partial r_2}(\hat q) - \frac{\partial^2 h_k}{\partial r_1 \partial r_2} (q)  } = : L_1'' + L_2'' + L_3''.
\end{align*}
We will estimate $L''_1$, $L''_2$ and $L''_3$ term by term by integrating appropriate functions along the geodesics $\tilde \gamma_1$, $\tilde \gamma$ and $\hat \gamma$ as follows: The points in $\hat \gamma$ have fixed $(r_1,\theta_1; r_2; s)$-coordinates $(r_1(q),\theta_1(q); r_2(q); s(q))$ and by \eqref{eqn:2.93}
\begin{align*}
L''_3  = & \Ba{ \int_{\hat \gamma} \frac{\partial}{\partial \theta_2} \bk{ \frac{\partial^2 h_k}{\partial r_1 \partial r_2}   } d\theta_2   } \le C\omega(\tau^k) r_1(q)^{\ibone  -1 } r_2(q)^{\ibtwo - 1}\tau^{-k ( - 2+ \ibone + \ibtwo)}\\
 \le & C\omega(\tau^k) r_1(q)^{\ibone  -1 } \tau^{-k ( -1+ \ibone )}\le C \tau^{-k (\ibone - 1)} \omega(\tau^k) d^{\ibone  - 1}.
\end{align*}
Integrating along $\tilde \gamma$ on which the points have constant $r_1$-coordinate $r_1(q)$,  we get by \eqref{eqn:you 2}
{\small
\begin{align*}
L_2'' = & \Ba{ \int_{\tilde \gamma} \frac{\partial }{\partial r_2} \xk{\frac{\partial^2 h_k}{\partial r_1 \partial r_2}   } dr_2     }\\
\le & C \omega(\tau^k) \bk{ r_1(q)^{\ibone  -1} \tau^{-\frac{k}{\beta_1}} | r_2(q) - r_2(p)  | + r_1(q)^{\ibone  -1} \tau^{-k (-2 + \ibone + \ibtwo)} \ba{ r_2(q)^{\ibtwo -1} - r_2(p)^{\ibtwo - 1}   }    }\\
\le & C \omega(\tau^k) \bk{ r_1(q)^{\ibone  -1} \tau^{-\frac{k}{\beta_1}} d + r_1(q)^{\ibone  -1} \tau^{-k (-2 + \ibone + \ibtwo)} d^{\ibtwo -1}    }\\
 \le & C\omega(\tau^k) r_1(q)^{\ibone  -1 } \tau^{-k ( -1+ \ibone )}\le C \tau^{-k (\ibone - 1)} \omega(\tau^k) d^{\ibone  - 1}.
\end{align*}
}
Integrating along $\tilde \gamma_1$ on which the points have constant $(r_1,\theta_1; r_2,\theta_2)$-coordinates,  we have by \eqref{eqn:my 8} that
\begin{align*}
L_1'' \le & { \int_{\tilde \gamma_1}\Ba{ D'\xk{ \frac{\partial^2 h_k}{\partial r_1 \partial r_2}  }}   }\le C r_1(q)^{\ibone - 1} r_2(p)^{\ibtwo - 1} \tau^{-k ( -1 + \ibone + \ibtwo  )} d\\
\le & ~ C d \tau^{-k}  \omega(\tau^k) \le Cd^{\ibone - 1} \tau^{-k( \ibone - 1 )} \omega(\tau^k).
\end{align*}

Combining both cases, we conclude that $L_3'\le C \tau^{-k (\ibone - 1)} \omega(\tau^k) d^{\ibone  - 1}$. Then by the estimates above for $L'_1$ and $L'_2$, we finally get for all $k\le \ell$
\begin{equation*}
 \Ba{ \frac{\partial ^2 h_k}{\partial r_1 \partial r_2}(p) - \frac{\partial ^2 h_k}{\partial r_1 \partial r_2}(q)   } \le  C \omega(\tau^k) \tau^{- k (\ibone  -1)} d^{\ibone  -1}\le  C \omega(\tau^k) \tau^{- k (\frac{1}{\beta_{\max}}  -1)} d^{\frac{1}{\beta_{\max}}  -1},
\end{equation*}
where in the last inequality we use the fact that $\tau^{-k} d\le 1/8<1$, when $ k\le \ell$. Hence we finish the proof of Lemma \ref{lemma 2.21} in case $r_p\le 2d$. 

\smallskip

Now we deal with the remaining cases.

\medskip

\noindent $ \blacktriangleright$ {\bf Case 2:} here we assume $\min(r_p,r_q) = r_p\ge 2d$ and $\ell \le k_p$.   In this case $\tau^{k_p}\approx r_p \ge 2d \ge \tau^{\ell +3}$, so $\ell +3\ge k_p$. It follows by triangle inequality that $d_{\bbeta}(\gamma(t), \sS)\ge d$, where $\gamma $ is the $g_{\bbeta}$-geodesic joining $p$ to $q$ as before. In particular this implies that $\min(r_1(\gamma(t)), r_2(\gamma(t))  )\ge d$.

\smallskip

 In this case $\ell\le k_p $, Lemmas \ref{lemma 2.16} - \ref{lemma 2.20} hold for all $k\le \ell$ and $r_1(p)\approx \tau^{k_p}\le \tau^\ell$, so $r_1(\gamma(t))\le d + r_1(p)\le \frac{9}{8}\tau^\ell\le \frac 9 8 \tau^k$. We calculate the gradient of $\frac{\partial^2 h_k}{\partial r_1 \partial 2}$ along the geodesic $\gamma$ as follows 
 {\small
\begin{align*}
\big|\nabla_{g_{\bbeta}} \frac{\partial^2 h_k}{\partial r_1 \partial r_2}\big|^2\big(\gamma(t)\big)
 = &~~ \Ba{\frac{\partial}{\partial r_1}\bk{ \frac{\partial^2 h_k}{\partial r_1 \partial r_2}}}^2 + \Ba{ \frac{1}{\beta_1 r_1 \partial \theta_1}\bk{\frac{\partial^2 h_k}{\partial r_1 \partial r_2}  }     }^2 + \Ba{ \frac{\partial}{\partial r_2} \bk{\frac{\partial^2 h_k}{\partial r_1 \partial r_2}  }   }^2 \\
 &~~ + \Ba{ \frac{\partial}{\beta_2  r_2 \partial \theta_2} \bk{\frac{\partial^2 h_k}{\partial r_1 \partial r_2}  }   }^2 + \sum_j \Ba{\frac{\partial}{\partial s_j} \bk{\frac{\partial^2 h_k}{\partial r_1 \partial r_2}  }   }^2.
\end{align*}
}
(1). When $k_{2,p} + 1\le k\le \ell$  we have along $\gamma$
 $$r_2(\gamma(t))\ge r_2(p) - d \ge \tau^k - d \ge \frac{7}{8}\tau^k.  $$ Then by Lemmas \ref{lemma 2.16} - \ref{lemma 2.20} we have along $\gamma$ that 
\begin{align*}
|\nabla_{g_{\bbeta}} \frac{\partial^2 h_k}{\partial r_1 \partial r_2}(\gamma(t))| \le & C\omega(\tau^k)\bk{  \tau^{-k} + d^{\ibone - 2} \tau^{-k( \ibone -1  )}         }
\end{align*}
Integrating this inequality along $\gamma$ we get
\begin{align*}
\Ba{ \frac{\partial^2 h_k}{\partial r_1 \partial r_2}(p) - \frac{\partial^2 h_k}{\partial r_1 \partial r_2}(q)  } \le &  \int_{\gamma} \big |\nabla_{g_{\bbeta}} \frac{\partial^2 h_k}{\partial r_1 \partial r_2}\big|  \le C\omega(\tau^k)\bk{ d \tau^{-k} + d^{\ibone - 1} \tau^{-k( \ibone -1  )}         }\\
\le & C d^{\ibone - 1} \tau^{-k ( \ibone - 1  )}\omega(\tau^k).
\end{align*}

\smallskip

\noindent (2). When $k\le k_{2,p}$, we have along $\gamma$ $$r_2(\gamma(t))\le r_2(p) + d \le \tau^k +d\le \frac{9}{8}\tau^k.   $$
Then by Lemmas \ref{lemma 2.16} - \ref{lemma 2.20} we have along $\gamma$ that 
\begin{align*}
|\nabla_{g_{\bbeta}} \frac{\partial^2 h_k}{\partial r_1 \partial r_2}| \le & C\omega(\tau^k)\bk{  \tau^{-k} + d^{\ibone - 2} \tau^{-k( \ibone -1  )}         }
\end{align*}
Integrating this inequality along $\gamma$ we again get
\begin{align*}
\Ba{ \frac{\partial^2 h_k}{\partial r_1 \partial r_2}(p) - \frac{\partial^2 h_k}{\partial r_1 \partial r_2}(q)  } \le &  \int_{\gamma} \big |\nabla_{g_{\bbeta}} \frac{\partial^2 h_k}{\partial r_1 \partial r_2}\big| 
\le  C d^{\ibone - 1} \tau^{-k ( \ibone - 1  )}\omega(\tau^k).
\end{align*}

\medskip

\noindent $ \blacktriangleright$ {\bf Case 3:} here we assume $\min(r_p,r_q) = r_p\ge 2d$, but $\ell\ge k_{p} + 1$. The case when $k\le k_p$ can be dealt with by the same argument as in {\bf Case 2}, so we omit it and  only consider the cases when $k_{p}+1\le k\le \ell$, so here  $r_2(p)\ge r_1(p)\ge \tau^k\ge \tau^\ell > 8 d$ so $r_1(\gamma(t))\ge \frac{7}{8}\tau^k$ and $r_2(\gamma(t))\ge \frac 78 \tau^k$ for any point $\gamma(t)$ in the geodesic $\gamma$. By triangle inequality it follows that $\gamma\subset \frac 13 \hat B_k(p) = B_{\bbeta}(p, \tau^k/ 3)$.

As before, we can introduce smooth coordinates $w_1 = z_1^{\beta_1}$ and $w_2 = z_2^{\beta_2}$, and $g_{\bbeta}$ becomes the standard smooth Euclidean metric $g_{\mathbb C^n}$ under these coordinates. Moreover, $h_k$ are the usual Euclidean harmonic function $\Delta_{g_{\mathbb C^n}} h_k = 0$ on $\hat B_k(p)$. By the standard derivative estimates we have
\begin{equation*}
\sup_{\frac{1}{2}\hat B_k(p)} \bk{ \ba{ D^3_{w_1,w_2} h_k } + \ba{D' (D^2_{w_1, w_2}) h_k }  }\le C \tau^{-k}\omega(\tau^k).
\end{equation*}
From the equation
\begin{align*}
\frac{\partial^2 h_k}{\partial r_1 \partial r_2} = & ~\frac{w_1 w_2}{r_1 r_2} \frac{\partial^2 h_k}{\partial w_1 \partial w_2}  + \frac{\bar w_1 w_2}{r_1 r_2} \frac{\partial^2 h_k}{\partial \bar w_1 \partial w_2}  + \frac{w_1 \bar w_2}{r_1 r_2} \frac{\partial^2 h_k}{\partial w_1 \partial\bar  w_2}  + \frac{\bar w_1 \bar w_2}{r_1 r_2} \frac{\partial^2 h_k}{\partial\bar  w_1 \partial\bar w_2} 
\end{align*}
we see that for $i=1,2$
\begin{equation*}
\sup_{\frac{1}{2}\hat B_k(p)}\Ba{ \frac{\partial}{\partial w_i } \bk{ \frac{\partial^2 h_k}{\partial r_1 \partial r_2}   }  }\le \frac{C}{r_i}\omega(\tau^k) + C\tau^{-k}\omega(\tau^k),\quad \sup_{\frac{1}{2}\hat B_k(p)} \Ba{ D'\bk{ \frac{\partial^2 h_k}{\partial r_1 \partial r_2}  }  }\le C \tau^{-k}\omega(\tau^k). 
\end{equation*}
From this we see that $$\sup_{\gamma}  \Ba{ \nabla_{g_{\bbeta}} \frac{\partial^2 h_k}{\partial r_1 \partial r_2}  } \le \sup_{\gamma}\bk{ C\tau^{-k}\omega(\tau^k) + \frac{C}{r_1} \omega(\tau^k) + \frac{C}{r_2}\omega(\tau^k)}\le C \tau^{-k}\omega(\tau^k).$$
Integrating along $\gamma$ we see that 
\begin{equation*}
\Ba { \frac{\partial^2 h_k}{\partial r_1 \partial r_2} (p) - \frac{\partial^2 h_k}{\partial r_1 \partial r_2} (q)   }\le \int_{\gamma}  \Ba{ \nabla_{g_{\bbeta}} \frac{\partial^2 h_k}{\partial r_1 \partial r_2}  } \le C d \tau^{-k}\omega(\tau^k)\le C d^{\ibone - 1} \tau^{-k(\ibone - 1)} \omega(\tau^k).
\end{equation*}

Combining the estimates in all three cases, we finish the proof of Lemma \ref{lemma 2.21}.

\end{proof}

By Lemma \ref{lemma 2.21} it follows that
{\small
\begin{equation}\label{eqn:L2 estimate}
L_2 = \Ba{\frac{\partial^2 u_\ell}{\partial r_1 \partial r_2}(p) - \frac{\partial^2 u_\ell}{\partial r_1 \partial r_2}(q)   } \le \Ba{\frac{\partial^2 u_2}{\partial r_1 \partial r_2}(p) - \frac{\partial^2 u_2}{\partial r_1 \partial r_2}(q)   } + C d^{\frac{1}{\beta_{\max}} - 1} \sum_{k=3}^\ell \tau^{-k (\frac{1}{\beta_{\max}} - 1) } \omega(\tau^k). 
\end{equation} }
To finish the proof, it suffices to estimate the first term on the RHS of the above equation. Recall we assume $u_2$ is a $g_{\bbeta}$-harmonic function defined on the ball $\hat B_2(p)$, which is centered at $p_{1,2}\in\sS_1\cap\sS_2$ and has radius $2\tau^2$. $u_2$ satsifies the $L^\infty$ estimate by maximum principle: there exists some $C=C(n)>0$ such that
\begin{equation}\label{eqn:u2 L infty}\| u_2\|_{L^\infty( \hat B_k(p)  )} \le C ( \| u\|_{L^\infty(B_{\bbeta}(0,1))} + \omega(\tau^2)   ). \end{equation}
Recall that the proofs of the estimates in Lemmas \ref{lemma 2.16} - \ref{lemma 2.20} in the case when $k\le k_{2,p}$ work for any $g_{\bbeta}$-harmonic functions defined on suitable balls, and we can repeat the arguments there  replacing the $L^\infty$-estimates of $h_k$ that $\| h_k\|_{L^\infty}\le C\tau^{2k}\omega(\tau^k)$, by the $L^\infty$-estimate of $u_2$ as in \eqref{eqn:u2 L infty}, to get similar estimates as in those lemmas, which we will not repeat here. Given these estimates, we can repeat the proof of Lemma \ref{lemma 2.21} to prove the following estimates
\begin{equation*}
\Ba{\frac{\partial^2 u_2}{\partial r_1 \partial r_2} (p) - \frac{\partial^2 u_2}{\partial r_1 \partial r_2}(q)  } \le C d^{\frac{1}{\beta_{\max} } - 1} ( \| u\|_{L^\infty(B_{\bbeta}(0,1))}+ \omega(\tau^2) )  .
\end{equation*}
This inequality, combined with \eqref{eqn:L2 estimate} give the final estimate of the term $L_2$, that
\begin{equation}\label{eqn:L2 final}
L_2 \le C d^{\frac{1}{\beta_{\max}} -1  }   \| u\|_{L^\infty(B_{\bbeta}(0,1))}  +  C d^{\frac{1}{\beta_{\max}} -1  }  \sum_{k=2}^\ell \tau^{-k( \frac{1}{\beta_{\max}} - 1   )} \omega(\tau^k).
\end{equation}



By Lemma \ref{lemma 2.14}, Lemma \ref{lemma 2.15} and the estimate \eqref{eqn:L2 final} for $L_2$, we are ready to prove the following estimate (see the equation \eqref{eqn:newly added 1})
\begin{prop}\label{prop:2.3}
For the given $p, q\in B_{\bbeta}(0,1/2)\backslash \sS$, there is a constant $C=C(n,\bbeta)>0$ such that 
\begin{equation*}
\Ba{ \frac{\partial ^2 u}{\partial r_1 \partial r_2}(p) - \frac{\partial ^2 u}{\partial r_1 \partial r_2}(q)   } \le C\bk{ d^{\frac{1}{\beta_{\max}} - 1} \| u\|_{L^\infty(B_{\bbeta}(0,1))} + \int_0^d \frac{\omega(r)}{r}dr + d^{ \frac{1}{\beta_{\max}}   - 1}    \int_d^1 \frac{\omega(r)}{r^{1/\beta_{\max}} }dr   }.
\end{equation*}

\end{prop}
\begin{proof}
From  Lemma \ref{lemma 2.14}, Lemma \ref{lemma 2.15} and the estimate \eqref{eqn:L2 final} for $L_2$, we have
\begin{align*}
& ~\Ba{ \frac{\partial ^2 u}{\partial r_1 \partial r_2}(p) - \frac{\partial ^2 u}{\partial r_1 \partial r_2}(q)   }\\
 \le & ~ C \bk{ d^{\frac{1}{\beta_{\max}}-1} \| u\|_{L^\infty(B_{\bbeta}(0,1))} + d^{\frac{1}{\beta_{\max}}-1} \sum_{k=2}^\ell \tau^{-k ( \frac{1}{\beta_{\max}}-1  )} \omega(\tau^k) + \sum_{k=\ell}^\infty \omega(\tau^k)           }\\
\le & ~C\bk{ d^{\frac{1}{\beta_{\max}} - 1} \| u\|_{L^\infty(B_{\bbeta}(0,1))} + \int_0^d \frac{\omega(r)}{r}dr + d^{ \frac{1}{\beta_{\max}}   - 1}    \int_d^1 \frac{\omega(r)}{r^{1/\beta_{\max}} }dr   },
\end{align*}
where the last inequality follows from the fact that $\omega(r)$ is monotonically increasing.

\end{proof}
Finally, we remark that the estimates for the other operators in \eqref{eqn:operators} follows similarly, so we omit the detailed proof and just state that the estimates are the same as the estimates for $\frac{\partial^2 u}{\partial r_1 \partial r_2}$ as in Proposition \ref{prop:2.3}.

\subsection{Non-flat conical K\"ahler metrics}\label{section 3.5}
In this section, we will consider the Schauder estimates for general conical K\"ahler metrics on $B_\bbeta(0,2)\subset \mathbb C^n$ with cone angle $2\pi\bbeta$ along the simple normal crossing hyper-surface $\sS$. Let $\omega$ be such a metric. By definition, there exists a constant $C\ge 1$ such that
\begin{equation}\label{eqn:gen cone}C^{-1}\omega_\bbeta \le \omega \le C \omega_\bbeta,\quad \text{in }B_\bbeta(0,2)\backslash \sS,\end{equation}
where $\omega_\bbeta$ is the standard flat conical metric as before. Since $\omega$ is closed and $B_\bbeta(0,2)$ is simply connected, we can write $\omega = \ddbar \phi$ for some strictly pluri-subharmonic function $\phi$. By elliptic regularity, $\phi$ is Holder continuous under the Euclidean metric on $B_\bbeta(0,2)$. 

We fix an $\alpha\in(0, \min\{ \frac{1}{\beta_{\max}} - 1, 1  \})$.
\begin{defn}\label{defn:3.1}
We say $\omega = g$  a $C^{0,\alpha}_\bbeta$ K\"ahler metric on $B_\bbeta(0,2)$ if it satisfies \eqref{eqn:gen cone} and the K\"ahler potential $\phi$ of $\omega$ belongs to $C^{2,\alpha}_\bbeta(B_\bbeta(0,2))$.
\end{defn}

\medskip

Our interest is to study the Laplacian equation 
\begin{equation}\label{eqn:lap general}
\Delta_g u = f,\quad \text{in } B_\bbeta(0,1),
\end{equation}
where $f\in C^{0,\alpha}_\bbeta(\overline{B_\bb(0,1)})$ and $u\in C^{2,\alpha}_{\bbeta}$. We will prove the following scaling-invariant interior Schauder estimates, and the proof follows closely from that of Theorem 6.6 in \cite{GT}. So we mainly focus on the differences. 

\begin{prop}\label{prop:invariant}
There exists a constant $C=C(n,\bb, \| g\|_{C^{0,\alpha}_\bb}^*)>0$ such that if $u\in C^{2,\alpha}_\bb( B_\bb(0,1)  )$ satisfies the equation \eqref{eqn:lap general}, then
\begin{equation}\label{eqn:scaling invariant}
\| u\|_{C^{2,\alpha}_\bb(B_\bb(0,1))}^*\le C \xk{ \| u\|_{C^0(B_\bb(0,1))} +  \| f\|_{C^{0,\alpha}_\bb (B_\bb(0,1))  }^{(2)}     }.
\end{equation}
\end{prop}
\begin{proof}
Given any points $x_0\neq y_0\in B_\bb(0,1)$, assume $d_{x_0} = \min(d_{x_0}, d_{y_0}  )$ (recall $d_x = d_\bb( x, \partial B_\bb(0,1)  )$). Let $\mu\in (0,1/4)$ be a small number to be determined later. Denote $d=\mu d_{x_0}$ and $B:= B_\bb(x_0, d)$, and $\frac 1 2 B := B_\bb(x_0, d/2)$.

\smallskip

\noindent $\bullet$ {\bf Case 1.} $d_\bb(x_0,y_0)< d/2$. 

\smallskip

{\em Case 1.1.} $B_\bb(x_0, d)\cap \sS = \emptyset$. On $B_\bb(x_0, d)$ we can introduce the smooth complex coordinates $\{w_1 = z_1^{\beta_1}, w_2 = z_2^{\beta_2}, z_3,\ldots, z_n\}$, under which $g_\bb$ becomes the Euclidean one and the components of $g$ become $C^\alpha$ in the usual sense. The equation \eqref{eqn:lap general} has $C^\alpha$ leading coefficients and we can apply Theorem 6.6 in \cite{GT} to conclude that (the following inequality is understood in the new coordinates)
\begin{equation}\label{eqn:GT 1}
[u]^*_{C^{2,\alpha}( B)} \le C\xk{ \| u\|_{C^0(B)} + \| f\|_{C^{0,\alpha} (B)   }^{(2)}    }.
\end{equation}
Recall $T$ denotes  the  second order operators appearing in \eqref{eqn:operator T}.
Let $D$ denote the ordinary first order operators in $\{w_1,w_2,z_3,\ldots, z_n\}$. Then we calculate
\begin{align*}
|Tu(x_0) - Tu(y_0)   |&  \le | D^2 u (x_0) - D^2 u(y_0)  | + \frac{d_\bb(x_0,y_0)}{d} ( |D^2 u(x_0)| + |D^2 u(y_0)|   ) \\
& \le \frac{4 d_\bb(x_0,y_0)^\alpha}{ d^{2+\alpha}  } [u]^*_{C^{2,\alpha}(B)} + \frac{4 d_\bb(x_0,y_0)}{d^3} [u]^*_{C^2(B)}\\
\text{(by interpolation inequality)}& \le \frac{8 d_\bb(x_0,y_0)^\alpha}{ d^{2+\alpha}  } [u]^*_{C^{2,\alpha}(B)} +C  \frac{d_\bb(x_0,y_0)^\alpha}{ d^{2+\alpha} } \| u\|_{C^0(B)}.
\end{align*}
Then we get
\begin{equation}\label{eqn:case 1.1}
d_{x_0}^{2+\alpha}\frac{ |Tu(x_0) - Tu(y_0)   |  }{d_\bb(x_0,y_0)^\alpha   } \le \frac{C }{ \mu^{2+\alpha}  } \| f\|^{(2)}_{ C^{0,\alpha}_\bb( B  )  } +  \frac{C}{ \mu ^{2+\alpha} } \| u\|_{C^0(B)}.
\end{equation}

\smallskip

{\em Case 1.2:} $B_\bb(x_0, d)\cap \sS \neq \emptyset$. Let $\hat x_0\in \sS$ be the nearest point of $x_0$ to $\sS$. We consider the balls $\hat B: = B_\bb( \hat x_0, 2d  )$  which is contained in $B_\bb(0,1)$ by triangle inequality. As in \cite{D}, we introduce a (non-holomorphic) basis of $T^*_{1,0}(\mathbb C^n\backslash \Ss)$ as 
$$\big\{\epsilon_j: = dr_j + \sqrt{-1} \beta_j r_j d\theta_j, dz_k\big\}_{j=1,2; \,k=3,\ldots, n},$$
and the dual basis of $T_{1,0}( \mathbb C^n\backslash \sS  )$:
$$\big\{\gamma_j := \frac{\partial}{\partial r_j} - \sqrt{-1}\frac{1}{\beta_j r_j} \frac{\partial}{\partial\theta_j}, \frac{\partial}{\partial z_k} \big\}_{j=1,2; \,k=3 ,\ldots, n   }.$$ We can write the $(1,1)$-form $\omega$ in the basis $\{\epsilon_j\wedge \bar \epsilon_k, \epsilon_j\wedge d\bar z_k, dz_k\wedge \bar \epsilon_j, dz_j\wedge d\bar z_k\}$ as
\begin{equation}\label{eqn:expand cone} \omega = g_{\epsilon_j \bar \epsilon_k} \epsilon_j\wedge \bar \epsilon_k + g_{\epsilon_j \bar k} \epsilon_j\wedge d\bar z_k + g_{k\bar \epsilon_j} dz_k\wedge \bar \epsilon_j + g_{j\bar k} dz_j \wedge d\bar z_k,  \end{equation}
where 
\begin{equation}\label{eqn:second operators}
g_{\epsilon_j\bar \epsilon_k} = \ddb\phi(\gamma_j,\bar \gamma_k), \, g_{\epsilon_j \bar k} =\ddb\phi( \gamma_j, \frac{\partial}{\partial \bar z_k}),\, g_{k\bar \epsilon_j} = \ddb(\frac{\partial}{\partial z_k},\bar \gamma_j), \, g_{k\bar j} = \frac{\partial^2}{\partial z_k\partial\bar z_j}\phi.
\end{equation}
We remark that all the second order derivatives of $\phi$ appearing in \eqref{eqn:second operators} are linear combination of $|z_j|^{2-2\beta_j}\frac{\partial^2}{\partial z_j\partial z_{\bar j}}$  $N_jN_k$ ($j\neq k$), $N_j D'$ and $(D')^2$, which are studied in Theorem \ref{thm:main 1}. The standard metric $\omega_\bbeta$ becomes the identity matrix under the basis above for $(1,1)$-forms.
If $\omega$ is $C^{0,\alpha}_\bbeta$, all the coefficients in the expression of $\omega$ in \eqref{eqn:expand cone} are $C^{0,\alpha}_\bbeta$ continuous and the cross terms $g_{\epsilon_j\bar \epsilon_k}$ with $j\neq k$, $g_{\epsilon_j \bar k}$ tend to zero when approaching the corresponding singular sets $\sS_j$ or $\sS_k$. Moreover, the limit of $g_{j\bar k}dz_j\wedge d\bar z_k$ as tending to $\sS_1\cap\ldots\cap \Ss_p$ defines a K\"ahler metric on it. Rescaling or rotating the coordinates if necessary we may assume at $\hat x_0\in \sS$, $g_{\epsilon_j\bar \epsilon_j}(\hat x_0) = 1$, $g_{j\bar k}(\hat x_0)=\delta_{jk}$ and the cross terms vanish at $\hat x_0$. Let $\omega_\bbeta$ be the standard cone metric under these new coordinates near $\hat x_0$, and we can write the  equation \eqref{eqn:lap general} as 
$$\Delta_g u (z)= \Delta_{g_\bbeta} u(z) + \eta(z). i\partial\bar \partial u(z) = f(z), \quad\forall ~ z\not\in \sS$$
for some Hermtian matrix $\eta(z) = (\eta^{j\bar k})_{j,k=1}^n$, $\eta^{j\bar k} = g^{j\bar k}(z) - g_\bbeta^{j\bar k}$. It is not hard to see the term $ \eta(z). i\partial\bar \partial u$ can be written as
\begin{equation}\label{eqn:eta}\sum_{j,k=1}^2 (g^{\epsilon_j\bar \epsilon_k} (z)- \delta_{jk}) u_{\epsilon_j\bar \epsilon_j} + 2 Re\sum_{1\le j\le 2, 3\le k\le n} g^{\epsilon_j \bar k} u_{\epsilon_j \bar k} + \sum_{j,k=3}^n ( g^{j\bar k}(z) - \delta_{jk} ) u_{j\bar k},  \end{equation}
and $g$ with the upper indices denotes the inverse matrix of $g$.  We consider the equivalent form of the equation \eqref{eqn:lap general} on $\hat B$, $$\Delta_{g_\bbeta} u = f - \eta. \ddb u=: \hat f, \quad u\in C^0(\hat B)\cap C^2(\hat B \backslash \sS).$$ Observe that $x_0,y_0\in B_\bb(\hat x_0, 3d/2)$ we can apply the scaled inequality \eqref{eqn:rescale Holder} of Theorem \ref{thm:main 1} to conclude that 
$$d^{2+\alpha} \frac{ |Tu(x_0) - Tu(y_0)  | }{d_\bb(x_0,y_0)^\alpha} \le  C \xk{ \| u\|_{C^0(\hat B)} + \| \hat f\|^{(2)}_{C^{0,\alpha}_\bb( \hat B  )}    } ,$$ thus
\begin{equation}\label{eqn:case 1.2}
d_{x_0}^{2+\alpha} \frac{ |Tu(x_0) - Tu(y_0)  | }{d_\bb(x_0,y_0)^\alpha} \le  \frac{C}{\mu^{2+\alpha}} \xk{ \| u\|_{C^0(\hat B)} + \| \hat f\|^{(2)}_{C^{0,\alpha}_\bb( \hat B  )}    } .
\end{equation}

\medskip

\noindent $\bullet$ {\bf Case 2.} $d_\bb(x_0,y_0)\ge d/2$. 
\begin{equation}\label{eqn:case 2.0}
d_{x_0}^{2+\alpha} \frac{ |Tu(x_0) - Tu(y_0)  | }{d_\bb(x_0,y_0)^\alpha} \le  4 d_{x_0}^{2+\alpha} \frac{ |Tu(x_0)| +| Tu(y_0)  | }{d ^\alpha}\le \frac{8}{\mu^\alpha} [u]_{C^{2}_\bb ( B_\bb(0,1)   )   }^* .
\end{equation}
Combining \eqref{eqn:case 1.1}, \eqref{eqn:case 1.2} and \eqref{eqn:case 2.0} we get
\eqsp{\label{eqn:case final}
d_{x_0}^{2+\alpha} \frac{ |Tu(x_0) - Tu(y_0)  | }{d_\bb(x_0,y_0)^\alpha} \le & \frac{8}{\mu^\alpha} [u]_{C^{2}_\bb ( B_\bb(0,1)   )   }^* +  \frac{C}{\mu^{2+\alpha}} \xk{ \| u\|_{C^0(\hat B)} + \| \hat f\|^{(2)}_{C^{0,\alpha}_\bb( \hat B  )}    }\\
& ~~ +  \frac{C }{ \mu^{2+\alpha}  } \| f\|^{(2)}_{ C^{0,\alpha}_\bb( B  )  } +  \frac{C}{ \mu ^{2+\alpha} } \| u\|_{C^0(B)}.
}
By definition it is easy to see that (we denote $B_\bb = B_\bb(0,1)$)
\begin{equation*}
\| f\|_{C^{0,\alpha }_\bb( B )}^{(2)} \le C \mu^2 \| f\|^{(2)} _{ C^0(B_\bb)   } + C \mu^{2+\alpha} [ f  ]_{C^{0,\alpha}_\bb(B_\bb)}^{(2)} \le \mu^2 \| f\|_{C^{0,\alpha}_\bb( B_\bb  )} ^{ (2)  }.  
\end{equation*}
We calculate 
\begin{align*}
\| \hat f\|_{C^{0,\alpha}_\bb(\hat B  )  }^{(2)}\le & \| \eta\|_{C^{0,\alpha}_\bb (\hat B)  }^{(0)   } \| Tu\|^{(2)}_{C^{0,\alpha}_\bb( \hat B )  } + \| f\|^{(2)}_{ C^{0,\alpha}_\bb (\hat B)  }\\
\le & C_0 [g]^{*}_{C^{0,\alpha}_\bb( B_\bb  )} \mu^\alpha \xk{ \mu^2 [u]^{*}_{C^{2}_\bb(B_\bb)} + \mu^{2+\alpha} [u]^*_{C^{2,\alpha}_\bb(B_\bb)   }      }   + \mu^2 \| f\|_{C^{0,\alpha}_\bb( B_\bb  )} ^{ (2)  }\\
\le & C_0 [g]^{*}_{C^{0,\alpha}_\bb( B_\bb  )} \mu^\alpha \xk{ C(\mu) \| u\|_{C^0(B_\bb)}+ 2 \mu^{2+\alpha} [u]^*_{C^{2,\alpha}_\bb(B_\bb)   }      }   + \mu^2 \| f\|_{C^{0,\alpha}_\bb( B_\bb  )} ^{ (2)  }.  
\end{align*}
\begin{equation*}
\frac{8}{\mu^\alpha}[u]^*_{C^2_\bb(B_\bb)} \le \mu^\alpha [u]^*_{C^{2,\alpha}_\bb(B_\bb)} + C(\mu) \| u\|_{C^0(B_\bb)}.
\end{equation*}
If we choose $\mu>0$ small such that $\mu^\alpha(2C_0 [g]^*_{C^{0,\alpha}_\bb(B_\bb)} + 1)\le 1/2$, then we get from \eqref{eqn:case final} and the inequalities above that
\begin{equation*}
d_{x_0}^{2+\alpha} \frac{ |Tu(x_0) - Tu(y_0)  | }{d_\bb(x_0,y_0)^\alpha}\le \frac 1 2 [u]^*_{C^{2,\alpha}_\bb(B_\bb)} + C(\mu)\xk{ \| u\|_{C^0(B_\bb)} + \| f\|^{(2)}_{C^{0,\alpha}_\bb (B_\bb)}   }.
\end{equation*}
Taking supremum over $x_0\neq y_0\in B_\bb(0,1)$ we conclude from the inequality above that
$$ [u  ]^*_{C^{2,\alpha}_\bb( B_\bb  )}\le C \xk{ \| u\|_{C^0( B_\bb  )} + \| f\|_{C^{0,\alpha}_\bb( B_\bb  )}^{(2)}    } . $$
Proposition \ref{prop:invariant} then follows from interpolation inequalities.

\end{proof}

\begin{remark}
It follows easily from the proof of Proposition \ref{prop:invariant} that the estimate \eqref{eqn:scaling invariant} also holds for metric balls $B_\bb(p,R)\subset B_\bb(0,1)$ whose center $p$ may not lie at $\sS$. 

\end{remark}

An immediate corollary to Proposition \ref{prop:invariant} is the following interior Schauder estimate.

\begin{corr}\label{prop 3.5}
Suppose $u$ satisfies the equation \eqref{eqn:lap general}. For any compact subset $K\Subset B_\bb(0,1)$, there exists a constant $C=C(n,\bb,K, \|g\|_{C^{0,\alpha}_\bb( B_\bb(0,1)   )})>0$ such that 
\begin{equation*}
\| u\|_{C^{2,\alpha}_\bb(K)}\le C \xk{ \|u\|_{C^{0}(B_\bb(0,1))} + \| f \|_{C^{0,\alpha}_\bb ( B_\bb(0,1)  )   }  }.
\end{equation*}
\end{corr}

\medskip

\noindent Next we will show that the equation \eqref{eqn:lap general} admits a unique $C^{2,\alpha}_\bb$-solution for any $f\in C^{0,\alpha}_\bb(\overline{B_\bb(0,1)})$ and boundary value $\varphi\in C^0(\partial B_\bb(0,1))$. We will follow the argument in Section 6.5 in \cite{GT}. In the following we will write $B_\bb= B_\bb(0,1)$ for simplicity.

\begin{lemma}\label{lemma 3.22}
Let $\sigma\in (0,1)$ be a given number. Suppose $u\in C^{2,\alpha}_\bb(B_\bb)$ solves \eqref{eqn:lap general} and $\| u\|^{(-\sigma)}_{C^0(B_\bb)}<\infty$ and $\| f\|_{C^{0,\alpha}_\bb( B_\bb  )   }^{(2-\sigma)}<\infty$. Then there exists a $C=C(n, \bb, \alpha, g ,\sigma  )>0$ such that 
$$ \| u\|_{ C^{2,\alpha}_\bb (B_\bb)  }^{(-\sigma)} \le C \xk{ \| u\|^{(-\sigma)} _{ C^{0}(B_\bb)   }  + \| f\|^{(2 -\sigma)}_{ C^{0,\alpha}_\bb(B_\bb)    }   }.   $$
\end{lemma}
\begin{proof}

Given the estimates in Proposition \ref{prop:invariant}, the proof is identical to that of Lemma 6.20 in \cite{GT}. So we omit the details.
\end{proof}

\smallskip

\begin{lemma}\label{lemma 3.23}
Let $u\in C^{2}_\bb(B_\bb   )\cap C^0(\overline{B_\bb})$ solve the equation $\Delta_g u = f$ and $u\equiv 0 $ on $\partial B_\bb$. For any $\sigma\in (0,1)$, there exists a constant $C=C(n,\bb, \sigma, g)>0$ such that 
$$ \| u\|_{C^0(B_\bb)}^{ (-\sigma)  } =     \sup_{x\in B_\bb} d_x^{-\sigma} |u(x)  |\le C \sup_{x\in B_\bb} d_x^{2-\sigma} |f(x)  | = C \| f\|_{C^0(B_\bb)}^{(2-\sigma)}   ,$$
where as before $d_x = d_\bb(x, \partial B_\bb)$.
\end{lemma}
\begin{proof}
Consider the function $w_1 = (1 - d_\bb^2)^\sigma$ where $d_\bb(x) = d_\bb(x,0)$. We calculate 
\begin{align*}
\Delta_g w_1 & = \sigma ( 1 - d_\bb^2 )^{\sigma - 2} \xk{ - (1 - d_\bb^2) \tr_g g_\bb - (1-\sigma) \abs{\nabla d_\bb^2  }_{g}     }\\
&\le \sigma ( 1 - d_\bb^2 )^{\sigma - 2} \xk{- C^{-1} (1 - d_\bb^2) - 4 C^{-1}d_\bb^2 (1-\sigma)     }\\
&\le -c_0  \sigma ( 1 - d_\bb^2 )^{\sigma - 2}.
\end{align*}
Take a large constant $A>1$ such that for $w = Aw_1$
$$\Delta_g w\le - (1-d_\bb)^{\sigma - 2} \le - |f|/N,\quad\text{in }B_\bb,$$
where $N = \sup_{x\in B_\bb} d_x^{2-\sigma} |f(x)| = \sup_{x\in B_\bb} (1-d_\bb(x))^{2-\sigma} |f(x)|$. Hence $\Delta_g( N w \pm u   )\le 0$ and from the definition of $w$ we also have $w|_{\partial B_\bb} \equiv 0$, by maximum principle we obtain that $| u(x) |\le N w\le C N (1-d_\bb(x))^{\sigma} = CN d_x^\sigma$, hence the lemma is proved.
\end{proof}

\begin{prop} Given any function $f\in C^{0,\alpha}_\bb(\overline{B_\bb})$, 
the Dirichlet problem $\Delta_g u = f$ in $B_\bb$ and $u\equiv 0$ on $\partial B_\bb$ admits a unique solution $u\in C^{2,\alpha}_{\bb}(B_\bb)\cap C^0(\overline{B_\bb})$.
\end{prop}
\begin{proof}
The proof of this proposition is almost identical to that of Theorem 6.22 in \cite{GT}.  For completeness, we provide the detailed argument.   Fix a $\sigma\in (0,1)$.  We define a family of operators $\Delta_t = t \Delta_g + (1-t) \Delta_{g_\bb}$ and it is straightforward to see that $\Delta_t$ is associated to some cone metric which also satisfies \eqref{eqn:gen cone}. We    study the Dirichlet problem
\begin{equation}\label{eqn:t 1}\Delta_t u_t = f,\text{ in }B_\bb, \quad u_t\equiv 0 \text{ on }\partial B_\bb. \tag{$*_t$} \end{equation}
Equation ($*_0$) admits a unique solution $u_0\in C^{2,\alpha}_\bb(B_\bb)\cap C^0(\overline{B_\bb})$ by Proposition \ref{prop:3.2 new}. By Theorem 5.2 in \cite{GT}, in order to apply the continuity method to solve ($*_1$), it suffices to show $\Delta_t^{-1}$ defines a bounded linear operator between some Banach spaces. More precisely, define
$$\mathcal B_{1}: = \big\{ u\in C^{2,\alpha}_\bb(B_\bb) ~|~ \| u\|_{C^{2,\alpha}_\bb(B_\bb)}^{(-\sigma)} <\infty \big\},$$
$$\mathcal B_2: = \big\{ f\in C^{0,\alpha}_{\bb}(B_\bb) ~|~ \| f\|^{(2 -\sigma)}_{ C^{0,\alpha}_\bb (B_\bb)  } < \infty   \big\}.$$
By definition any $u\in \mathcal B_1$ is continuous on $\overline{B_\bb}$ and $u = 0$ on $\partial B_\bb$. By Lemmas \ref{lemma 3.22} and \ref{lemma 3.23}, we have
$$\| u\|_{\mathcal B_1} = \| u\|_{C^{2,\alpha}_\bb(B_\bb)}^{(-\sigma)} \le C \| f\|_{ C^{0,\alpha}  _\bb(B_\bb)}^{(2-\sigma)} = C \| \Delta_t u\|_{\mathcal B_2},    $$
for some constant $C$ independent of $t\in [0,1]$. Thus ($*_1$) admits a solution $u\in \mathcal B_1$. 
\end{proof}

\begin{corr}\label{corollary 3.2}
For any given $\varphi\in C^0(\partial  B_\bb)$ and $f\in C^{0,\alpha}_\bb(\overline{B_\bb})$, the Dirichlet problem 
\begin{equation}\label{eqn:cor 3.1}
\Delta_g u = f, \text{ in }B_\bb,\text{ and }
u = \varphi, \text{ on }\partial B_\bb,
\end{equation}admits a unique solution $u\in C^{2,\alpha}_\bb(B_\bb)\cap C^0(\overline{B_\bb})$.
\end{corr}
\begin{proof}
We may extend $\varphi$ continuously to $B_\bb$ and assume $\varphi\in C^0(\overline{B_\bb})$. Take a sequence of functions $\varphi_k\in C^{2,\alpha}_\bb(\overline{B_\bb})\cap C^0({\overline{B_\bb}})$ which converges uniformly to $\varphi$ on $\overline{B_\bb}$. The Dirichlet problem $\Delta_g v_k = f - \Delta_g \varphi_k$ in $B_\bb$ and $v_k = 0$ on $\partial B_\bb$ admits a unique solution $v_k\in C^{2,\alpha}_\bb(B_\bb)\cap C^0(\overline{B_\bb})$. Thus the function $u_k: = v_k + \varphi_k\in C^{2,\alpha}_\bb$ satisfies $\Delta_g u_k = f$ in $B_\bb$ and $u_k = \varphi_k$ on $\partial B_\bb$.  $u_k$ is uniformly bounded in $C^0(\overline{B_\bb})$ by maximum principle. Corollary \ref{prop 3.5} gives uniformly $C^{2,\alpha}_\bb(K)$-bound on any compact subset $K\Subset B_\bb$. Letting $k\to\infty$ and $K\to B_\bb$, by a diagonal argument and up to a subsequence $u_k\to u\in C^{2,\alpha}_{\bb} (B_\bb)  $. On the other hand, from $\Delta_g( u_k - u_l ) = 0$ we see that $\{u_k\}$ is a Cauchy sequence in $C^0(\overline{B_\bb})$ thus $u_k$ converges uniformly to $u$ on $\overline{B_\bb}$. Hence $u\in C^0(\overline{B_\bb})$ and satisfies  the equation \eqref{eqn:cor 3.1}.

\end{proof}

\begin{corr}\label{corollary 3.3}
Given $f\in C^{0,\alpha}_\bb({B_\bb})$, suppose $u$ is a weak solution to the equation $\Delta_g u = f$ in the sense that 
$$\int_{B_\bb} \innpro{ \nabla u,\nabla \varphi  }\omega_g^n = -\int_{B_\bb} f \varphi \omega_g^n,\quad \forall \varphi\in H^1_0(B_\bb),    $$
then $u\in C^{2,\alpha}_\bb(B_\bb  )$.
\end{corr}
\begin{proof}
We first observe that the Sobolev inequality \eqref{eqn:SOB} also holds for the metric $g$, since $g$ is equivalent to $g_\bb$. The metric space $(B_\bb, g)$ also has maximal volume growth/decay, so we can apply the same proof of De Giorgi-Nash-Moser theory (\cite{HL}) to conclude that $u$ is continuous in $B_\bb$. The standard elliptic theory implies that $u\in C^{2,\alpha}_{loc}(B_\bb\backslash \sS  )$. For any $r\in (0,1)$, by Corollary \ref{corollary 3.2}, the Dirichlet problem $\Delta_g \tilde u = f$ in $B_\bb(0,r)$, $\tilde u = u$ on $\partial B_\bb(0,r)$ admits a unique solution $\tilde u\in C^{2,\alpha}_\bb( B_\bb(0,r) )\cap C^0(\overline{ B_\bb(0,r)  })$. Then $\Delta_g(u - \tilde u) = 0$ in $B_\bb(0,r)$ and $u-\tilde u = 0$ on $\partial B_\bb(0,r)$. By maximum principle we get $u = \tilde u$ in $B_\bb(0,r)$, so we conclude $u\in C^{2,\alpha}_{\bb}(B_\bb(0,r))$. Since $r\in (0,1)$ is arbitrary, we get $u\in C^{2,\alpha}_\bb(B_\bb)$.

\end{proof}

\begin{corr}\label{corr 3.1}
Let $X$ be a compact K\"ahler manifold and $D=\sum_j D_j$ be a divisor with simple normal crossings. Let $g$ be a conical K\"ahler metric with cone angle $2\pi\bb$ along $D$. Suppose $u\in H^1(g)$ is a weak solution to the equation $\Delta_{g} u = f$ in the sense that
$$\int_X \innpro{\nabla u, \nabla \varphi} \omega_g^n = - \int_X f \varphi \omega_g^n,\quad \forall \varphi\in C^1(X)$$ for some $f\in C^{0,\alpha}_{\bb}(X)$. Then $u\in C^{2,\alpha}_\bb(X)\cap C^0(X)$ and  there exists a constant $C=C( n,\bb, g, \alpha)$ such that
$$\| u\|_{C^{2,\alpha}_\bb (X)  }\le C\xk{ \| u\|_{C^0(X)} + \| f\|_{C^{0,\alpha}_\bb( X )}   }.$$
\end{corr}
\begin{proof}
We can choose finite covers of $D$, $\{B_a\}$, $\{B_a'\}$ with $B_a'\Subset B_a$ and centers at $D$. By assumption $u$ is a weak solution to the equation $\Delta_g u = f$ in each $B_a$, then by Corollary \ref{corollary 3.3} we conclude that $u\in C^{2,\alpha}_\bb(B_a)$ for each $B_a$. On $X\backslash \sS$, the metric $g$ is smooth so standard elliptic theory implies that $u\in C^{2,\alpha}_{loc}(X\backslash \sS  )$.  Since $\{B_a\}$ covers $D$, $u\in C^{2,\alpha}_\bb(X)$. 

  We can apply Corollary \ref{prop 3.5} to obtain that for some constant $C>0$
$$ \| u\|_{C^{2,\alpha}_\bb (B_a')} \le C\xk{ \| u\|_{C^0(B_a)} + \| f\|_{C^\alpha_\bb( B_a )}}.  $$
On $X\backslash \cup_a\{B_a'\}$ the metric $g$ is smooth, the usual Schauder estimates apply. We finish the proof of the Corollary by the definition of  $C^{2,\alpha}_\bb(X)$ (c.f. Definition \ref{defn:2.8}).
\end{proof}

\begin{remark}
Let $(X,D,g)$ be as in Corollary \ref{corr 3.1}. It is easy to see by variational method  weak solutions to $\Delta_g u = f$ always exist for any $f\in L^2(X,\omega_g^n)$ satisfying $\int_X f\omega_g^n = 0$. 
\end{remark}

\section{Parabolic estimates}
In this section, we will study the heat equation with background metric $\omega_\bb$ and prove the Schauder estimates for such solution $u\in C^0(\qq_\bb)\cap \C^{2,1}(\qq_\bb^\#)$ to the equation
\begin{equation}\label{eqn:para main}
\frac{\partial u}{\partial t} = \Delta_{g_\bb} u + f,
\end{equation} 
for a function $f\in \C^0(\qq_\bb)$ with some better regularity.

\subsection{Conical heat equations} In this section, we will show that for any $\varphi\in \C^0(\partial_{\pp} \qq_\bb)$, the Dirichlet problem \eqref{eqn:para Diri} admits a unique $\C^{2,1}(\mathcal Q_\bb^\#)\cap \cC ^0(\overline{\qq_\bb})$-solution in $\qq_\bb$. We first observe that a maximum principle argument yields the uniqueness of the solution.

Suppose $u\in\C^{2,1}(\mathcal Q_\bb^\#)\cap \cC ^0(\overline{\qq_\bb})$ solves the Dirichlet problem 
\begin{equation}\label{eqn:para Diri}
\left\{\begin{aligned}
\frac{\partial u}{\partial t} = \Delta_{g_\bb} u,\text{~ in ~} \mathcal Q_\bb\\
u = \varphi,\text{~ on ~} \partial_{\mathcal P} \mathcal Q_\bb,
\end{aligned}
\right.
\end{equation}
for some given continuous function $\varphi \in \C^0(\partial_{\mathcal P} \mathcal Q_\bb) $. It follows from maximum principle as in Lemma \ref{lemma:MP} that  
\begin{equation}\label{eqn:para mp}\inf_{ \partial_{\mathcal P} {\mathcal Q_\bb} }u\le  \inf_{\mathcal Q_\bb} u\le \sup_{{\mathcal Q_\bb}} u \le \sup_{\partial_{\mathcal P} {\mathcal Q_\bb}} u . \end{equation}
So  the $\C^{2,1}(\mathcal Q_\bb^\#)\cap \cC ^0(\overline{\qq_\bb})$-solution to \eqref{eqn:para Diri} is unique, if exists.

Now we prove the existence of solutions to \eqref{eqn:para Diri}. As before, we will use an approximation argument.  Let $g_\epsilon$ be the smooth approximation metrics in $B_\bb$ as defined in \eqref{eqn:para app}. Let $u_\epsilon$ be the $\C^{2,1}(\qq_\bb)\cap \C^0(\overline{\qq_\bb})$-solution to the equation
\begin{equation}\label{eqn:para eps}
\frac{\partial u_\epsilon}{\partial t} = \Delta_{g_\epsilon} u_\epsilon,\text{~ in ~} \mathcal Q_\bb, ~~\text{and }
u_\epsilon = \varphi,\text{~ on ~} \partial_{\mathcal P} \mathcal Q_\bb,
\end{equation}
\subsubsection{Estimates of $u_\ep$}
We first recall  the Li-Yau gradient estimates (\cite{LY, SY}) for  positive solutions to the heat equations.
\begin{lemma}\label{lemma LY}
Let $(M,g)$ be a complete manifold with $\ric(g)\ge 0$, and $B(p,R)$ be the geodesic ball with center $p\in M$ and radius $R>0$. Let $u$ be a positive solution to the heat equation $\partial_t u - \Delta_g u = 0$ on $B(p,R)$, then there exists $C=C(n)>0$ such that for all $t>0$, %
$$\sup_{B(p,2R/3)} \bk{ \frac{\abs{\nabla u}}{u^2} -  \frac{2\dot u_t}{u}  }\le \frac{C}{R^2} + \frac{2n}{t} ,$$
where $\dot u_t = \frac{\partial u}{\partial t}$.
\end{lemma}
By considering the functions $u_\epsilon - \inf u_\epsilon$ and $\sup u_\epsilon - u_\epsilon$, from the Lemma \ref{lemma LY}, we see that
there exists a constant $C=C(n)>0$ such that  for any $R\in (0,1)$ and $t\in (0, R^2)$
\begin{equation}\label{eqn:para grad 1}
\sup_{B_{g_\ep}(0, 2R/3)} |\nabla{} u_\epsilon|^2_{g_\epsilon}\le C\bk{ \frac{1}{R^2} + \frac{1}{t}  } (\osc_R u_\epsilon)^2,
\end{equation}
and 
\begin{equation}\label{eqn:para grad 2}
\sup_{B_{g_\ep}(0,2R/3)} | \Delta_{g_\epsilon} u_\epsilon|  = \sup_{B_{g_\ep}(0,2R/3)} \ba{ \frac{\partial u_\ep}{\partial t}  } \le C\bk{ \frac{1}{R^2} + \frac{1}{t}  } \osc_R u_\epsilon,
\end{equation}
where $\osc_R u_\epsilon: = \osc_{B_{g_\ep}(0,R)\times (0,R^2)   } u_\epsilon$ is the oscillation of $u_\epsilon$ in the cylinder $B_{g_\ep}(0,R)\times (0,R^2) $. Replacing $u_\epsilon$ by $u_\epsilon - \inf u_\epsilon$, we may assume $u_\epsilon>0$ and define $f_\epsilon = \log u_\ep$. Then we have
$$\frac{\partial f_\ep}{\partial t} = \Delta_{g_\ep} f_\ep + \abs{\nabla f_\ep}.$$ Let $\varphi(x) = \varphi( \frac{r(x)}{R}  )$ where $\varphi$ is a cut-off function equal to $1$ on $[0,3/5]$, $0$ on $[2/3, \infty)$, and satisfies the inequalities $|\varphi''|\le 10$ and  $(\varphi')^2 \le 10 \varphi$. $r(x)$ is the distance function under $g_\ep$ to the center $0$.  

\begin{lemma}\label{lemma 4.2}
There exists a constant $C=C(n)>0$ such that for any small $\epsilon>0$
$$\sup_{B_{g_\ep}(0, 3R/5)} | \Delta_i u_\ep |\le C \xk{ \frac{1}{t} + \frac{1}{R^2}  }\osc_R u_\ep,\quad \forall ~ t\in (0, R^2),$$
where we denote $\Delta_i u_\ep := (|z_i|^2 + \ep)^{1-\beta_i} \frac{\partial^2u_\ep}{\partial z_i \partial z_{\bar i}}$ for $i=1,\ldots, p$.
\end{lemma}
\begin{proof}
We only prove the case when $i=1$. 
We denote $F: = t\varphi ( -\Delta_1 f_\ep - 2 \dot f_\ep  )$, and we calculate
\begin{align*}
& ( \frac{\partial}{\partial t} - \Delta_{g_\ep}  ) (  -\Delta_1 f_\ep - 2 \dot f_\ep)\\
= & - \abs{\nabla_2 \nabla f_\ep} - \abs{\nabla_1\bar \nabla f_\ep  } - 2 \Re \innpro{\nabla f_\ep, \bar \nabla ( -\Delta_1 f_\ep - 2 \dot f_\ep  )   } - R_{1\bar 1 j\bar k} f_{\ep,\bar j} f_{\ep, k}\\
\le & - (- \Delta_1 f_\ep  )^2 - 2 \Re \innpro{\nabla f_\ep, \bar \nabla ( -\Delta_1 f_\ep - 2 \dot f_\ep  )   }.
\end{align*}
$F$ achieves its maximum at  a point $(p_0,t_0)$, where we may assume $F(p_0,t_0)>0$, otherwise we are done yet. In particular, $p_0\in B_{g_\ep}(0, 2r/3)$ by the definition of $\varphi$ and $t_0>0$. Then at $(p_0,t_0)$, we have 
\eqsp{\label{eqn:para new 1}
0& \le \hoep F \\
& = \frac{F}{t_0} + t_0 \varphi \hoep ( -\Delta_1 f_\ep - 2 \dot f_\ep  ) - \frac{F}{\varphi} \Delta_{g_\ep} \varphi - 2 t_0\Re\innpro{ \nabla \varphi, \bar \nabla\xk{ \frac{F}{t_0\varphi}}  } \\
&\le \frac{F}{t_0} + t_0\varphi \bk{ - ( -\Delta_1 f_\ep  )^2 - 2 \frac{F}{t_0\varphi^2} \Re \innpro{\nabla f_\ep, \bar\nabla \varphi}  } + C\frac{F}{R^2 \varphi}( \varphi' + \varphi '' )\\
& \quad + 2 \frac{F}{ R^2\varphi^2 } (\varphi')^2
}
where we use the Laplacian comparison and the fact that $\nabla F = 0$ at $(p_0,t_0)$. The second term on the RHS satisfies (we denote $\tilde F: = - \Delta_1 f_\ep - 2 \dot f_\ep$ for notation convenience)
{\small
\begin{align*}
 t_0\varphi \bk{ - ( -\Delta_1 f_\ep  )^2 - 2 \frac{F}{t_0\varphi^2} \Re \innpro{\nabla f_\ep, \bar\nabla \varphi}  }
\le  &~ t_0 \varphi \bk{ -\tilde F^2- 4 \tilde F \dot f_\ep - 4 (\dot f_\ep)^2 +     \frac{2 \tilde F}{  \varphi}  \frac{|\nabla f_\ep| |\varphi'|}{R}  }\\
\le  &~ t_0 \varphi \bk{ -\tilde F^2- 4 \tilde F \dot f_\ep  + 2 \tilde F \abs{\nabla f_\ep} + \frac{\tilde F |\varphi'|^2}{2 R^2 \varphi^2}}\\ 
\text{(by Lemma \ref{lemma LY})~}\le & ~ t_0 \varphi \bk{ -\tilde F^2 + \frac{\tilde F |\varphi'|^2}{2 R^2 \varphi^2}  +  C \frac{\tilde F}{t_0} + C \frac{\tilde F}{ R^2}   }\\
= & -\frac{F^2}{t_0 \varphi} + C \frac{F }{2 R^2 \varphi} + C \frac{F}{t_0} + C \frac{F}{ R^2}, 
\end{align*}
}
inserting this to \eqref{eqn:para new 1}, we get for some constant $C=C(n)>0$ at $(p_0,t_0)$
$$ - F^2 + C \varphi F + \frac{t_0\varphi F}{R^2} + C t_0 \frac{F}{R^2}\ge 0,   $$
from which we obtain $F(p_0,t_0)\le \frac{Ct_0}{R^2} + C$, and by the choice of $(p_0,t_0)$, we can see that
\begin{equation*}
\sup_{B_{g_\ep}(0,R/2)} ( - \Delta_1 f_\ep - 2 \dot f_\ep   )\le C\xk{ \frac{1}{R^2} + \frac{1}{t}  },\quad \forall~ t\in (0,R^2),
\end{equation*}
which implies that on $B_{g_\ep}(0,3R/5) \times (0,R^2)$ 
\begin{equation}\label{eqn:para new 2}
-\Delta_1 u_\ep \le \dot u_\ep + C\xk{\frac{1}{t} + \frac{1}{R^2}  } u_\ep.
\end{equation}
Applying \eqref{eqn:para new 2} to the function $\sup u_\ep - u_\ep$, we obtain that on $B_{g_\ep}(0,3R/5)\times (0,R^2)$
$$  |\Delta_1 u_\ep| \le |\dot u_\ep| + C \xk{ \frac{1}{t} + \frac{1}{R^2}  }\osc_R u_\ep \le C \xk{ \frac{1}{t} + \frac{1}{R^2}  }\osc_R u_\ep $$
by equation \eqref{eqn:para grad 2}. Thus we finish the proof of the lemma.
\end{proof}

\begin{lemma}\label{lemma 4.3}
There exists a constant $C=C(n)>0$ such that 
$$\sup_{i\neq j }\sup_{B_{g_\ep}(0, R/2)}(  | \nabla_i\nabla_j u_\ep   | +  |\nabla_i \bar \nabla_j u_\ep|  ) \le C \xk{ \frac{1}{t} + \frac{1}{R^2}  } \osc_R u_\epp, $$
for all $t\in (0, R^2)$. Recall here $|\nabla_i \nabla_j u_\ep|^2 = \nabla_i \nabla_j u_\ep \nabla_{\bar i}\nabla_{\bar j} u_\ep g_\ep^{i\bar i} g^{j\bar j}_\ep $ (no summation over $i, j$ is taken).
\end{lemma}
\begin{proof}
We will only prove the estimate for $|\nabla_1 \nabla_2 u_\ep|$. The others are similar, so we omit the proof. 
 
By similar calculations as in deriving \eqref{eqn:first step 6}, we have
\begin{equation}\label{eqn:step 61}
\hoep |\nabla_1 \nabla_2 f_\ep| \le 2 \Re\innpro{ \nabla f_\ep, \bar \nabla | \nabla_1 \nabla_2 f_\ep  |    } + \sum_k \bk{ | \na_1 \na_{k} f_\ep |  | \nabla_2 \nabla_{\bar k} f_\ep  | + |\nabla_2 \nabla_k  f_\ep||\na_1 \na_{\bar k} f_\ep|   } ,
\end{equation}
and similar to \eqref{eqn:useful 2}
\eqsp{\label{eqn:useful 21}
 \hoep( -\Delta_1 f_\ep & - \Delta_2 f_\ep  )\le  ~2\Re\innpro{ \nabla f_\ep, \bar \na( -\Delta_1 f_\ep - \Delta_2 f_\ep  )  } \\
 &~  -\sum_k \xk{ \abs{ \na_1 \na_k f_\ep  } + \abs{\nabla_1 \na_{\bar k} f_\ep} + \abs{\na_2 \na_k f_\ep} + \abs{\na_2 \na_{\bar k} f_\ep}   }
.}
Combining \eqref{eqn:useful 21}, \eqref{eqn:step 61} and Cauchy-Schwarz inequality, we get 
{\small
\begin{align*}
&\hoep\big( |\nabla_1 \na_2 f_\ep| +  2( -\Delta_1 f_\ep - \Delta_2 f_\ep  )   \big)\le 2\Re\innpro{ \nabla f_\ep, \bar \na \xk{|\nabla_1 \na_2 f_\ep| +  2( -\Delta_1 f_\ep - \Delta_2 f_\ep  )}  } \\
 &~~   -\sum_k \xk{ \abs{ \na_1 \na_k f_\ep  } + \abs{\nabla_1 \na_{\bar k} f_\ep} + \abs{\na_2 \na_k f_\ep} + \abs{\na_2 \na_{\bar k} f_\ep}} \\
 \le &~~    2\Re\innpro{ \nabla f_\ep, \bar \na \xk{|\nabla_1 \na_2 f_\ep| +  2( -\Delta_1 f_\ep - \Delta_2 f_\ep  )}  }  - \frac{1}{10} \bk{  |\nabla_1 \na_2 f_\ep| +  2( -\Delta_1 f_\ep - \Delta_2 f_\ep  )  }^2.
\end{align*}
}
We define a similar cut-off function $\eta$ as $\varphi$ in the proof of Lemma \ref{lemma 4.2}, such that $\eta = 1$ on $B_{g_\ep}( 0, R/2  )$ and vanishes outside $B_{g_\ep}(0,3R/5)$. We denote
$$G = t \eta \xk{|\nabla_1 \na_2 f_\ep| +  2( -\Delta_1 f_\ep - \Delta_2 f_\ep  ) - 2 \dot f_\ep   }.$$
We can argue similarly as the $F$ in the proof of Lemma \ref{lemma 4.2} that at the maximum point $(p_0,t_0)$ of $G$, for which we assume $G(p_0,t_0)>0$
\begin{align*}
0\le & ~~\hoep G\\
\le &~~ \frac{G}{t_0} - \frac{G^2}{t_0 \eta} + C \frac{G}{R^2 \eta} + C \frac{G}{t_0} + C \frac{G}{R^2} + C \frac{G}{R^2} \frac{\eta'+\eta''}{\eta} + \frac{2 G}{R^2 \eta^2}(\eta')^2\\
\le & ~~ \frac{1}{t_0\eta}\bk{- G^2 + C \eta G + \frac{t_0\eta G}{R^2} + C t_0 \frac{G}{R^2}},
\end{align*}
so it follows that $G(p_0,t_0)\le C\xk{ 1+ \frac{t_0}{R^2}   }$. Therefore by definition of $G$, it holds that on $B_{g_\ep}(0,R/2)\times (0,R^2)$
$$|\nabla_1 \na_2 f_\ep| + 2( -\Delta_1 f_\ep - \Delta_2 f_\ep  ) -2 \dot f_\ep \le C \xk{ \frac{1}{R^2} + \frac{1}{t}  },$$
thus by Lemmas \ref{lemma LY} and \ref{lemma 4.2}, we conclude that on $B_{g_\ep}(0,R/2)\times (0,R^2)$
\begin{align*}
| \na_1 \na_2 u_\ep  |\le &  ~\dot u_\ep + 2 | \Delta_1 u_\ep  | + 2 |\Delta_2 u_\ep| + \frac{\abs{\na u_\ep}}{ u_\ep  } + C u_\ep\xk{ \frac{1}{R^2} + \frac{1}{t} }\\
\le &~ C  \xk{ \frac{1}{t} + \frac{1}{R^2}   }\osc_R \,u_\ep,
\end{align*}
as desired.
\end{proof}

\subsubsection{Existence of $u$ to \eqref{eqn:para Diri}}We will show the limit function of $u_\ep$ as $\ep\to 0$ solves \eqref{eqn:para Diri}.
\begin{prop}\label{prop 4.1}
Given any $R\in (0,1)$ and  any $\varphi\in \C^0(\partial_{\pp} \qq_\bb(0,R))$, there exists a unique function $u\in \cC^{2,1}(\qq_\bb(0,R)^\#)\cap C^0(\overline{\qq_\bb(0,R)})$ solving the equation \eqref{eqn:para Diri}. Moreover, there exists a constant $C=C(n,\bb)>0$ such that for any $t\in (0,R^2)$ (we denote $B_\bb(r)^\#:= B_\bb(0,r)\backslash \sS$)
\begin{equation}\label{eqn:para new grad 1}
\sup_{B_\bb(R/2)^\# }\bk{ \sum_{j=1}^p |z_j|^{2-2\beta_j} \ba{ \frac{\partial u}{\partial z_j}  }^2 + \ba{D' u}^2      } \le C \xk{ \frac{1}{t} + \frac{1}{R^2}  } ( \osc_R u  )^2,
\end{equation}
\begin{equation}\label{eqn:para lap 1}
\sup_{B_\bb(R/2)^\#} \bk{ \sum_{i\neq j} ( | \nabla_i \nabla_j u  |_{g_\bb}+ |\nabla_i \bar \na_j u|_{g_\bb}) + \big| \frac{\partial u}{\partial t}  \big|     } \le C \xk{ \frac{1}{t} + \frac{1}{R^2}  } \osc_R u,
\end{equation}
and 
\begin{equation}\label{eqn:para new grad 2}
\sup_{B_\bb(R/2)^\#} \bk{ \sum_{j=1}^p |\nabla_{g_\bb} \Delta_j u| + |\nabla_{g_\bb} (D')^2 u   |  + | \nabla_{g_\bb} \frac{\partial u}{\partial t}  |   }  \le C \xk{ \frac{1}{t} + \frac{1}{R^2}  }^{3/2} \osc_R u,  
\end{equation}
where by abusing notation we denote $\osc_R u : = \osc_{B_\bb(0,R)\times (0,R^2)} u$. 
\end{prop}
\begin{proof}Let $u_\epsilon$ be the $\cC^{2,1}$-solution to the equation \eqref{eqn:para eps}.   The $\C^0$-norm of $u_\epsilon$ follows from maximum principle \eqref{eqn:para mp}. 

To prove the higher order estimates, for any fixed compact subset $K\Subset B_\bb(0,R)$ and $\delta>0$,  standard parabolic Schauder theory yields uniform $\C^{4+\alpha,\frac{4+\alpha}{2}}$-estimates of $u_\ep$ on $(K\backslash T_\delta\sS)\times ( \delta, R^2]$, for any $\alpha\in(0,1)$. As $\epsilon\to 0$, $u_\epsilon$ converges in $\C^{4+\alpha,\frac{4+\alpha}{2}}( K\backslash T_\delta \sS \cap (\delta, R^2] )$ to some function $u$ which is also $\C^{4+\alpha,\frac{4+\alpha}{2}}$ in  $(K\backslash T_\delta\sS)\times ( \delta, R^2]$. Let $\delta\to 0$, $K\to B_\bb(0,R)$ and use a diagonal argument, then we can assume that
$$u_\epsilon\xrightarrow{\C^{4+\alpha,\frac{4+\alpha}{2}}_{loc} ( B_\bb(0,R)^\#\times (0, R^2]    )      } u,\quad \text{as ~}\ep\to 0.$$
Letting $\ep\to 0$, the estimate \eqref{eqn:para new grad 1} follows from \eqref{eqn:para grad 1};  \eqref{eqn:para lap 1} is a consequence of  Lemma \ref{lemma 4.3}; and  \eqref{eqn:para new grad 2} follows by applying the gradient estimate \eqref{eqn:para grad 1} to the $\Delta_{g_\ep}$-harmonic functions $\Delta_j u_\ep$ and $(D')^2 u_\ep$, and then letting $\ep\to 0$.

The gradient estimate \eqref{eqn:para new grad 1} implies that for any compact $K\Subset B_\bb(0,R)$
\begin{equation*}
\sup_{K\backslash \sS_j }\ba{\frac{\partial u}{\partial z_j}}\le \frac{C(n, K,\bb) (\osc_R u)^2}{t} |z_j|^{\beta_j - 1},\quad \forall~ t\in (0, R^2).  
\end{equation*}
From this we see that for any $t\in (0,R^2)$,  $u(\cdot, t)$ can be continuously extended to $\sS$ and thus $u\in C^0( B_\bb(0,R)\times (0, R^2)  )$.

\medskip

It only remains to show $u = \varphi$ on $\partial_\pp \qq_\bb(0,R)$. Fix an arbitrary point $(q_0,t_0)\in \partial_\pp\xk{ \qq_\bb(0,R)}$. 

\medskip

\noindent{\bf Case 1.} $t_0 = 0$ and $q_0\in \overline{B_\bb(0, R)}$.  We define a barrier function $\phi_{1}(z, t) = e^{-d_{\mathbb C^n}(z, q_0)^2 - \lambda t} - 1$, where $\lambda >0$ is to be determined. We calculate 
\begin{align*}
\hoep \phi_1 & = -\lambda e^{ -d_{\mathbb C^n}(z, q_0)^2 - \lambda t } - (- \Delta_{g_\ep} d_{\mathbb C^n}^2 + \abs{ \na d^2_{\mathbb C^n}  }_{g_\ep}   ) e^{ -d_{\mathbb C^n}(z, q_0)^2 - \lambda t  }\\
&\le \big( -\lambda + \sum_{j=1}^p (|z_j|^2 +\ep   )^{1-\beta_j} + (n-p)  \big)  e^{ -d_{\mathbb C^n}(z, q_0)^2 - \lambda t  }<0,
\end{align*}
if $\lambda \ge 4n$. On the other hand, $\phi_1(q_0,t_0) = 0$ and $\phi_1(z, t )< 0 $ for any $(z, t)\neq (q_0,t_0)$. For any $\varepsilon>0$, we can find a small neighborhood $V\cap \partial_\pp \xk{\qq_\bb(0,R)}$ of $(q_0,t_0)$ such that on $V$, $ \varphi(q_0,t_0) + \varepsilon  >\varphi(z, t) > \varphi(q_0, t_0) - \varepsilon$ since $\varphi$ is continuous. On $\partial_\pp\xk{\qq_\bb(0,R)}\backslash V$, the function $\phi_1$ is bounded above by a negative constant. Therefore the function $\phi_1^-: = \varphi(q_0, t_0) - \varepsilon + A \phi_1(z,t) \le \varphi(z,t)  $ for any $(z,t)\in\partial_\pp\xk{ \qq_\bb(0,R)}$ if $A>>1$. Therefore by maximum principle $\phi_1^-(z,t)\le u_\ep(z,t)$ for any $(z,t)\in \qq_\bb(0,R)$. Letting $\epsilon\to 0$, $\phi_1^-(z,t)\le u(z,t)$. So letting $(z,t)\to (q_0,t_0)$, $\varphi(q_0,t_0) - \varepsilon\le \liminf_{(z,t)\to (q_0,t_0)}  u(z,t)$. Setting $\varepsilon\to 0$ we conclude that 
$ \varphi(q_0,t_0)\le \liminf_{(z,t)\to (q_0,t_0)}  u(z,t).$ By considering $\phi_1^+(z,t) = \varphi(q_0,t_0) + \varepsilon - A \phi_1(z,t)$ and similar argument as above we can get $\varphi(q_0,t_0) \ge \limsup_{(z,t)\to (q_0,t_0)} u(z,t)$. Thus $u$ coincides with $\varphi$ at $(q_0,t_0)$.

\medskip

\noindent{\bf Case 2.} $t_0>0$ and $q_0\in \partial B_\bb(0,R)\cap(\sS_1\cap \sS_2)$. In this case $z_1(q_0) = z_2(q_0) = 0$. Denote $q_0' = -q_0\in \partial B_\bb(0,R)$ to be the (Euclidean) opposite point of $q_0$. Define for some small $\delta>0$ $$\phi_2(z, t) = d_{\mathbb C^n}(z, q_0')^2 - 4 R^2 - \delta (t-t_0)^2.$$  
$\phi_2(q_0,z_0) = 0$ and $\phi(z,t)<0$ for any $(z,t)\neq (q_0,t_0)$. We calculate that $\partial_t \phi_2 - \Delta_{g_\ep} \phi_2\le 0$. Then by similar argument as in {\bf Case 1} replacing $\phi_1$ by $\phi_2$, we get $\lim_{(z,t)\to (q_0,t_0)} u(z,t) = \varphi(q_0,t_0)$.

\medskip

\noindent{\bf Case 3.} $t_0>0$ and $q_0\in\partial B_\bb(0,R)\backslash  ( \sS_1\cap \sS_2   )$.  As the {\bf Case 2} in the proof of Proposition \ref{prop:existence}, we define a similar function $G$ as there. Define $\phi_3(z, t) = A ( d_\bb(z,0)^2 - R^2  ) + G(z) - \delta (t -t_0)^2$ for $A>>1$ and small $\delta>0$. Then we can calculate that $\partial_t \phi_3 \le \Delta_{g_\ep} \phi_3$ and $\phi_3(q_0,t_0) = 0$ and $\phi_3(z, t) <0$ for any other $(z,t)\neq (q_0,t_0)$. Similar argument as in {\bf Case 1} proves that
$$\lim_{(z,t)\to (q_0,t_0)} u(z, t) = \varphi(q_0,t_0).$$

Combining all the three cases above, we obtain that $u$ coincides with $\varphi$ on $\partial_\pp \qq_\bb(0,R)$. Thus the Dirichlet problem \eqref{eqn:para Diri} admits a unique solution $u\in \C^0(\overline{\qq_\bb(0,R)  })\cap \C^{2,1}( \qq_\bb(0,R)^\#  )$.

\end{proof}

\begin{corr}\label{cor 4.1 new}
Given any $f\in \C^{\alpha, \frac\alpha 2}_\bb( \overline{\qq_\bb }  )$ and $\varphi\in \C^0(\partial_\pp \qq_\bb)$, there exists a unique solution $v\in \C^{2,1}(\qq_\bb^\#)\cap \C^0(\overline{\qq_\bb})$ to the Dirichlet problem
\begin{equation}\label{eqn:general para new}
\frac{\partial v}{\partial t} = \Delta_{g_\bb} v + f,\text{ in }\qq_\bb, \text{ and } v = \varphi, \text{ on }\partial_\pp \qq_\bb.
\end{equation} 

\end{corr}
\begin{proof}
Let $v_\ep\in \C^{2+\alpha, \frac{2+\alpha}{2}}(\qq_\bb)\cap \C^0(\qq_\bb) $ be the unique solution to the equations
$$\frac{\partial v_\ep}{\partial t} = \Delta_{g_\ep} v_\ep + f,\text{ in }\qq_\bb,\text{ and } v_\ep = \varphi, \text{ on }\partial_\pp \qq_\bb.$$
For any compact subset $K\Subset B_\bb(0,1)$ and $\delta\in (0,1)$, the standard Schauder estimates for parabolic equations provide uniform $\C^{2+\alpha,\frac{2+\alpha}{2}}$-estimates for $v_\ep$ on $K\backslash T_\delta \sS\times (\delta^2, 1)$. Then $v_\ep \to v$ for some $v\in \C^{2+\alpha,\frac{2+\alpha}{2}}(K\backslash T_\delta\sS \times (\delta^2, 1))$. Taking $\delta\to 0$ and $K\to B_\bb(0,1)$ and by a diagonal argument, we can take $v_\ep$ converges in $\C^{2+\alpha,\frac{2+\alpha}{2}}_{loc}( B_\bb\backslash \sS \times (0,1)  )$ to $v$, and $v$ satisfies the equation $\frac{\partial v}{\partial t} = \Delta_{g_\bb} v + f$ on $B_\bb\backslash\sS \times (0,1)$. 

It only remains to show $v\in \C^0(\qq_\bb)$ and $v = \varphi$ on $\partial_\pp\qq_\bb$. The same proof as in {\bf Cases 1, 2, 3} in Proposition \ref{prop 4.1} yields that $v$ must coincide with $\varphi$ on $\partial_\pp\qq_\bb$, since we can always choose $A>1$ large enough such that (for example in {\bf Case 1}) $\frac{\partial \phi_1^-}{\partial t} - \Delta_{g_\ep} \phi_1^-\le \inf_{\qq_\bb} f \le \frac{\partial v_\ep}{\partial t} - \Delta_{g_\ep} v_\ep$. To see the continuity of $v$ in $\qq_\bb$, because of the Sobolev inequality \eqref{eqn:Sobolev} for metric spaces $(B_\bb, g_\ep)$, by the proof of the standard De Giorgi-Nash-Moser theory for parabolic equations, we conclude that for any $p\in \sS$, $t_0\in (0,1)$, there exists a small number $R_0 = R_0(p,t_0)$ such that on the cylinder $\tilde \qq_{R_0}:=B_\bb(p, R_0)\times (t_0 - R_0^2, t_0  )$, $\osc_{\tilde \qq_{r}} v_\ep \le C r^{\alpha'}$ for any $r\in (0,R_0)$ and some $\alpha'\in (0,1)$. Therefore $\osc_{\tilde \qq_{r}}v\le C r^{\alpha'}$ and $v$ is continuous at $(p, t_0)$, as desired.

The uniqueness of the solution to \eqref{eqn:general para new} follows from maximum principle. 
\end{proof}
\begin{remark}\label{remark 4.1}
Corollary \ref{cor 4.1 new} is not needed in the proof of Theorem \ref{thm:main 2}. So by Theorem \ref{thm:main 2}, the solution $u$ to \eqref{eqn:general para new} is in $\C^{2+\alpha, \frac{2+\alpha}{2}}_\bb( \qq_\bb )\cap \C^0(\overline{\qq_\bb})$.

\end{remark}

\subsection{Sketched proof of Theorem \ref{thm:main 2}}
With Proposition \ref{prop 4.1}, we can prove the Schauder estimates for the solution $u\in \C^0(\overline{\qq_\bb})\cap \cC^{2,1}(\qq_\bb^\#)$ to the equation \eqref{eqn:para main} for a Dini-continuous function $f$, by making use of almost the same arguments as in the proof of Theorem \ref{thm:main 1}. So we will not provide the full details, and only point out the main differences. For any given points $Q_p = (p,t_p)$, $Q_q = (q,t_q)\in (B_\bb(0,1/2)\backslash \sS)\times (\hat t,1)$. To define the approximating functions $u_k$ as in \eqref{eqn:uk 1}, we define $u_k$ in this case as the solution to the heat equation
$$ \frac{\partial u_k}{\partial t} = \Delta_{g_\bb} u_k + f(Q_p), \quad\text{in ~} \hat B_k(p)\times (t_p - \hat t\cdot \tau^{2k}, t_p  ]   $$
and $u_k = u$ on $\partial_\pp\xk{  \hat B_k(p)\times (t_p -\hat t\cdot \tau^{2k}, t_p  ]    }$, where $\hat B_k(p)$ is defined in \eqref{eqn:hat Bk}. We can now apply the estimates in Proposition \ref{prop 4.1} to the functions $u_k$ or $u_k - u_{k-1}$, instead of the ones in Lemmas \ref{lemma 2.1} and \ref{lemma 2.2} as we did in Sections \ref{section 3.2}, \ref{section 3.3} and \ref{section 3.4} to prove the Schauder estimates for $u$. Thus we finish the proof of Theorem \ref{thm:main 2}. \qed

\subsection{Interior Schauder estimate for non-flat conical K\"ahler metrics}
Let $g=\sqrt{-1} g_{j\bar k}(z,t) dz_j\wedge dz_{\bar k}$ be a $\cC^{\alpha,\frac{\alpha}{2}}_\bb$ conical K\"ahler metric on $\qq_\bb$ with conical singularity along $\sS$, that is,  for any $t\in [0,1]$, $g(\cdot, t)$ is a $C^{0,\alpha}_\bb$ conical K\"ahler metric as in Section \ref{section 3.5} and 
the coefficients of $g$ in the basis $\{\epsilon_j\wedge \bar \epsilon_k$, $\cdots\}$, are $\frac{\alpha}{2}$-H\"older continuous in $t\in [0,1]$. Suppose $u\in \C^{2+\alpha, \frac{2+\alpha}{2}}_\bb( \qq_\bb  )$ satisfies the equation
\begin{equation}\label{eqn:para general}
\frac{\partial u}{\partial t} = \Delta_{g} u + f,\quad \text{in ~ }\qq_\bb,
\end{equation} 
for some $f\in \C^{\alpha, \frac{\alpha}{2}}(\overline{\qq_\bb})$.
\begin{prop}\label{prop 4.2 new}
There exists a constant $C=C(n,\bb, \alpha, g)$  such that 
\begin{equation*}
\| u\|_{\C^{2+\alpha,\frac{2+\alpha}{2}}_\bb(\qq_\bb)    }^* \le C\xk{ \| u\|_{\C^0( \qq_\bb  )} + \| f\|_{ \C^{\alpha, \frac{\alpha}{2}}_\bb(\qq_\bb)  }^{(2)}    }.
\end{equation*}

\end{prop}

\begin{proof}
The proof is parallel to that of Proposition \ref{prop:invariant}. Given any two points $P_{x} = (x,t_x)$, $P_y = (y, t_x)\in \qq_\bb$, we may assume $d_{P_x} = \min\{d_{P_x}, d_{P_y}\}>0$ where $d_{P_x} := d_{\pp, \bb}(P_x,\partial_\pp \qq_\bb)$ is the parabolic distance of  
$P_x$ to the parabolic boundary $\partial_\pp \qq_\bb$. Let $\mu\in (0,1/4)$ be a positive number to be determined later. Denote $d := \mu d_{P_x}$, $\qq:= B_{\bb}(x, d)\times ( t_x - d^2, t_x  ]$ the ``parabolic ball'' centered at $P_x$, and $\frac 12 \qq := B_\bb(x, d/2)\times (t_x - d^2/4, t_x]$. 

\medskip

\noindent $\bullet$ {\bf Case 1.} $d_{\pp,\bb} ( P_x, P_y  )< d/2$. In this case we always have $P_y\in\frac 12 \qq$.

\smallskip

{\em Case 1.1}. $B_\bb(x, d)\cap \sS = \emptyset$. As in the proof of Proposition \ref{prop:invariant}, we can introduce smooth complex coordinates $\{ w_1,w_2,z_3,\ldots, z_n  \}$ on $B_\bb(x, d)$, under which $g_\bb$ becomes the standard Euclidean metric and the components of $g$ are $\C^{\alpha,\frac \alpha 2}$ in the usual sense on $\qq$. The leading coefficients and constant term $f$ in \eqref{eqn:para general} are both $\C^{\alpha,\frac \alpha 2}$ in the usual sense, so we can apply the standard parabolic Schauder estimates (see Theorem 4.9 in \cite{Lie}) to get that there exists some constant $C = C(n,\bb, \alpha, g)$ which is independent of $\qq$
\begin{equation}\label{eqn:lemma Lie}
[  u ]^*_{\C^{2+\alpha, \frac{2+\alpha}{2}} (\qq)   } \le C\xk{ \| u\|_{\C^0(\qq)  } + \| f\|^{(2)} _{ \C^{\alpha, \frac\alpha 2} (\qq)  }   }.
\end{equation}
Let $D$ denote the ordinary first order operators in the coordinates $\{w_1,w_2,z_3,\ldots, z_n\}$. We calculate
\begin{align*}
| Tu(P_x)  - Tu(P_y)  | & \le ~ | D^2 u(P_x) - D^2 u(P_y)   | + \frac{ d_{\pp, \bb}(P_x, P_y)  }{d} \xk{ |D^2 u(P_x)| + |D^2 u(P_y)|   }\\
& \le ~ \frac{4 d_{\pp,\bb}(P_x, P_y)^\alpha}{d^{2+\alpha}} [u  ]^*_{ \C^{2+\alpha, \frac{2+\alpha}{2}}(\qq)  } + \frac{4 d_{\pp,\bb} (P_x, P_y)   }{d^3} [u]^*_{\C^{2,1}(\qq)}\\
& \le ~ \frac{8 d_{\pp,\bb}  (P_x, P_y)^\alpha }{d^{2+\alpha}} [u]^*_{\C^{2+\alpha, \frac{2+\alpha}{2}} (\qq)  } + C \frac{d_{\pp,\bb}(P_x,P_y)^\alpha}{d^{2+\alpha}}\| u\|_{\C^0(\qq)}.
\end{align*}
\begin{align*}
\big|\frac{\partial u}{\partial t}(P_x) - \frac{\partial u}{\partial t}(P_y)\big   |\le  \frac{ 4 d_{\pp,\bb}( P_x,P_y  )^\alpha}{ d^{2+\alpha} } [u]^*_{\C^{ 2+\alpha, \frac{2+\alpha}{2}  } (\qq) }. 
\end{align*}
Recall $\newT $ denotes the operators in $T$ and $\frac{\partial}{\partial t}$, then  by \eqref{eqn:lemma Lie} it follows that
\begin{equation}\label{eqn:para case 1.1}
d_{P_x}^{2+\alpha} \frac{|\newT u(P_x) - \newT u(P_y)|}{d_{\pp,\bb}(P_x, P_y)^\alpha} \le \frac{C}{\mu^{2+\alpha}} \|  f\|^{(2)}_{ \C^{\alpha, \frac{\alpha}{2}}_\bb(\qq)  } + \frac{C}{\mu^{2+\alpha}} \| u\|_{\C^0(\qq)}.
\end{equation}

\smallskip

{\em Case 1.2}. $B_\bb(x, d)\cap \sS\neq \emptyset$. Let $\hat x\in \sS$ be the projection of $x$ to $\sS$ and $\hat P_x = ( \hat x, t_x )$ be the corresponding space-time point. Denote $\hat \qq := B_\bb(\hat x, 2d) \times ( t_x - 4d^2, t_x]  $.  As the {\em Case 1.2} in the proof of Proposition \ref{prop:invariant}, we may choose suitable complex coordinates so that $g_{\ep_j \bar \ep_k}(\hat P_x  ) = \delta_{jk}  $ and for $j,k\ge p+1$ $g_{j\bar k}(\hat P_x) = \delta_{jk}$, and the cross terms in the expansion of $g$ in \eqref{eqn:expand cone} vanish at $\hat P_x$. Thus the equation \eqref{eqn:para general} can be re-written as
\begin{equation*}
\frac{\partial u}{\partial t} = \Delta_{g_\bb} u + \eta.\ddb u+ f =: \Delta_{g_\bb} u + \tilde f,\quad u\in \C^0(\hat \qq)\cap \C^{2,1}(\hat \qq^\#),
\end{equation*}
for some $(1,1)$-form $\eta$ as in the proof of Proposition \ref{prop:invariant}. From the rescaled version of Theorem \ref{thm:main 2} we conclude that 
\begin{equation*}
d^{2+\alpha}\frac{ |\newT u(P_x)  - \newT u(P_y)  |   } { d_{\pp, \bb}(P_x, P_y)^\alpha    } \le C \xk{ \| u\|_{\C^0(\hat \qq)} + \|\tilde f\|^{(2)}_{ \C^{\alpha, \frac\alpha 2}_\bb( \hat \qq  )  }     },
\end{equation*}
hence 
\begin{equation}\label{eqn:para case 1.2}
d_{P_x}^{2+\alpha}\frac{ |\newT u(P_x)  - \newT u(P_y)  |   } { d_{\pp, \bb}(P_x, P_y)^\alpha    } \le \frac{C}{\mu^{2+\alpha}} \xk{ \| u\|_{\C^0(\hat \qq)} + \|\tilde f\|^{(2)}_{ \C^{\alpha, \frac\alpha 2}_\bb( \hat \qq  )  }     }.
\end{equation}

\smallskip

\noindent $\bullet$ {\bf Case 2.} $d_{\pp,\bb}(P_x,P_y) \ge d/2$. Then we calculate (recall $\qq_\bb := B_\bb(0,1)\times (0,1]$)
\begin{equation}\label{eqn:para case 2}
d_{P_x}^{2+\alpha}\frac{ |\newT u(P_x)  - \newT u(P_y)  |   } { d_{\pp, \bb}(P_x, P_y)^\alpha    } \le 4 d_{P_x} ^{2+\alpha} \frac{|\newT  u|(P_x) + |\newT  u|(P_y)  }{ d^\alpha  }\le \frac{8}{\mu^\alpha} [ u  ]^*_{ \C^{2,1}_\bb (\qq_\bb)   }.  
\end{equation}

Combining \eqref{eqn:para case 1.1}, \eqref{eqn:para case 1.2} and \eqref{eqn:para case 2}, we obtain
\begin{align*}
d_{P_x}^{2+\alpha}\frac{ |\newT u(P_x)  - \newT u(P_y)  |   } { d_{\pp, \bb}(P_x, P_y)^\alpha    } \le &  \frac{8}{\mu^\alpha} [ u  ]^*_{ \C^{2,1}_\bb (\qq_\bb)   } + \frac{C}{\mu^{2+\alpha}} \xk{ \| u\|_{\C^0(\hat \qq)} + \|\tilde f\|^{(2)}_{ \C^{\alpha, \frac\alpha 2}_\bb( \hat \qq  )  }     }\\
&+  \frac{C}{\mu^{2+\alpha}} \|  f\|^{(2)}_{ \C^{\alpha, \frac{\alpha}{2}}_\bb(\qq)  } + \frac{C}{\mu^{2+\alpha}} \| u\|_{\C^0(\qq)}.
\end{align*}

Observe that for any $P\in \qq$ or $\in \hat \qq $, $d_{\pp,\bb}(P,\partial_\pp\qq_\bb)\ge (1-2\mu)d_{P_x}$. Then it follows from definition that 
\begin{equation*}
\| f\|_{\C^{\alpha, \frac\alpha 2}_\bb( \qq)    }^{(2)}\le C \mu^2 \| f \|^{(2)}_{\C^0( \qq_\bb  )} + C \mu^{2+\alpha} [ f  ]^{(2)}_{ \C^{\alpha,\frac\alpha 2}_\bb(\qq_\bb)   } \le C \mu^2 \| f\|^{(2)}_{\C^{\alpha,\frac\alpha 2}_\bb(\qq_\bb)}.
\end{equation*}
We calculate
\begin{align*}
\| \tilde f\|^{(2)}_{ \C^{\alpha,\frac\alpha 2}_\bb(\hat \qq  )  } & \le \| \eta\|_{ \C^{\alpha, \frac\alpha 2}_\bb(\hat\qq)  }^{(0)} \| Tu\|^{(2)}_{ \C^{\alpha, \frac\alpha 2}_\bb(\hat \qq)   } + \| f\|^{(2)}_{\C^{\alpha,\frac\alpha 2}_\bb( \hat \qq  )  }\\
& \le~ C_1 [g]^*_{\C^{\alpha, \frac\alpha 2}_\bb( \qq_\bb  )  }  \mu^\alpha \xk{ \mu^2 [u]^*_{\C^{2,1}_\bb( \qq_\bb  )} + \mu^{2+\alpha} [u]^*_{\C^{2+\alpha,\frac{2+\alpha}{2}}_\bb(\qq_\bb)   }      }  + \mu^2 \| f\|^{(2)}_{ \C^{\alpha,\frac\alpha 2}_\bb(\qq_\bb)  }  \\
&  \le~ C_1 [g]^*_{\C^{\alpha, \frac\alpha 2}_\bb( \qq_\bb  )  }  \mu^\alpha \xk{C(\mu) [u]^*_{\C^{2,1}_\bb( \qq_\bb  )} +  2 \mu^{2+\alpha} [u]^*_{\C^{2+\alpha,\frac{2+\alpha}{2}}_\bb(\qq_\bb)   }      }  + \mu^2 \| f\|^{(2)}_{ \C^{\alpha,\frac\alpha 2}_\bb(\qq_\bb)  },  
\end{align*}
where in the last inequality we use the interpolation inequality, by which we also have
\begin{equation*}
\frac{8}{\mu^\alpha} [ u  ]^*_{ \C^{2,1}_\bb(\qq_\bb)  } \le \mu^\alpha [u]^*_{\C^{2+\alpha,\frac{2+\alpha}{2}}_\bb( \qq_\bb )   } + C(\mu)\| u\|_{\C^0(\qq_\bb)}.
\end{equation*}
If $\mu$ is chosen small such that $\mu^\alpha( 2 C_1 [g]^*_{\C^{\alpha,\frac\alpha 2}_\bb  (\qq_\bb)  } + 1   )< 1/2$, combining the above inequalities we get
\begin{equation*}
d_{P_x}^{2+\alpha}\frac{ |\newT u(P_x)  - \newT u(P_y)  |   } { d_{\pp, \bb}(P_x, P_y)^\alpha    } \le \frac1 2[ u  ]^*_{ \C^{2+\alpha,\frac{2+\alpha}{2}}_\bb(\qq_\bb)  } + C(\mu)\xk{ \| u\|_{\C^0(\qq_\bb)}  + \| f\|^{(2)}_{ \C^{\alpha, \frac\alpha 2}_\bb(\qq_\bb)  }   }.
\end{equation*}
Taking supremum over all $P_x\neq P_y\in \qq_\bb$, we obtain that 
$$[u]^*_{\C^{2+\alpha, \frac{2+\alpha}{2}  } _\bb(\qq_\bb)   }\le C \xk{ \| u\|_{\C^0(\qq_\bb)} + \| f\|^{(2)}_{\C^{\alpha,\frac\alpha 2}_\bb(\qq_\bb)}   }.$$
The proposition is proved by invoking the interpolation inequalities.

\end{proof}
\begin{remark}
It follows from the proof that the estimates in Proposition \ref{prop 4.2 new} also hold on $\qq_\bb(p,R):= B_\bb(p, R)\times (0,R^2)\subset \qq_\bb$, i.e. the cylinder whose spatial center $p$ may not lie in $\sS$. 

\end{remark}

It is easy to derive the following local Schauder estimate for $\C^{2+\alpha,\frac{2+\alpha}{2}}_\bb$-solutions to \eqref{eqn:para general} from Proposition \ref{prop 4.2 new}.

\begin{corr}\label{cor 4.2 new}
Let $K\Subset B_\bb(0,1)$ be a compact subset and $\varepsilon_0\in (0,1)$ be a given number. Assumptions as in Proposition \ref{prop 4.2 new}, there exists a constant $C=C(n,\bb,\alpha, g, K,\varepsilon_0)>0$ such that 
$$ \| u\|_{\C^{2+\alpha,\frac{2+\alpha}{2}} (K\times [\varepsilon_0, 1]  )  }\le C \xk{ \| u\|_{\C^0(\qq_\bb)} + \| f\|_{\C^{\alpha,\frac\alpha 2}_\bb(\qq_\bb)}  }.   $$
\end{corr}

With the interior Schauder estimates in Proposition \ref{prop 4.2 new}, we can show the existence of $\C^{2+\alpha, \frac{2+\alpha}{2}}_\bb(\qq_\bb)$-solutions to the Dirichlet problem:
\begin{equation}\label{eqn:Diri para general}
\frac{\partial u}{\partial t} = \Delta_g u + f,\text{ in } \qq_\bb,\text{ and } u = \varphi \text{ on }\partial_\pp \qq_\bb,
\end{equation}
for any given $f\in \C^{\alpha,\frac\alpha 2}_\bb(\overline{\qq_\bb})$ and $\varphi \in \C^0(\partial_\pp \qq_\bb)$. We first show the existence of solutions to \eqref{eqn:Diri para general} in case $\varphi\equiv 0 $. 

\begin{lemma}\label{lemma 4.4}
Let $\sigma\in (0,1)$ be a given number and $u\in \cC^{2+\alpha, \frac{2+\alpha}{2}}_\bb( \qq_\bb )$ solves \eqref{eqn:Diri para general} with $\| u\|^{(-\sigma)}_{\C^0(\qq_\bb)  }< \infty$ and $\| f\|^{(2-\sigma)}_{\C^{\alpha, \frac\alpha 2} _\bb(\qq_\bb) }<\infty$. Then there is a constant $C=C(n,\alpha, \bb, g, \sigma)>0$ such that 
$$\| u\|^{(-\sigma)}_{ \C^{ 2+\alpha, \frac{2+\alpha}{2}  }_\bb(\qq_\bb  )   }\le  C \xk{  \| u\|^{(-\sigma)}_{\C^0(\qq_\bb)} + \| f\|^{(2-\sigma)}_{ \C^{ \alpha, \frac\alpha 2   }_\bb(\qq_\bb)  }        }.$$
\end{lemma}
\begin{proof}
The lemma follows from definitions of the norms and the estimates in Proposition \ref{prop 4.2 new}.
\end{proof}

\medskip

\begin{lemma}\label{lemma 4.5}
Suppose $u\in \C^{2,1}_\bb( \qq_\bb )\cap \C^0 ( \overline{\qq_\bb}  )$ satisfies $\frac{\partial u}{\partial t} = \Delta_g u + f$ and $u \equiv 0 $ on $\partial_\pp \qq_\bb$. For any $\sigma\in (0,1)$, there exists a constant $C=C(n,\bb, g, \sigma)>0$ such that
$$ \| u\|^{ (-\sigma) }_{\C^0(\qq_\bb)} =  \sup_{P_x\in \qq_\bb} d_{P_x}^{-\sigma} |u(P_x)|\le C \sup_{ P_x\in \qq_\bb} d_{P_x}^{2-\sigma} |f(P_x)  | = C \| f\|^{(2-\sigma)  }_{ \C^0( \qq_\bb )   },    $$
where $d_{P_x} = d_{\pp,\bb}(P_x,\partial_\pp \qq_\bb)$ is the parabolic distance of $P_x$ to the parabolic boundary $\partial_\pp \qq_\bb$.
\end{lemma}
\begin{proof}
We denote $N: = \| f\|^{(2-\sigma)  }_{ C^0( \qq_\bb )   }<\infty$ and $P_x = (x,t_x)$. Define functions $$w_1(P_x) =\big (1 - d_\bb(x)^2 \big)^\sigma, \text{ and } w_2(P_x) = t_x^{\sigma/2},$$
where $d_\bb(x) = d_\bb(x, 0)$ is the $g_\bb$-distance between $x$ and $0$. Observe that by definition $d_{P_x} = \min\{ 1- d_\bb(x), t_x^{1/2}   \}$. By a straightforward calculation there is a constant $c_0>0$ such that 
$$ ( \frac {\partial}{\partial t} - \Delta_g )w_1\ge c_0 ( 1-d_\bb(x) )^{\sigma - 2}  ,\text{ and } ( \frac {\partial}{\partial t} - \Delta_g   ) w_2\ge c_0 (t_x^{1/2})^{\sigma - 2}.   $$
By maximum principle we  get \begin{equation}\label{eqn:lemma 4.5 1} | u(P_x) |\le N c_0^{-1} \xk{w_1(P_x)+ w_2(P_x)},\quad \forall P_x\in \qq_\bb.   \end{equation}
We decompose $\qq_\bb$ into different regions, $\qq_\bb = \Omega_1\cup\Omega_2$, where 
\begin{align*}
\Omega_1: = & ~ \big\{ P_x\in\qq_\bb~|~ t_x^{1/2} > 1- d_\bb(x)  \big\},\\
\Omega_2: = &~ \big\{ P_x\in \qq_\bb~|~ t_x^{1/2} \le 1- d_\bb(x)   \big\}.
\end{align*}
\eqref{eqn:lemma 4.5 1} implies that on the parabolic boundaries $\partial_\pp \Omega_1$, $\partial_\pp\Omega_2$, $| u(P_x) |\le 2 N c_0^{-1} d_{P_x}^\sigma$. On $\Omega_1$ we have $  ( \frac {\partial}{\partial t} - \Delta_g ) ( 2Nc_0^{-1} w_1 \pm u  )\ge 0 $ and $2Nc_0^{-1} w_1 \pm u \ge 0$ on $\partial_\pp \Omega_1$, then maximum principle implies that $2Nc_0^{-1} w_1 \pm u \ge 0$ in $\Omega_1$, i.e. $|u(P_x)|\le 2 N c_0^{-1} d_{P_x}^\sigma$ in $\Omega_1$. Similarly we also have $2Nc_0^{-1} w_2 \pm u \ge 0$ in $\Omega_2$ and thus $|u(P_x)|\le 2N c_0^{-1} d_{P_x}^\sigma$ in $\Omega_2$.  In conclusion, we get
$$| u(P_x) |\le 2 c_0^{-1} N d_{P_x}^{\sigma},\quad \forall P_x\in\qq_\bb, $$
and the lemma is proved.

\end{proof}

\begin{prop}\label{prop:4.3}
If $\varphi\equiv 0 $, the equation \eqref{eqn:Diri para general} admits a unique solution $u\in \C^{2+\alpha,\frac{2+\alpha}{2}}_\bb( \qq_\bb  )\cap \C^0(\overline{\qq_\bb})$ for any $f\in \C^{\alpha,\frac\alpha 2}_\bb(\overline{\qq_\bb})$.
\end{prop}
\begin{proof}
The uniqueness follows from maximum principle, so it suffices to show the existence. We will use the continuity method. Define a continuous family of linear operators: for $s\in [0,1]$, $L_s := s ( \frac{\partial }{\partial t} - \Delta_g  ) + (1-s) ( \frac{\partial }{\partial t} - \Delta_{g_\bb}  )$. It can been seen that $L_s = \frac{\partial}{\partial t} - \Delta_{g_s}$ for some conical K\"ahler metric $g_s$ which uniformly equivalent to $g_\bb$ and has uniform $\C^{\alpha,\frac\alpha 2}_\bb$-estimate. So the interior Schauder estimates holds also for $L_s$. Fix a $\sigma\in (0,1)$. Define
\begin{equation*}
\mathcal B_1: = \big\{ u\in \C^{2+\alpha,\frac{2+\alpha}{2}}_\bb( \qq_\bb)|~ \| u\|^{(-\sigma)}_{\C^{2+\alpha, \frac{2+\alpha}{2}}_\bb(\qq_\bb)    }< \infty\big  \},
\end{equation*}
\begin{equation*}
\mathcal B_2: = \big\{ f\in \C^{\alpha,\frac\alpha 2}_\bb(\qq_\bb)|~ \| f\|^{(2-\sigma)}_{\C^{\alpha, \frac\alpha 2}_\bb(\qq_\bb)   }<\infty  \big\}.
\end{equation*}
Observe that any $u\in \mathcal B_1$ is continuous in $\overline{\qq_\bb}$ and vanishes on $\partial_\pp \qq_\bb$. $L_s$ defines a continuous family of linear operators from $\mathcal B_1$ to $\mathcal B_2$.  By Lemmas \ref{lemma 4.4} and \ref{lemma 4.5} we have
\begin{equation*}
\| u \|_{\mathcal B_1} \le C\xk{ \| u\|_{\C^0(\qq_\bb)}^{(-\sigma)}  + \| L_s u\|_{\mathcal B_2}   }\le C \| L_s u\|_{\mathcal B_2},\quad \forall s\in [0,1],\text{ and } \forall u\in\mathcal B_1.
\end{equation*}
By Corollary \ref{cor 4.1 new} and Remark \ref{remark 4.1}, $L_0$ is invertible, thus by Theorem 5.2 in \cite{GT}, $L_1$ is also invertible. 

\end{proof}
\begin{corr}\label{cor 4.3}
For any $\varphi\in \C^0(\partial_\pp \qq_\bb)$ and $f\in \C^{\alpha,\frac\alpha 2}_\bb(\overline{\qq_\bb})$, the equation \eqref{eqn:Diri para general} admits a unique solution $u\in \C^{2+\alpha, \frac{2+\alpha}{2}}_\bb(\qq_\bb)\cap \C^0(\overline{\qq_\bb})$.

\end{corr}
\begin{proof}
The proof is identical to that of Corollary \ref{corollary 3.2} by an approximation argument. We may assume $\varphi\in \C^0(\overline{\qq_\bb})$ and choose a sequence of $\varphi_k\in \C^{2+\alpha,\frac{2+\alpha}{2}}_\bb(\overline{\qq_\bb})$ which converges uniformly to $\varphi$ on $\overline{\qq_\bb}$. The equations $\frac{\partial v_k}{\partial t}  =  \Delta_g v_k + f - \Delta_g \varphi_k$, $v_k\equiv 0 $ on $\partial_\pp \qq_\bb$ admits a unique $\C^{2+\alpha, \frac{2+\alpha}{2}}_\bb$-solution by Proposition \ref{prop:4.3}. The interior Schauder estimates in Corollary \ref{cor 4.2 new} imply that $u_k: = v_k + \varphi_k$ converges in $\C^{2+\alpha, \frac{\alpha +2}{2}}_{\bb,loc}$ to some function $u$ in $\C^{2+\alpha, \frac{2+\alpha}{2}}_\bb(\qq_\bb)$ which solves the equation \eqref{eqn:Diri para general}. The $\C^0$-convergence $u_k\to u$ is uniform on $\overline{\qq_\bb}$ by maximum principle so $u = \varphi$ on $\partial_\pp\qq_\bb$, as desired.

\end{proof}
We recall the definition of weak solutions and refer to Section 7.1 in \cite{Ev} for the notations.
\begin{defn}
We say a function $u$ on $\qq_\bb$ is a weak solution to the equation $\frac{\partial u}{\partial t} = \Delta_g u + f$, if 
\begin{enumerate}[label=(\arabic*)]
\item  $u\in L^2\xk{0,1; H^1(B_\bb)   }$ and $\frac{\partial u}{\partial t} \in L^2\xk{ 0,1; H^{-1}(B_\bb)  }$; 
\item For any $v\in H^1_0(B_\bb)$ and $t\in (0,1)$
$$\int_{B_\bb} \frac{\partial u(x,t)}{\partial t} v (x)\omega_g^n = - \int_{B_\bb} \innpro{\nabla u(x,t), \nabla v(x)}_g\; \omega_g^n + \int_{B_\bb} f(x,t) v(x) \omega_g^n.    $$
\end{enumerate}

\end{defn}
On can use the classical {\em Galerkin approximations} to construct weak solution  to the equation $\frac{\partial u}{\partial t} = \Delta_g u + f$ (see Section 7.1.2 in \cite{Ev}). If $f$ has better regularity, so does the weak solution $u$. 
\begin{lemma}\label{lemma 4.6}
If $f\in \C^{\alpha,\frac\alpha 2}_\bb(\qq_\bb)$, then any weak solution to $\frac{\partial u}{\partial t} = \Delta_g u + f$ belongs to $\C^{2+\alpha, \frac{\alpha+2}{2}}_\bb(\qq_\bb)$.
\end{lemma}
\begin{proof}
Sobolev inequality holds for the metric $g$ so by the proof of the standard De Giorgi-Nash-Moser theory for parabolic equations implies that $u$ is in fact continuous on $\qq_\bb$. 
Since the metric $g$ is smooth on $\qq_\bb^\#$, the weak solution $u$ is also a weak solution in $\qq_\bb^\#$ with the smooth background metric, so $u\in \C^{2+\alpha, \frac{2+\alpha}{2}}_{loc}(\qq_\bb^\#)$ in the usual sense by the classical Schauder estimates. 
Thus it suffices to consider points at $\sS$. We choose the worst such points $0\in\sS$ only, since the case when centers are in other components of $\sS$ is even simpler. We fix the point $P_0=(0, t_0)\in\qq_\bb$ with $t_0>0$. Fix an $r\in (0,\sqrt{t_0})$. By Corollary \ref{cor 4.3} the equation 
$$\frac{\partial v}{\partial t} = \Delta v + f,\text{ in }\qq_\bb(P_0,r): = B_\bb(0,r)\times (t_0 - r^2, t_0],$$
with boundary value $v = u$ on $\partial_\pp \qq_\bb(P_0, r)$ admits a unique solution $v\in \C^{2+\alpha,\frac{\alpha+2}{2}}_\bb(\qq_\bb(P_0, r))$. Then by maximum principle $u = v$ in $\qq_\bb(P_0, r)$. Thus $u\in \C^{2+\alpha,\frac{\alpha +2}{2}}_\bb(\qq_\bb(P_0, r))$ too. The argument also works at other space-time points in $\sS_\pp$, we see that $u\in \C^{2+\alpha,\frac{2+\alpha}{2}}_\bb( \qq_\bb )$, as desired.
\end{proof}

\begin{corr}\label{cor 4.4}
Let $(X,g,D)$ be as in Corollary \ref{corr 3.1}, $u_0\in C^0(X)$ and  $f\in \C^{\alpha,\frac\alpha 2}_\bb(X\times(0,1])$ be  given functions. The weak solution $u$ to the equation  $$\frac{\partial u}{\partial t} = \Delta_g u + f,\text{ in } X\times(0,1], \, u|_{t= 0} = u_0$$ always exists. Moreover, $u\in \C^{2+\alpha, \frac{2+\alpha}{2}}_{\bb}(X\times(0,1])$ and there exists a constant $C=C(n, g, \bb,\alpha)>0$ such that 
%
$$\| u\|_{\C^{2+\alpha,\frac{2+\alpha}{2}}_\bb (X\times (1/2,1])   }\le C \xk{ \| u_0\|_{C^0( X  )} + \| f\|_{ \C^{\alpha,\frac\alpha 2}_\bb (X\times (0,1])  }   }.   $$
\end{corr}
\begin{proof}
The weak equation can be constructed using the {\em Galerkin approximations} (\cite{Ev}). The uniqueness is an easy consequence of maximum principle. The regularity of $u$ follows from the local results in  Lemma \ref{lemma 4.6}. 
The estimate follows from maximum principle, a covering argument as in Corollary \ref{corr 3.1} and the local estimates in  Corollary \ref{cor 4.2 new}.
\end{proof}
The {\em interior} estimate in Corollary \ref{cor 4.4} is not good enough to show the existence of solutions to non-linear partial differential equations since the estimate becomes worse as $t$ approaches $t= 0$. We need some global estimates in the whole time interval $t\in [0,1]$ if the initial $u_0$ has better regularity. 


\subsection{Schauder estimate near $t= 0$}

In this subsection, we will prove a Schauder estimate in the whole time interval for the solutions to the heat equation when the initial value is $0$ or has better regularity. We consider the model case with the background metric $g_\bb$ first, then we generalize the estimate to general non-flat conical K\"ahler metrics.

\subsubsection{The model case}
In this subsection, we will assume the background metric is $g_\bb$. Let $u$ be the solution to the equation
\begin{equation}\label{eqn:para t0}
\frac{\partial u}{\partial t} = \Delta_{g_\bb} u + f,\text{ in }\qq_\bb,\quad u|_{t= 0}\equiv 0,
\end{equation}
and $u = \varphi\in \C^0$ on $\partial B_\bb\times (0,1]$, where $f\in \C^{\alpha,\alpha/2}_\bb(\overline{\qq_\bb})$. 
%
In the calculations below, we should have used the smooth approximating solutions $u_\ep$, where $\partial_t u_\ep = \Delta_{g_\ep} u_\ep + f$ and $u_\ep = u$ on $\partial_\pp \qq_\bb$. But by letting $\ep\to 0$, the corresponding estimates also hold for $u$. So for simplicity, below we will work directly on $u$. 

We fix $0<\rho< R\le 1$ and denote $B_R : = B_\bb(0,R)$ and $Q_R: = B_R\times [0,R^2]$ in this section. Let $u$ be the solution to \eqref{eqn:para t0}. We first have the following Caccioppoli inequalities.
\begin{lemma}
There exists a constant $C=C(n)>0$ such that 
\begin{equation}\label{eqn:lemma 4.7 1}
\sup_{t\in [0,\rho^2]}\int_{B_\rho} u^2 \omega_\bb^n + \iint_{\qq_\rho} \abs{\nabla u}_{g_\bb}\;\omega_\bb^n dt
\le C\bk{ \frac{1}{(R-\rho)^2} \iint_{Q_R} u^2 \omega_\bb^n dt    + (R-\rho)^2 \iint_{Q_R} f^2 \omega_\bb^n dt   },
\end{equation}
and
\eqsp{\label{eqn:lemma 4.7 2}
& \sup_{t\in[0,\rho^2]}\int_{B_\rho} \abs{\nabla u}_{g_\bb}\omega_\bb^n + \iint_{Q_\rho} ( \abs{\nabla \nabla u}_{g_\bb} + \abs{\nabla \bar \na u}_{g_\bb}    )\omega_\bb^n dt\\
\le & ~ C \bk{ \frac{1}{(R-\rho)^2} \iint_{Q_R} \abs{\na u}_{g_\bb} \omega_\bb^n dt + \iint_{Q_R} (f-f_R)^2 \omega_\bb^n dt     }
}
where $f_R: = \frac{1}{|Q_R|_{\omega_\bb}} \iint_{Q_R} f \omega_\bb^n dt$ is the average of $f$ over the cylinder $Q_R$.
\end{lemma}
\begin{proof}
We fix a cut-off function $\eta$ such that $\mathrm{supp}\; \eta\subset B_R$ and $\eta = 1$ on $B_\rho$, $|\nabla\eta|_{g_\bb}\le \frac{2}{R-\rho}$. Multiplying both sides of the equation \eqref{eqn:para t0} by $\eta^2 u$, and applying integration by parts we get
\begin{align*}
\frac{d}{dt}\int_{B_R} \eta^2 u^2& = \int_{B_R} 2 \eta^2 u \Delta_{g_\bb} u + 2 \eta^2 u f\\
 & = \int_{B_R} - 2 \eta^2 \abs{\nabla u}_{g_\bb} - 4 u\eta \innpro{\nabla u,\nabla \eta}_{g_\bb} + 2 \eta^2 u f\\
 & \le \int_{B_R} -\eta^2 \abs{\nabla u}_{g_\bb} + 4 u^2 \abs{\nabla \eta}_{g_\bb} + \eta^2\frac{u}{(R-\rho)^2} + \eta^2 (R-\rho)^2 f^2,
\end{align*}
\eqref{eqn:lemma 4.7 1} follows from this inequality by integrating over $t\in [0,s^2]$ for all $s\le \rho$.  To see \eqref{eqn:lemma 4.7 2}, observe that the Bochner formula yields that
\begin{equation*}
\frac{\partial }{\partial t} \abs{\nabla u}\le \Delta_{g_\bb} \abs{\nabla u} - \abs{\nabla \nabla u}_{g_\bb} - \abs{\nabla \bar \na u}_{g_\bb} - 2 \innpro{\nabla u,\nabla f  }_{g_\bb}.
\end{equation*}
Multiplying both sides of this inequality by $\eta^2 $ and applying IBP, we get
\begin{align*}
\frac{d}{dt}\int_{B_R} \eta^2 \abs{\nabla u} & \le \int_{B_R} -2\eta \innpro{\nabla \eta, \nabla \abs{\na u}  }_{g_\bb} -\eta^2 \abs{\na\na u} - \eta^2\abs{\na\bar \na u} - 2 \eta^2 \innpro{\nabla u,\na f}\\
&\le \int_{B_R} 4\eta |\nabla u| |\na\eta| \big| \na |\na u| \big|  -\eta^2 \abs{\na\na u} - \eta^2\abs{\na\bar \na u}  \\
& ~~~\qquad + 4 \eta |f-f_R| |\na \eta| |\na u| + 2 \eta^2 |f-f_R| |\Delta_{g_\bb} u|\\
&\le \int_{B_R} -\frac {\eta^2} 2( \abs{\na\na u} + \abs{\na\bar \na u}   ) + 10 \eta^2 \abs{\nabla u} \abs{\na \eta} + 20 \eta^2 (f-f_R)^2,
\end{align*}
then \eqref{eqn:lemma 4.7 2} follows from this inequality by integrating over $t\in [0,s^2]$ for any $s\in [0,\rho]$. Thus we finish the proof of the lemma.
\end{proof}
Combining \eqref{eqn:lemma 4.7 1} and \eqref{eqn:lemma 4.7 2} we conclude that 
\eqsp{\label{eqn:lemma 4.7 3}
& \sup_{t\in [0,R^2/4]} \int_{B_{R/2}} \abs{\nabla u} + \iint_{Q_{R/2}} |\Delta_{g_\bb} u|^2\\
\le & ~ \frac{C}{R^4}\iint_{Q_R} u^2 + C R^{2n+2} \| f\|^2_{\C^0(Q_R)} + C R^{2n+2 + 2\alpha} \xk{ [f  ]_{\C^{\alpha, \alpha/2}_\bb(Q_R)   }     }^2.
}
By a standard Moser iteration argument we get the following sub-mean value inequality.
\begin{lemma}\label{lemma Moser iteration}
If in addition $f\equiv 0$, then there exists a constant $C=C(n,\bb)>0$ such that 
$$\sup_{Q_\rho} | u | \le C \bk{\frac{1}{(R-\rho)^{2n+2}} \iint_{Q_R} u^2 \omega_\bb^n dt   }^{1/2}.$$
\end{lemma}
\begin{proof}
For any $p\ge 1$, multiplying both sides of the equation by $\eta^2 u_+^p$ where $u_+ = \max\{u, 0\}$ and applying IBP, we get
$$ \frac{d}{dt}\int_{B_R} \frac{\eta^2}{p+1} u_+^{p+1}  = \int_{B_R}  - p \eta^2 u_+^{p-1} \abs{\nabla u_+} - 2\eta u_+^p \innpro{\na u_+, \na \eta}.   $$
By Cauchy-Schwarz inequality and integrating over $t\in [0,R^2]$, we conclude that
\begin{align*}
\sup_{s\in [0,R^2]} \int_{B_R} \eta^2 u_+^{p+1}\Big|_{t=s} + \iint_{Q_R} \abs{ \na (\eta u_+^{\frac{p+1}{2}}  )  }\le \frac{C}{(R-\rho)^2}\iint_{Q_R} u^{p+1}_+\omega_\bb^n dt=:A.
\end{align*}
By Sobolev inequality we get
\begin{align*}
\int_0^{R^2} \int_{B_R} \xk{ \eta^2 u^{p+1}_+   }^{1+\frac 1 n} \le & \int_0^{R^2} \bk{ \int_{B_R} \eta^2 u^{p+1}_+   }^{1/n} \bk{\int_{B_R} \xk{\eta u_+^{\frac{p+1}{2}}}^{\frac{2n}{n-1}}  }^{\frac{n-1}{n}}\\
\le & A^{1/n} C \int_0^{R^2} \int_{B_R} \abs{\nabla (\eta u_+^{\frac{p+1}{2}})  }\\
\le & C A^{(n+1)/n}.
\end{align*}
If we denote $H(p, \rho) = \bk{ \int_0^{\rho^2} \int_{B_\rho} u_+^p   }^{1/p}$, the inequality above implies that
$$H((p+1)\xi, \rho  )\le \frac{C^{1/(p+1)}}{(R-\rho)^{2/(p+1)}} H(p+1, R),$$
where $\xi = \frac{n+1}{n}>1$. Denote $p_k+1 = 2 \xi^k$ and $\rho_k = \rho + (R-\rho) 2^{-k}$, then $H(p_{k+1} + 1, \rho_{k+1})\le H(p_k+1, \rho_k)$. Iterating this inequality we get
$$H(\infty, \rho) = \sup_{Q_\rho} u_+ \le \frac{C}{(R-\rho)^{n+1}} \bk{ \iint_{Q_R} u_+^2  }^{1/2}.$$ 
Similarly we get the same inequality for $u_- = \max\{-u, 0\}$.

\end{proof}
\begin{corr}
If in addition $f\equiv 0$, then there is a $C=C(n,\bb)>0$ such that
\begin{equation}\label{eqn:liu liu} \iint_{Q_\rho} u^2 \omega_\bb^n dt \le C \bk{ \frac{\rho}{R} }^{2+2n} \iint_{Q_R} u^2 \omega_\bb^n dt.  \end{equation}
\end{corr}
\begin{proof}
When $\rho\in [\frac R 2, R]$, the inequality is trivial; when $\rho\in [0,R/2)$, it follows from Lemma \ref{lemma Moser iteration}.

\end{proof}

\begin{lemma}\label{lemma 4.9}
If in addition $f\equiv 0$, then there is a $C=C(n,\bb)>0$ such that for any $\rho\in (0,R)$
$$ \iint_{Q_\rho} u^2 \omega_\bb^n dt \le C \bk{\frac{\rho}{R}   }^{2n+4}\iint_{Q_R} u^2 \omega_\bb^n dt.   $$

\end{lemma}
\begin{proof}
The inequality is trivial in case $\rho\in [R/2,R]$ so we assume $\rho< R/2$. First we observe that $\Delta_\bb u$ also satisfies the equation $\partial_t (\Delta_\bb u) = \Delta_\bb ( \Delta_\bb u )$ and $(\Delta_\bb u) |_{t= 0} \equiv 0$, so \eqref{eqn:liu liu} holds with $u^2$ replaced by $(\Delta_\bb u)^2$, i.e.
\begin{equation*}
 \iint_{Q_\rho}( \Delta_\bb u)^2 \omega_\bb^n dt \le C \bk{ \frac{\rho}{R} }^{2+2n} \iint_{Q_R} (\Delta_\bb u)^2 \omega_\bb^n dt. 
\end{equation*}
Since $u|_{t=0} = 0$, $u(x,t) = \int_{0}^t \partial_su(x, s) ds$,  we calculate
\begin{align*}
\iint_{Q_\rho} u^2 & \le \rho^4 \iint_{Q_\rho} \ba{ \frac{\partial u}{\partial t}  }^2  = \rho^4 \iint_{Q_\rho} ( \Delta_\bb u  )^2\\
& \le C \rho^4 \bk{ \frac{\rho}{R}  }^{2n+2} \iint_{Q_{R/2}} (\Delta_\bb u)^2\\
\text{by \eqref{eqn:lemma 4.7 3}  }& \le C \bk{\frac{\rho}{R}}^{2n+6} \iint_{Q_R} u^2 \omega_\bb^n dt.
\end{align*}

\end{proof}

\begin{lemma}\label{lemma 4.10}
Let $u$ be a solution to \eqref{eqn:para t0}. There exists a constant $C=C(n,\bb,\alpha)>0$ such that
\begin{equation*}
\frac{1}{\rho^{ 2n + 2 + 2\alpha  }} \iint_{Q_\rho} (\Delta_\bb u)^2 \le \frac{C}{R^{2n+ 2 + 2\alpha}} \iint_{Q_R} (\Delta_\bb u)^2\omega_\bb^n dt  + C \xk{  [f  ]_{ \C^{\alpha, \alpha/2}_\bb( Q_R )  }   }^2.
\end{equation*}

\end{lemma}
\begin{proof}
Let $u = u_1 + u_2$, where 
\begin{equation*}
\frac{\partial u_1}{\partial t} = \Delta_\bb u_1 + f_R,\text{ in }Q_R,\quad u_1 = u\text{ on }\partial_\pp Q_R,
\end{equation*}
and 
$$ \frac{\partial u_2}{\partial t} = \Delta_\bb u_2 + f - f_R,\text{ in }Q_R,\quad u_2 = 0 \text{ on }\partial_\pp Q_R.   $$
The function $(\Delta_\bb u_1)$ satisfies the assumptions of Lemma \ref{lemma 4.9}. Thus 
$$ \iint_{Q_\rho} (\Delta_\bb u_1)^2 \omega_\bb^n dt \le C \bk{\frac{\rho}{R}   }^{2n+4}\iint_{Q_R} (\Delta_\bb u_1)^2 \omega_\bb^n dt.   $$
Multiplying $\dot u_2 = \frac{\partial u_2}{\partial t}$ on both sides of the equation for $u_2$ and noting that $\dot u_2 = 0$ on $\partial B_R \times (0,R^2)$, we get
\begin{align*} \int_{B_R} (\dot u_2)^2  = & \int_{B_R} \dot u_2 \Delta_\bb u_2 +  \dot u_2 (f-f_R)  = \int_{B_R} - 2 \innpro{\nabla \dot u_2, \na u_2} + \dot u (f-f_R) \\
\le & \int_{B_R} -\frac{\partial}{\partial t} \abs{\na u_2} + \frac 12(\dot u_2)^2 + 2 (f-f_R  )^2.
\end{align*}
Integrating over $t\in [0,R^2]$, we obtain
\begin{equation*}
\iint_{Q_R} (\dot u_2)^2\le  - 2 \int_{B_R} \abs{\nabla u_2}\Big|_{t= R^2} + 4 \iint_{Q_R} (f-f_R)^2,
\end{equation*}
therefore
$$\iint_{Q_R} (\Delta_\bb u_2)^2\le 2 \iint_{Q_R} (\dot u_2)^2 + 2 \iint_{Q_R} (f-f_R)^2 \le C R^{2n+2 + 2\alpha} \big( [ f ]_{\C^{\alpha, \alpha/2}_\bb(Q_R)}  \big )^2.  $$
Then for $\rho < R$ we have
\begin{align*}
\iint_{Q_\rho} ( \Delta_\bb u )^2 & \le 2 \iint_{Q_\rho} ( \Delta_\bb u_1 )^2 + 2 \iint_{Q_\rho} ( \Delta_\bb u_2 )^2 \\
& \le  C \bk{\frac{\rho}{R}   }^{2n+4}\iint_{Q_R} (\Delta_\bb u_1)^2 \omega_\bb^n dt + C R^{2n+2 + 2\alpha} \big( [ f ]_{\C^{\alpha, \alpha/2}_\bb(Q_R)}  \big )^2.  
\end{align*}
The estimate is proved by an iteration lemma (see Lemma 3.4 in \cite{HL}).

\end{proof}

\begin{lemma}
Suppose $u$ satisfies the equations \eqref{eqn:para t0}. There exists a constant $C=C(n,\bb, \alpha)>0$ such that for any $\rho \in (0,R/2)$ 
$$\iint_{Q_\rho} \xk{ \Delta_\bb u - (\Delta_\bb u)_\rho    }^2\omega_\bb^n dt \le C M_R \rho^{2n + 2 + 2\alpha},  $$
where $$ M_R : = \frac{1}{R^{4+2\alpha}} \| u\|_{\C^0(Q_R)}^2 + \frac{1}{R^{2\alpha}} \| f\|^2_{\C^0( Q_R )} + \xk{[ f ]_{\C^{\alpha,\alpha/2}_\bb(Q_R)}}^2. $$
\end{lemma}
\begin{proof}
From Lemma \ref{lemma 4.10}, we get
\begin{align*}
\iint_{Q_\rho} (\Delta_\bb u )^2 & \le C \rho^{2+2n+2\alpha} \bk{ \frac{1}{R^{2n+2 + 2\alpha}} \iint_{Q_{2R/3}} (\Delta_\bb u )^2 + \xk{[f]_{ \C^{\alpha,\alpha/2} _\bb(Q_{2R/3}) }   }^2    }\\
\text{ by \eqref{eqn:lemma 4.7 3} }& \le C \rho^{2+2n+2\alpha} \bk{ \frac{1}{R^{2n+6 + 2\alpha}} \iint_{Q_R} u^2 + \frac{1}{R^{2\alpha}} \| f\|^2_{ C^0(Q_R)  } + \xk{[f]_{ \C^{\alpha,\alpha/2} _\bb(Q_R) }   }^2    }\\
& \le C \rho^{2+2n+2\alpha} M_R.
\end{align*}
On the other hand by H\"older inequality
\begin{align*}
(\Delta_\bb u)_\rho^2 & = \frac{1}{|Q_\rho|_{g_\bb}^2} \bk{\iint_{Q_\rho} ( \Delta_\bb u) \omega_\bb^n dt  }^2 \le \frac{C}{\rho^{2+2n}} \iint_{Q_\rho} (\Delta_\bb u)^2 \le CM_R \rho^{2\alpha}. 
\end{align*}
The lemma is proved by combining the two inequalities above.

\end{proof}

By Campanato's lemma (see Theorem 3.1 in \cite{HL}), we get

\begin{corr}There is a constant $C=C(n,\bb, \alpha)>0$
such that for any $x\in B_\bb(0, 3/4)$ and $R<1/10$
\eqsp{\label{eqn:Lap Holder 1}
& [ \Delta_\bb u  ]_{ \C^{\alpha, \alpha/2}_\bb \xk{ B_\bb(x, R/2)\times [0, R^2/4]   } } \\
 \le & C\bk{\frac{1}{R^{2+\alpha}}  \| u\|_{\C^0\xk{ B_\bb(x, R)\times [0, R^2]   }} + \frac{1}{R^\alpha} \| f\|_{\C^0\xk{ B_\bb(x, R)\times [0, R^2]    }} +  {[ f ]_{\C^{\alpha,\alpha/2}_{\bb}\xk{B_\bb(x, R)\times [0, R^2]    }}}}.
}
\end{corr}

\begin{lemma}
There exists a constant $C = C(n,\bb,\alpha)>0$ such that for any $x\in B_\bb(0,3/4)$ and $R< 1/10$
{\small
\eqsp{\label{eqn:Lap Holder 2}
& [T u  ]_{ \C^{\alpha, \frac{\alpha}2}_\bb \xk{ B_\bb(x, R/2)\times [0, R^2/4]   } } +\big [\frac{\partial u}{\partial t}  \big]_{ \C^{\alpha, \frac{\alpha}2}_\bb \xk{ B_\bb(x, R/2)\times [0, R^2/4]   } } \\
 \le & C\bk{\frac{1}{R^{2+\alpha}}  \| u\|_{\C^0\xk{ B_\bb(x, R)\times [0, R^2]   }} + \frac{1}{R^\alpha} \| f\|_{\C^0\xk{ B_\bb(x, R)\times [0, R^2]    }} +  {[ f ]_{\C^{\alpha,\alpha/2}_{\bb}\xk{B_\bb(x, R)\times [0, R^2]    }}}}.
}
}

\end{lemma}
\begin{proof}

It follows from \eqref{eqn:Lap Holder 1} and the elliptic Schauder estimates in Theorem \ref{thm:main 1}  by adjusting $R$ slightly that  for any $t\in [0,R^2/4]$
\begin{align*}
& [ T u(\cdot, t)  ]_{ C^{0, \alpha}_\bb \xk{ B_\bb(x, R/2)   } } \\
 \le & C\bk{\frac{1}{R^{2+\alpha}}  \| u\|_{\C^0\xk{ B_\bb(x, R)\times [0, R^2]   }} + \frac{1}{R^\alpha} \| f\|_{\C^0\xk{ B_\bb(x, R)\times [0, R^2]    }} +  {[ f ]_{\C^{\alpha,\alpha/2}_{\bb}\xk{B_\bb(x, R)\times [0, R^2]    }}}},
\end{align*}
that is,  in the spatial variables the estimate \eqref{eqn:Lap Holder 2} holds. 
 It only remains to show the H\"older continuity of $Tu$ in the time-variable. For this, we fix any two times $0\le t_1 < t_2 \le R^2/4$ and denote $ r : = \sqrt{t_2 - t_1}/2$. For any $x_0\in B_\bb(x,R/4)$, $B_\bb(x_0,r)\subset B_\bb(x,R/2)$. By \eqref{eqn:Lap Holder 1} and the equation for $u$, it is not hard to see that the inequality \eqref{eqn:Lap Holder 1} holds when the $\Delta_\bb u$ on LHS is replaced by $\dot u = \frac{\partial u}{\partial t}$. In particular
$$\frac { | \dot u(y,t) - \dot u (y, t_1)     } {|t - t_1|^{\alpha/2}} \le A_R, \forall y\in B_\bb(x,R/2) $$
where $A_R:=$ the constant on the RHS of \eqref{eqn:Lap Holder 1}. Integrating over $t\in [t_1,t_2]$ we get
$$ \ba{ u(y,t_2) - u(y,t_1) - \dot u(y, t_1) (t_2 - t_1)     }\le C A_R (t_2 - t_1)^{1+\frac\alpha 2},  $$
thus for any $y\in B_\bb(x_0, r)$
\begin{align*}
&  \ba{ u(y,t_2) - u(y,t_1) - \dot u(x_0, t_1) (t_2 - t_1)     }\\
\le &   \ba{ u(y,t_2) - u(y,t_1) - \dot u(y, t_1) (t_2 - t_1)     } + \ba{ \dot u(x_0, t_1) - \dot u(y, t_1)   } (t_2 - t_1)\\
\le  & C A_R (t_2 - t_1)^{1+\frac\alpha 2} + A_R r^{\alpha}(t_2 - t_1).
\end{align*}
Denote $\tilde u(y): = u(y,t_2 ) - u(y,t_1) - \dot u(x_0, t_1) (t_2 - t_1)$, which is a function on $B_\bb(x_0, r)$ and $\tilde f: = \Delta_\bb \tilde u = \Delta_\bb u(\cdot, t_2) - \Delta_\bb u(\cdot, t_1)$ satisfies $\| \tilde  f\|_{C^0(B_\bb(x_0,r))  } \le A_R (t_2 - t_1)^\alpha   $ and $[\tilde f  ]_{C^{0,\alpha}_\bb(B_\bb(x_0,r))}\le A_R$ by \eqref{eqn:Lap Holder 1}. It follows from the rescaled version of Proposition \ref{prop:invariant} that 
\begin{align*} \ba{T\tilde u  }_{C^0(B_\bb(x_0, r/2)  )}& \le C(n,\bb,\alpha) \xk{ \frac{ \| \tilde u\|_{C^0(B_\bb(x_0,r))}  }{ r^2  } + \| \tilde f\|_{C^{0}( B_\bb(x_0,r)  )} + r^\alpha [ \tilde f  ] _{C^{0,\alpha}(B_\bb(x_0, r))  }        }  \\
& \le C (t_2 - t_1)^{\alpha/2} A_R.
\end{align*} 
Therefore for any $x_0\in B_\bb(x, R/4)$
\begin{equation*}
\frac{ \ba{ Tu(x_0, t_2)  - Tu(x_0, t_1)  }   }{|t_2 - t_1|^{\alpha/2}} \le C A_R.
\end{equation*}
It is then elementary to see by triangle inequality that (by adjusting $R$ slightly if necessary)
\begin{equation*}
[Tu  ]_{ \C^{\alpha, \alpha/2}_{\bb} \xk{ B_\bb(x,R/2) \times [0,R^2/4]   }  }\le C A_R,
\end{equation*}
as desired. The estimate for $\dot u$ follows from the equation $\dot u = \Delta_g u + f$.

\end{proof}


\begin{remark}
By a simple parabolic rescaling of the metric and time, we see from \eqref{eqn:Lap Holder 2} that for any $0<r<R<1/10$ that 
\begin{equation}\label{eqn:Lap Holder new}
[Tu]_{ \C^{\alpha, \alpha /2}_\bb( Q_{r})  } \le C \bk{ \frac{\| u\|_{\C^0(  Q_R  )}}{(R- r)^{2+\alpha}}  + \frac{\| f\|_{ \C^0( Q_R) }}{(R-r)^\alpha}     + [ f  ]_{ \C^{\alpha,\alpha/2}_\bb ( Q_R)  }}.
\end{equation}
\end{remark}

\subsubsection{the non-flat metric case}
In this subsection, we will consider the case when the background metrics are general non-flat $\C^{\alpha,\alpha/2}_\bb$-conical K\"ahler metrics $g=g(z,t)$. Suppose $u\in\C^{2+\alpha,\frac{2+\alpha}{2}}_\bb(\qq_\bb)$ satisfies the equation 
\begin{equation}\label{eqn:para gen t0}
\frac{\partial u}{\partial t} = \Delta_g u + f,\text{ in }\qq_\bb,\quad u|_{t= 0} = 0,
\end{equation} 
and $u\in \C^0(\partial_\pp \qq_\bb)$.


\begin{prop}\label{prop:4.4}
There exists a constant $C=C(n,\bb,\alpha, g)>0$ such that
\begin{equation*}
\| u\|_{\C^{2+\alpha, \frac{2+\alpha}{2}}_\bb\xk{ B_\bb(0,1/2)\times [0,1/4]  } }\le C\bk{ \| u\|_{\C^0(\qq_\bb)}  + \| f\|_{ \C^{\alpha,\alpha/2}_\bb(\qq_\bb)  }    }.
\end{equation*}
\end{prop}
\begin{proof}
Choose suitable complex coordinates at the origin $x= 0$, we may assume the components of $g$ in the basis $\{\ep_j\wedge \bar \ep_k,\cdots\}$ satisfies $g_{\ep_j\bar \ep_k}(\cdot, 0) = \delta_{jk}$ and $g_{j\bar k} (\cdot, 0)= \delta_{jk}$ at the origin $0$. As in the proof of Proposition \ref{prop 4.2 new}, we can write the equation \eqref{eqn:para gen t0} as
\begin{equation*}
\frac{\partial u}{\partial t} = \Delta_\bb u + \eta.\ddb u + f =: \Delta_\bb u + \hat f,
\end{equation*}
where $\eta$ is given in the proof of Proposition \ref{prop:invariant}. By \eqref{eqn:Lap Holder new} we get
\begin{align*}
[Tu  ]_{ \C^{\alpha,\alpha/2}_\bb(\tilde Q_r  )  } \le C\bk{ \frac{\| u\|_{\C^0(\tilde Q_R)}}{(R-r)^{2+\alpha}} +\frac{1}{(R-r)^\alpha} \| \hat f\|_{\C^0(\tilde Q_R  )} + [ \hat f ]_{\C^{ \alpha,\alpha/2  } _\bb(\tilde Q_R)  }      },
\end{align*}
where $\tilde Q_R := B_\bb(0, R)\times[0,R^2]$. Observe that 
{\small
\begin{align*}
\frac{1}{(R-r)^\alpha} \| \hat f\|_{\C^0(\tilde Q_R  )} & \le \frac{1}{(R-r)^\alpha} \| f\|_{\C^0(\tilde Q_R  )} + \frac{1}{(R-r)^\alpha} \| \eta\|_{\C^0(\tilde Q_R)}  \| Tu\|_{C^0(\tilde Q_R)}\\
& \le \frac{1}{(R-r)^\alpha} \| f\|_{\C^0(\tilde Q_R  )} + \frac{[\eta]_{\C_\bb^{\alpha,\alpha/2} (\tilde Q_R)  } R^\alpha }{(R-r)^\alpha} \xk{ \varepsilon [Tu]_{\C^{\alpha,\alpha/2}_\bb(\tilde Q_R)  } + C(\varepsilon) \| u\|_{\C^0(\tilde Q_R)}  }
 \end{align*}
 }
 and 
\begin{align*}
 [ \hat f ]_{\C^{ \alpha,\alpha/2  } _\bb(\tilde Q_R)  } & \le  [f ]_{\C^{ \alpha,\alpha/2  } _\bb(\tilde Q_R)  } + \| \eta\|_{\C^{0}(\tilde Q_R)  } [Tu  ]_{\C^{\alpha,\alpha/2}_\bb(\tilde Q_R)  } + \| Tu\|_{\C^0(\tilde Q_R)} [ \eta  ]_{\C^{\alpha,\alpha/2}_\bb(\tilde Q_R)}\\
 & \le  [f ]_{\C^{ \alpha,\alpha/2  } _\bb(\tilde Q_R)  } +  [\eta]_{\C^{\alpha,\alpha/2}_\bb(\tilde Q_R)  } R^\alpha [Tu  ]_{ \C^{\alpha,\alpha/2}_\bb(\tilde Q_R)  } \\
 & \qquad + [ \eta  ]_{\C^{\alpha,\alpha/2}_\bb(\tilde Q_R)} \xk{ \varepsilon [ Tu]_{\C^{\alpha,\alpha/2}_\bb (\tilde Q_R)}  + C(\varepsilon) \| u\|_{\C^0(\tilde Q_R)}  }.
\end{align*}
By choosing $R_0 = R_0(n,\bb, \alpha, g)>0$ small enough and suitable $\varepsilon>0$, for any $0<r<R<R_0<1/10$, the combination of the above inequalities yields that
\begin{align*}
[Tu  ]_{ \C^{\alpha,\alpha/2}_\bb(\tilde Q_r  )  } \le \frac 1 2 [Tu]_{ \C^{\alpha,\alpha/2}_\bb(\tilde Q_R)  } +  C\bk{ \frac{\| u\|_{\C^0(\tilde Q_R)}}{(R-r)^{2+\alpha}} +\frac{1}{(R-r)^\alpha} \| f\|_{\C^0(\tilde Q_R  )} + [f ]_{\C^{ \alpha,\alpha/2  } _\bb(\tilde Q_R)  }      },
\end{align*}
By Lemma \ref{lemma HL} below (setting $\phi(r) = [Tu]_{\C^{\alpha,\alpha/2}_\bb(\tilde Q_r)}$), we conclude that
$$[ Tu ]_{ \C^{\alpha,\alpha/2}_\bb\xk{ B_\bb(0,R_0/2) \times [0,R_0^2/4]  }  } \le C \bk{ \| u\|_{\C^0(\qq_\bb)} + \| f\|_{\C^{\alpha, \alpha/2}_\bb(\qq_\bb)}    }. $$ 
This is the desired estimate when the center of the ball is the worst possible. For the other balls $B_\bb(x,r)$ with center $x\in B_\bb(0,1/2)$, we can repeat the above procedures and use the smooth coordinates $w_j = z_j^{\beta_j}$ in case the ball is disjoint with $\sS_j$. Finitely many such balls cover $B_\bb(0,1/2)$ so we get the  
$$[ Tu ]_{ \C^{\alpha,\alpha/2}_\bb\xk{ B_\bb(0,1/2) \times [0,1/100]  }  } \le C \bk{ \| u\|_{\C^0(\qq_\bb)} + \| f\|_{\C^{\alpha, \alpha/2}_\bb(\qq_\bb)}    }. $$ 
The proposition is proved by combining this inequality, the equation for $u$, interpolation inequalities, and the interior Schauder estimates in Corollary \ref{cor 4.2 new}.
\end{proof}

\begin{lemma} [Lemma 4.3 in \cite{HL}]\label{lemma HL}
Let $\phi(t)\ge 0$ be bounded in $[0,T]$. Suppose for any $0< t< s\le T$ we have
$$\phi(t)\le \frac 1 2 \phi(s) + \frac{A}{(s-t)^a} + B$$
for some $a >0$, $A,B>0$. Then it holds that for any $0<t<s\le T$ 
$$\phi(t)\le c(a) \bk{ \frac{A}{(s-t)^{a}}  + B}.$$

\end{lemma}

\begin{corr}\label{corr:4.7}Suppose $u$ satisfies the equation
$$\frac{\partial u}{\partial t} = \Delta_g u + f,\text{ in }\qq_\bb,\quad u|_{t= 0} = u_0\in C^{2,\alpha}_{\bb} (B_\bb(0,1)),$$
then 
\begin{equation*}
\| u\|_{ \C^{2+\alpha, \frac{\alpha+2}{2}}_\bb( B_\bb(0,1/2)\times [0,1] )  }\le C \xk{ \| u\|_{\C^0(\qq_\bb)} + \| f\|_{\C^{\alpha, \alpha/2}_\bb( \qq_\bb )}  + \| u_0\|_{C^{2,\alpha}_\bb(B_\bb(0,1))}   },
\end{equation*}
for some constant $C=C(n,\bb,\alpha, g)>0$.
\end{corr}
\begin{proof}
We set $\hat u = u - u_0$ and $\hat f = f- \Delta_{g} u_0$. $\hat u$ satisfies the conditions in Proposition \ref{prop:4.4}, so the corollary follows from Proposition \ref{prop:4.4} applied to $\hat u$ and triangle inequalities.
\end{proof}

\begin{corr}\label{cor:existence lemma}
Let the assumptions be as in Corollary \ref{cor 4.4} except that in addition we assume $u_0\in C^{2,\alpha}_\bb(X)$. Then the  weak solution to $\frac{\partial u}{\partial t} = \Delta_g u + f$ with $u|_{t= 0} = u_0$ exists and is in $\C^{2+\alpha, \frac{2+\alpha}{2}}_\bb(X,\times [0,1])$. Moreover  there is a $C=C(n,\bb,\alpha, g)>0$ such that
\begin{equation}\label{eqn:main goal} \| u\|_{ \C^{2+\alpha, \frac{2+\alpha}{2}}_\bb(X\times [0,1]) }\le C\xk{ \| f\|_{\C^{\alpha,\alpha/2}_\bb(X\times [0,1])  } + \| u_0\|_{C^{2,\alpha}_{\bb}(X)}  }.
\end{equation}
\end{corr}
\begin{proof}
Observe that by maximum principle we have 
$$\| u\|_{\C^0(X\times [0,1])}\le \| f\|_{\C^0(X\times [0,1])} + \| u_0\|_{C^0(X)}.  $$
Then the estimate \eqref{eqn:main goal} follows from Corollary \ref{corr:4.7} and a covering argument as in the proof of Corollary \ref{corr 3.1}.
\end{proof}

\section{Conical K\"ahler-Ricci flow}
Let $X$ be a compact K\"ahler manifold and $D=\sum_j D_j$ be a divisor with simple normal crossings. Let $\omega_0$ be a fixed $C^{0,\alpha'}_\bb(X)$ conical K\"ahler metric with cone angle $2\pi \bb$ along $D$ and $\hat \omega_t$ be a family of $\C^{\alpha',\frac{\alpha'}{2}}_\bb$ conical metrics which are uniformly equivalent to $\omega_0$, $\hat \omega_0 = \omega_0$ and  $\| \hat \omega\|_{\C^{\alpha',\alpha'/2}_\bb( X\times [0,1] )}\le C_0$. We consider the complex Monge-Amp\`ere equation:
\begin{equation}\label{eqn:MA}
\left\{\begin{aligned}
&\frac{\partial \varphi}{\partial t} = \log\bk{ \frac{(\hat \omega_t + \ddb \varphi)^n}{\omega_0^n}  } + f\\
&\varphi|_{t= 0 } = 0,
\end{aligned}\right.
\end{equation}
where $f\in \C^{\alpha',\alpha'/2}_\bb(X\times [0,1])$ is a given function. We will use an inverse function theorem argument  in \cite{CL} which was  outlined in \cite{H} to show the short time existence of the flow \eqref{eqn:MA}.

\begin{theorem}
There exists a small $T=T(n,\bb, \omega_0, f,\alpha,\alpha')>0$ such that the equation \eqref{eqn:MA} admits a unique solution $\varphi\in \C^{2+\alpha, \frac{2+\alpha}{2}}_\bb(X\times [0,T])$, for any $\alpha<\alpha'$.
\end{theorem}
\begin{proof} The uniqueness of the solution follows from maximum principle. 
We will break the proof of short-time existence into three steps.

\smallskip

\noindent{\bf Step 1.}
Let $u\in \C^{2+\alpha', \frac{2+\alpha'}{2}}_\bb(X\times [0,1])$ be the solution to the equation
\begin{equation*}
\left\{\begin{aligned}
&\frac{\partial u}{\partial t} = \Delta_{g_0} u + f,\quad \text{in }X\times[0,1]\\
& u|_{t= 0} = 0.
\end{aligned}
\right.
\end{equation*}
Thanks to Corollary \ref{cor:existence lemma} such $u$ exists and satisfies the estimate \eqref{eqn:main goal}. We fix an $\varepsilon>0$ so that as long as $\| \phi\|_{C^{2,\alpha}_\bb(X)}\le \varepsilon$, $\hat \omega_{t,\phi}:=\hat \omega_t+\ddb\phi$ is equivalent to $\omega_0$, i.e. $C_0^{-1} \omega_0 \le \omega_{0,\phi}\le C_0 \omega_0$, and $\|\hat\omega_{t,\phi}\|_{\C_\bb^{\alpha,\alpha/2}}\le C_0$.

\smallskip

\noindent We {\bf claim} that for $T_1>0$ small enough, $\| u\|_{\C^{2+\alpha, (2+\alpha)/2}_\bb( X\times [0,T_1] )}\le \varepsilon$. We firs observe that by \eqref{eqn:main goal} that $$N: = \| u\|_{\C^{2+\alpha',\frac{\alpha'+2}{2}} _\bb( X\times [0,1] )   }\le C \| f\|_{\C^{\alpha',\alpha'/2}_\bb(X\times [0,1])}.$$ It suffices to show that $[u]_{\C^{2+\alpha',\frac{\alpha'+2}{2}} _\bb(X\times [0,T_1])}$ is small since the lower order derivatives are small since $u|_{t= 0 } = 0$. We calculate for any $t_1,t_2\in [0,T_1]$
$$ \frac{|Tu(x,t_1) - Tu(x,t_2)   |}{|t_1 - t_2  |^{\alpha/2}} +  \frac{|\dot u(x,t_1) - \dot u(x,t_2)   |}{|t_1 - t_2  |^{\alpha/2}}\le N |t_1 - t_2  |^{(\alpha' - \alpha)/2} \le \varepsilon/4,$$ if $N T_1^{(\alpha' - \alpha)/2}< \varepsilon/4$.  For any $x,y\in X$ and $t\in [0, T_1]$
$$ \frac{| Tu(x,t)  - Tu(y,t) |}{d_{g_0}(x,y)^\alpha  }\le N \min\Big\{ \frac{ 2 T_1^{\alpha'/2}}{d_{g_0}(x,y)^{\alpha}} , d_{g_0}(x,y)^{\alpha' - \alpha}      \Big\} \le \frac\varepsilon 2.  $$
The {\bf claim} then follows from triangle inequality.

\smallskip

%
We define a function $$w(x,t) :=\frac{\partial u}{\partial t}(x,t)  - \log \bk{ \frac{(\hat \omega_t + \ddb u)^n}{\omega_0^n}   }(x,t) - f(x,t),~\forall~ (x,t)\in X\times [0, T_1].   $$
It is clear that $w(x,0)\equiv 0$. 

\medskip

\noindent{\bf Step 2.}
We consider the small ball $$\mathcal B = \big\{ \phi\in \C^{2+\alpha,\frac{2+\alpha}{2}}_\bb( X\times [0, T_1]  ) ~|~ \| \phi\|_{ \C^{2+\alpha,\frac{\alpha+2}{2}}_\bb  }\le \varepsilon,\,    \phi(\cdot, 0) =0  \big\}$$ in the space $\C^{2+\alpha,\frac{2+\alpha}{2}}_\bb(X\times [0,T_1])$. $u|_{t\in [0,T_1]}\in \mathcal B$ by the discussion in {\bf Step 1}.

  Define the differential map $\Psi:  \mathcal B \to \C^{\alpha,\alpha/2}_\bb( X\times [0,T_1]  )$  by
$$\Psi(\phi) = \frac{\partial \phi}{\partial t} - \log\bk{ \frac{(\hat\omega_t + \ddb \phi)^n}{\omega_0^n}   } - f.   $$
The map $\Psi$ is well-defined and $C^1$ with the differential $D\Psi_\phi$ at any $\phi\in\mathcal B$ is given by 
$$D\Psi_\phi( v  ) = \frac{\partial v}{\partial t} - (\hat g_{\phi})^{i\bar j} v_{i\bar j} = \frac{\partial v}{\partial t} - \Delta_{\hat \omega_{t,\phi}} v,$$
for any $v\in T_\phi \mathcal B = \big\{ v\in \C^{2+\alpha,\frac{2+\alpha}{2}}_\bb(X\times [0,T_1]  )~|~ v(\cdot, 0) = 0       \big\}$, where $(\hat g_\phi)^{i\bar j}$ denotes the inverse of the metric $\hat\omega_t+\ddb \phi$. As a linear map, $D\Psi_\phi: T_\phi \mathcal B \to \C^{\alpha,\alpha/2}_\bb(X\times [0,T_1]  )$ is injective by maximum principle; is surjective by Corollary \ref{cor:existence lemma}. Thus $D\Psi_\phi$ is invertible at any $\phi\in\mathcal B$. In particular $D\Psi_u$ is invertible and by inverse function theorem $\Psi: \mathcal B\to \C^{\alpha,\alpha/2}_\bb(X\times [0,T_1])$ defines a local diffeomorphism from a small neighborhood of $u\in\mathcal B $ to an open neighborhood of $w = \Psi(u)$ in $\C^{\alpha,\alpha/2}_\bb(X\times [0,T_1]  )$. This implies that for any $\tilde w\in \C^{\alpha,\alpha/2}_\bb(X\times [0,T_1]  )$ with $\| w - \tilde w \|_{\C^{\alpha,\alpha/2}_\bb(X\times [0,T_1]  )}< \delta$ for some small $\delta>0$, there exists a unique $\varphi\in \mathcal B$ such that $\Psi(\varphi) = \tilde w$.

\medskip

\noindent{\bf Step 3.} For a small $T_2< T_1$ to be determined, we define a function
$$ \tilde w (x,t) = \left\{\begin{aligned}
& 0,\quad t\in [0,T_2]\\
& w(x, t - T_2),\quad t\in [T_2, T_1].
\end{aligned}\right.  $$ Since $u\in \C^{2+\alpha',\frac{2+\alpha'}{2}}_\bb$, we see that $w\in \C^{\alpha',\alpha'/2}_\bb(X\times [0,T_1]  )$ with $M: = \| w\|_{  \C^{\alpha',\alpha'/2}_\bb(X\times [0,T_1]  )} <\infty$. We {\bf claim} that if $T_2$ is small enough, then $\| w-\tilde w\|_{ \C^{\alpha,\alpha/2}_\bb(X\times [0,T_1]  )} < \delta$. We denote $\eta = w - \tilde w$. It is clear from the fact that $w(\cdot, 0) = 0$ that $\| \eta\|_{\C^0}\le \delta/2$ if $T_2$ is small enough.

\smallskip

\noindent{\bf Spatial directions:}
If $t<T_2$ then
\begin{align*}
\frac{ | \eta(x,t) - \eta(y,t)  |    }{d_{g_0}( x,y  )^\alpha}= \frac{ | w(x,t) - w(y,t)  |    }{d_{g_0}( x,y  )^\alpha} \le M \min\Big\{ \frac{2T_2^{\alpha'/2}}{ d_{g_0}(x,y)^\alpha  }, d_{g_0}(x,y)^{\alpha' - \alpha}   \Big\}\le 2 M T_2^{(\alpha' - \alpha)/2},
\end{align*}
if $t\in [T_2,T_1]$ then
\begin{align*}
\frac{ | \eta(x,t) - \eta(y,t)  |    }{d_{g_0}( x,y  )^\alpha} & = \frac{ | w(x,t) - w(y,t) - w(x,t-T_2) + w(y, t - T_2)  |    }{d_{g_0}( x,y  )^\alpha}\\
&\le  2 M \min\Big\{ \frac{T_2^{\alpha'/2}}{ d_{g_0}(x,y)^\alpha  }, d_{g_0}(x,y)^{\alpha' - \alpha}   \Big\}\le 2 M T_2^{(\alpha' - \alpha)/2}.
\end{align*}

\smallskip

\noindent{\bf Time direction:}
If $t ,t '< T_2$, then
$$\frac{ | \eta(x,t) - \eta(x,t')  |    }{|t-t'|^{\alpha/2}}  = \frac{ | w(x,t) - w(x,t' )  |    }{|t-t'|^{\alpha/2}} \le M |t- t'  |^{(\alpha' - \alpha)/2}\le M T_2^{(\alpha' - \alpha)/2};  $$
If $t,t'\in [T_2,T_1]$, then
\begin{align*}
\frac{ | \eta(x,t) - \eta(x,t')  |    }{|t-t'|^{\alpha/2}}  & = \frac{ | w(x,t) - w(x,t') - w(x,t-T_2) + w(x, t'-T_2)  |    }{|t-t'|^{\alpha/2}}  \le 2M T_2^{(\alpha' - \alpha)/2};
\end{align*}
If $t< T_2 \le t' \le T_1$, then
\begin{align*}
\frac{ | \eta(x,t) - \eta(x,t')  |    }{|t-t'|^{\alpha/2}}   = \frac{ | w(x,t) - w(x,t') + w(x, t' - T_2)  |    }{|t-t'|^{\alpha/2}} \le  2M T_2^{(\alpha' - \alpha)/2}.
\end{align*}

Therefore if we choose $T_2>0$ small so that $2M T_2^{(\alpha' - \alpha)/2}< \delta/4$, then we have 
$$  \frac{ | \eta(x,t) - \eta(x,t')  |    }{|t-t'|^{\alpha/2}} + \frac{ | \eta(x,t) - \eta(y,t)  |    }{d_{g_0}( x,y  )^\alpha} \le \frac\delta 2,\forall x\in X, \, t,t'\in [0,T_1].  $$ It then follows from triangle inequality that
\begin{align*}
 | \eta(x,t) - \eta(y,t')  | & \le | \eta(x,t) - \eta(y,t)  | + | \eta(y,t) - \eta(y,t')|\\
 & \le \frac{\delta}{2}\xk{ d_{g_0}(x,y)^\alpha + | t-t'  |^{\alpha/2}       }  \\
 & \le \delta/2 d_{\pp,g_0}\xk{ (x,t), (y,t')   }^{\alpha}.
\end{align*}

In conclusion, $\| \tilde w - w\|_{ \C^{\alpha,\alpha/2}_\bb(X\times [0,T_1])  }< \delta$ so by {\bf Step 2}, we conclude that  there exists a $\varphi\in\mathcal B$ such that $\Psi(\varphi) = \tilde w$. Since $\tilde w|_{t\in [0,T_2]}\equiv 0$ by definition,  $\varphi|_{t\in [0,T_2]}$ satisfies the equation \eqref{eqn:MA} for $t\in [0,T]$, where $T:=T_2$. This shows the short-time existence of the flow \eqref{eqn:MA}.
 
\end{proof}

\begin{proof}[Proof of Corollary \ref{corr:1.3}]
Recall in \eqref{eqn:omega 0 assumption} we write
$\omega_0^n = \frac{\Omega}{\prod_j (|s_j|_{h_j}^{2})^{1-\beta_j}}$ where $\Omega$ is a smooth volume form, $s_j$ and $h_j$ are holomorphic sections and hermtian metrics of the line bundle associated to the component  $D_j$, respectively.
Choose a smooth reference form $\chi =\ddb \log \Omega - \sum_j (1-\beta_j)\ddb\log h_j.$
Define the reference metrics $\hat \omega_t = \omega_0 + t \chi$ which are $\C^{\alpha',\alpha'/2}_\bb$-conical and K\"ahler for small $t>0$. 
Let $\varphi$ be the $\C^{2+\alpha,\frac{2+\alpha}{2}}_\bb$-solution to the equation \eqref{eqn:MA section1.1} with $f\equiv 0$. Then it is straightforward to check that $\omega_t = \hat\omega_t + \ddb\varphi$ satisfies the conical K\"ahler-Ricci flow equation \eqref{eqn:KRF} and $\omega\in \C^{\alpha,\alpha/2}_\bb(X\times [0,T])$ for some small $T>0$.

The smoothness of $\omega$ in $X\backslash D\times (0, T]$ follows from the general smoothing properties of parabolic equations (see \cite{ST2}). Taking $\frac{\partial }{\partial t}$ on both sides of \eqref{eqn:MA section1.1} we get
$$\frac{\partial \dot\varphi}{\partial t} = \Delta_{\omega_t} \dot \varphi + \tr_{\omega_t} \chi,\text{ ~ and ~} \dot\varphi|_{t= 0} = 0.$$
By Corollary \ref{cor:existence lemma}, $\dot \varphi\in \C^{2+\alpha,\frac{2+\alpha}{2}}_\bb(X\times [0,T]  )$ since $\tr_{\omega_t}\chi \in \C^{\alpha,\alpha/2}_\bb(X\times [0,T])$. Therefore the normalized Ricci potential $\log\xk{\frac{\omega_t^n}{\omega_0^n}}\in \C^{2+\alpha,\frac{2+\alpha}{2}}_\bb( X\times [0,T]  )$.

\end{proof}

\bigskip

\noindent {\bf{Acknowledgements:}} Both authors thank Duong H. Phong and Ved Datar for many insightful discussions.






\end{document}